\documentclass[a4paper,11pt]{amsart}
\usepackage{mathrsfs}
\usepackage{cases}
\usepackage{epic}
\usepackage{amsfonts}
\usepackage{graphicx}
\usepackage{amsmath}
\usepackage{amssymb, upgreek}
\usepackage{bm}
\usepackage{latexsym,todonotes}
\usepackage{pdflscape}
\usepackage[all]{xypic}
\usepackage[all]{xy}
\usepackage{color}
\usepackage{colordvi}
\usepackage{multicol}
\usepackage[linktocpage=true]{hyperref}
\hypersetup{colorlinks,linkcolor=blue,urlcolor=cyan,citecolor=blue}
\newcommand{\bvs}{\mathbf{\varsigma}}
\newcommand{\vs}{\varsigma}
\allowdisplaybreaks

\begin{document}
\input xy
\xyoption{all}

\newtheorem{innercustomthm}{{\bf Main~Theorem}}
\newenvironment{customthm}[1]
  {\renewcommand\theinnercustomthm{#1}\innercustomthm}
  {\endinnercustomthm}

  \newtheorem{innercustomcor}{{\bf Corollary}}
\newenvironment{customcor}[1]
  {\renewcommand\theinnercustomcor{#1}\innercustomcor}
  {\endinnercustomthm}

  \newtheorem{innercustomprop}{{\bf Proposition}}
\newenvironment{customprop}[1]
  {\renewcommand\theinnercustomprop{#1}\innercustomprop}
  {\endinnercustomthm}

\newcommand{\iadd}{\operatorname{iadd}\nolimits}
\newcommand{\Gr}{\operatorname{Gr}\nolimits}
\newcommand{\FGS}{\operatorname{FGS}\nolimits}

\renewcommand{\mod}{\operatorname{mod^{\rm nil}}\nolimits}
\newcommand{\proj}{\operatorname{proj}\nolimits}
\newcommand{\inj}{\operatorname{inj.}\nolimits}
\newcommand{\rad}{\operatorname{rad}\nolimits}
\newcommand{\Span}{\operatorname{Span}\nolimits}
\newcommand{\soc}{\operatorname{soc}\nolimits}
\newcommand{\ind}{\operatorname{inj.dim}\nolimits}
\newcommand{\Ginj}{\operatorname{Ginj}\nolimits}
\newcommand{\res}{\operatorname{res}\nolimits}
\newcommand{\np}{\operatorname{np}\nolimits}
\newcommand{\Fac}{\operatorname{Fac}\nolimits}
\newcommand{\Aut}{\operatorname{Aut}\nolimits}
\newcommand{\DTr}{\operatorname{DTr}\nolimits}
\newcommand{\TrD}{\operatorname{TrD}\nolimits}

\newcommand{\Mod}{\operatorname{Mod}\nolimits}
\newcommand{\R}{\operatorname{R}\nolimits}
\newcommand{\End}{\operatorname{End}\nolimits}
\newcommand{\lf}{\operatorname{l.f.}\nolimits}
\newcommand{\Iso}{\operatorname{Iso}\nolimits}
\newcommand{\aut}{\operatorname{Aut}\nolimits}
\newcommand{\Ui}{{\mathbf U}^\imath}
\newcommand{\UU}{{\mathbf U}\otimes {\mathbf U}}
\newcommand{\UUi}{(\UU)^\imath}
\newcommand{\tUU}{{\tU}\otimes {\tU}}
\newcommand{\tUUi}{(\tUU)^\imath}
\newcommand{\tUi}{\widetilde{{\mathbf U}}^\imath}
\newcommand{\sqq}{{\bf v}}
\newcommand{\sqvs}{\sqrt{\vs}}
\newcommand{\dbl}{\operatorname{dbl}\nolimits}
\newcommand{\swa}{\operatorname{swap}\nolimits}
\newcommand{\Gp}{\operatorname{Gp}\nolimits}

\newcommand{\U}{{\mathbf U}}
\newcommand{\tU}{\widetilde{\mathbf U}}
\newcommand{\fgm}{{\rm mod}^{{\rm fg}}}
\newcommand{\fgmz}{\mathrm{mod}^{{\rm fg},\Z}}
\newcommand{\fdmz}{\mathrm{mod}^{{\rm nil},\Z}}

\newcommand{\ov}{\overline}
\newcommand{\und}{\underline}
\newcommand{\tk}{\widetilde{k}}
\newcommand{\tK}{\widetilde{K}}
\newcommand{\tTT}{\operatorname{\widetilde{\texttt{\rm T}}}\nolimits}

\newcommand{\tH}{\operatorname{{\ch}_{\rm{tw}}}\nolimits}

\newcommand{\utM}{\operatorname{\cm\ch}\nolimits}
\newcommand{\tM}{\operatorname{\cs\cd\widetilde{\ch}}\nolimits}
\newcommand{\rM}{\operatorname{\cm\ch_{\rm{red}}}\nolimits}
\newcommand{\utMH}{\cs\cd\ch(\Lambda^\imath)}
\newcommand{\tMH}{\cs\cd\widetilde{\ch}(\Lambda^\imath)}
\newcommand{\tCMH}{{\cc\widetilde{\ch}(\K Q,\btau)}}

\newcommand{\rMH}{\operatorname{\cm\ch_{\rm{red}}(\Lambda^\imath)}\nolimits}
\newcommand{\utMHg}{\operatorname{\ch(Q,\btau)}\nolimits}
\newcommand{\tMHg}{\operatorname{\widetilde{\ch}(Q,\btau)}\nolimits}
\newcommand{\tMHk}{{\widetilde{\ch}(\K Q,\btau)}}
\newcommand{\rMHg}{\operatorname{\ch_{\rm{red}}(Q,\btau)}\nolimits}

\newcommand{\rMHd}{\operatorname{\cm\ch_{\rm{red}}(\Lambda^\imath)_{\bvsd}}\nolimits}
\newcommand{\tMHd}{\operatorname{\cs\cd\widetilde{\ch}(\Lambda^\imath)_{\bvsd}}\nolimits}

\newcommand{\tMHl}{\cs\cd\widetilde{\ch}({\bs}_\ell\Lambda^\imath)}
\newcommand{\rMHl}{\cm\ch_{\rm{red}}({\bs}_\ell\Lambda^\imath)_{\bvsd}}
\newcommand{\tMHi}{\cs\cd\widetilde{\ch}({\bs}_i\Lambda^\imath)}
\newcommand{\rMHi}{\cm\ch_{\rm{red}}({\bs}_i\Lambda^\imath)_{\bvsd}}
\newcommand{\tMHgi}{\widetilde{\ch}({\bs}_i Q,\btau)}

\newcommand{\utGpg}{\operatorname{\ch^{\rm Gp}(Q,\btau)}\nolimits}
\newcommand{\tGpg}{\operatorname{\widetilde{\ch}^{\rm Gp}(Q,\btau)}\nolimits}
\newcommand{\rGpg}{\operatorname{\ch_{red}^{\rm Gp}(Q,\btau)}\nolimits}

\newcommand{\colim}{\operatorname{colim}\nolimits}
\newcommand{\gldim}{\operatorname{gl.dim}\nolimits}
\newcommand{\cone}{\operatorname{cone}\nolimits}
\newcommand{\rep}{\operatorname{rep}\nolimits}
\newcommand{\Ext}{\operatorname{Ext}\nolimits}
\newcommand{\Tor}{\operatorname{Tor}\nolimits}
\newcommand{\Hom}{\operatorname{Hom}\nolimits}
\newcommand{\Top}{\operatorname{top}\nolimits}
\newcommand{\Coker}{\operatorname{Coker}\nolimits}
\newcommand{\thick}{\operatorname{thick}\nolimits}
\newcommand{\rank}{\operatorname{rank}\nolimits}
\newcommand{\Gproj}{\operatorname{Gproj}\nolimits}
\newcommand{\Len}{\operatorname{Length}\nolimits}
\newcommand{\RHom}{\operatorname{RHom}\nolimits}
\renewcommand{\deg}{\operatorname{deg}\nolimits}
\renewcommand{\Im}{\operatorname{Im}\nolimits}
\newcommand{\Ker}{\operatorname{Ker}\nolimits}
\newcommand{\Coh}{\operatorname{Coh}\nolimits}
\newcommand{\Id}{\operatorname{Id}\nolimits}
\newcommand{\Qcoh}{\operatorname{Qch}\nolimits}
\newcommand{\CM}{\operatorname{CM}\nolimits}
\newcommand{\sgn}{\operatorname{sgn}\nolimits}
\newcommand{\Gdim}{\operatorname{G.dim}\nolimits}
\newcommand{\fpr}{\operatorname{\mathcal{P}^{\leq1}}\nolimits}

\newcommand{\For}{\operatorname{{\bf F}or}\nolimits}
\newcommand{\coker}{\operatorname{Coker}\nolimits}
\renewcommand{\dim}{\operatorname{dim}\nolimits}
\newcommand{\rankv}{\operatorname{\underline{rank}}\nolimits}
\newcommand{\dimv}{{\operatorname{\underline{dim}}\nolimits}}
\newcommand{\diag}{{\operatorname{diag}\nolimits}}
\newcommand{\qbinom}[2]{\begin{bmatrix} #1\\#2 \end{bmatrix} }

\renewcommand{\Vec}{{\operatorname{Vec}\nolimits}}
\newcommand{\pd}{\operatorname{proj.dim}\nolimits}
\newcommand{\gr}{\operatorname{gr}\nolimits}
\newcommand{\id}{\operatorname{Id}\nolimits}
\newcommand{\Res}{\operatorname{Res}\nolimits}
\def \tT{\widetilde{\mathcal T}}
\def \tTL{\tT(\Lambda^\imath)}

\newcommand{\mbf}{\mathbf}
\newcommand{\mbb}{\mathbb}
\newcommand{\mrm}{\mathrm}
\newcommand{\cbinom}[2]{\left\{ \begin{matrix} #1\\#2 \end{matrix} \right\}}
\newcommand{\dvev}[1]{{B_1|}_{\ev}^{{(#1)}}}
\newcommand{\dv}[1]{{B_1|}_{\odd}^{{(#1)}}}
\newcommand{\dvd}[1]{t_{\odd}^{{(#1)}}}
\newcommand{\dvp}[1]{t_{\ev}^{{(#1)}}}
\newcommand{\ev}{\bar{0}}
\newcommand{\odd}{\bar{1}}
\newcommand{\Iblack}{\I_{\bullet}}
\newcommand{\wb}{w_\bullet}
\newcommand{\Uidot}{\dot{\bold{U}}^{\imath}}

\newcommand{\kk}{h}
\newcommand{\la}{\lambda}
\newcommand{\LR}[2]{\left\llbracket \begin{matrix} #1\\#2 \end{matrix} \right\rrbracket}
\newcommand{\ff}{B}
\newcommand{\pdim}{\operatorname{proj.dim}\nolimits}
\newcommand{\idim}{\operatorname{inj.dim}\nolimits}
\newcommand{\Gd}{\operatorname{G.dim}\nolimits}
\newcommand{\Ind}{\operatorname{Ind}\nolimits}
\newcommand{\add}{\operatorname{add}\nolimits}
\newcommand{\pr}{\operatorname{pr}\nolimits}
\newcommand{\oR}{\operatorname{R}\nolimits}
\newcommand{\oL}{\operatorname{L}\nolimits}
\newcommand{\ext}{{ \mathfrak{Ext}}}
\newcommand{\Perf}{{\mathfrak Perf}}
\def\scrP{\mathscr{P}}
\newcommand{\bk}{{\mathbb K}}
\newcommand{\cc}{{\mathcal C}}
\newcommand{\gc}{{\mathcal GC}}
\newcommand{\dg}{{\rm dg}}
\newcommand{\ce}{{\mathcal E}}
\newcommand{\cs}{{\mathcal S}}
\newcommand{\cl}{{\mathcal L}}
\newcommand{\cf}{{\mathcal F}}
\newcommand{\cx}{{\mathcal X}}
\newcommand{\cy}{{\mathcal Y}}
\newcommand{\ct}{{\mathcal T}}
\newcommand{\cu}{{\mathcal U}}
\newcommand{\cv}{{\mathcal V}}
\newcommand{\cn}{{\mathcal N}}
\newcommand{\mcr}{{\mathcal R}}
\newcommand{\ch}{{\mathcal H}}
\newcommand{\ca}{{\mathcal A}}
\newcommand{\cb}{{\mathcal B}}
\newcommand{\ci}{{\I}_{\btau}}
\newcommand{\cj}{{\mathcal J}}
\newcommand{\cm}{{\mathcal M}}
\newcommand{\cp}{{\mathcal P}}
\newcommand{\cg}{{\mathcal G}}
\newcommand{\cw}{{\mathcal W}}
\newcommand{\co}{{\mathcal O}}
\newcommand{\cq}{{Q^{\rm dbl}}}
\newcommand{\cd}{{\mathcal D}}
\newcommand{\ck}{\widetilde{\mathcal K}}
\newcommand{\calr}{{\mathcal R}}
\newcommand{\iLa}{\Lambda^{\imath}}
\newcommand{\La}{\Lambda}
\newcommand{\ol}{\overline}
\newcommand{\ul}{\underline}
\newcommand{\st}{[1]}
\newcommand{\ow}{\widetilde}
\renewcommand{\P}{\mathbf{P}}
\newcommand{\pic}{\operatorname{Pic}\nolimits}
\newcommand{\Spec}{\operatorname{Spec}\nolimits}

\newtheorem{theorem}{Theorem}[section]
\newtheorem{acknowledgement}[theorem]{Acknowledgement}
\newtheorem{algorithm}[theorem]{Algorithm}
\newtheorem{assumption}[theorem]{Assumption}
\newtheorem{axiom}[theorem]{Axiom}
\newtheorem{case}[theorem]{Case}
\newtheorem{claim}[theorem]{Claim}
\newtheorem{conclusion}[theorem]{Conclusion}
\newtheorem{condition}[theorem]{Condition}
\newtheorem{conjecture}[theorem]{Conjecture}
\newtheorem{construction}[theorem]{Construction}
\newtheorem{corollary}[theorem]{Corollary}
\newtheorem{criterion}[theorem]{Criterion}
\newtheorem{definition}[theorem]{Definition}
\newtheorem{example}[theorem]{Example}
\newtheorem{exercise}[theorem]{Exercise}
\newtheorem{lemma}[theorem]{Lemma}
\newtheorem{notation}[theorem]{Notation}
\newtheorem{problem}[theorem]{Problem}
\newtheorem{proposition}[theorem]{Proposition}
\newtheorem{solution}[theorem]{Solution}
\newtheorem{summary}[theorem]{Summary}
\numberwithin{equation}{section}

\theoremstyle{remark}
\newtheorem{remark}[theorem]{Remark}
\newcommand{\Pd}{\pi_*}
\def \bvs{{\boldsymbol{\varsigma}}}
\def \bvsd{{\boldsymbol{\varsigma}_{\diamond}}}
\def \btau{{{\tau}}}

\def \bp{{\mathbf p}}
\def \bq{{\bm q}}
\def \bv{{v}}
\def \bs{{\bm s}}

\def \bfK{{\mathbf K}}

\newcommand{\tCMHg}{\cc\widetilde{\ch}(Q,\btau)}
\newcommand{\bfv}{\mathbf{v}}
\def \bA{{\mathbf A}}
\def \ba{{\mathbf a}}
\def \bL{{\mathbf L}}
\def \bF{{\mathbf F}}
\def \bS{{\mathbf S}}
\def \bC{{\mathbf C}}
\def \bU{{\mathbf U}}
\def \bc{{\mathbf c}}
\def \fpi{\mathfrak{P}^\imath}
\def \Ni{N^\imath}
\def \fp{\mathfrak{P}}
\def \fg{\mathfrak{g}}
\def \fk{\fg^\theta}  

\def \fn{\mathfrak{n}}
\def \fh{\mathfrak{h}}
\def \fu{\mathfrak{u}}
\def \fv{\mathfrak{v}}
\def \fa{\mathfrak{a}}
\def \fq{\mathfrak{q}}
\def \Z{{\Bbb Z}}
\def \F{{\Bbb F}}
\def \D{{\Bbb D}}
\def \C{{\Bbb C}}
\def \N{{\Bbb N}}
\def \Q{{\Bbb Q}}
\def \G{{\Bbb G}}
\def \P{{\Bbb P}}
\def \K{{\mathbf k}}
\def \bK{{\Bbb K}}

\def \E{{\Bbb E}}
\def \A{{\Bbb A}}
\def \L{{\Bbb L}}
\def \I{{\Bbb I}}
\def \BH{{\Bbb H}}
\def \T{{\Bbb T}}
\newcommand{\TT}{\operatorname{\texttt{\rm T}}\nolimits}
\newcommand {\lu}[1]{\textcolor{red}{$\clubsuit$: #1}}

\newcommand{\nc}{\newcommand}
\newcommand{\browntext}[1]{\textcolor{brown}{#1}}
\newcommand{\greentext}[1]{\textcolor{green}{#1}}
\newcommand{\redtext}[1]{\textcolor{red}{#1}}
\newcommand{\bluetext}[1]{\textcolor{blue}{#1}}
\newcommand{\brown}[1]{\browntext{ #1}}
\newcommand{\green}[1]{\greentext{ #1}}
\newcommand{\red}[1]{\redtext{ #1}}
\newcommand{\blue}[1]{\bluetext{ #1}}

\newcommand{\wtodo}{\todo[inline,color=orange!20, caption={}]}
\newcommand{\lutodo}{\todo[inline,color=green!20, caption={}]}

\title[Hall algebras and quantum symmetric pairs of Kac-Moody type]{Hall algebras and quantum symmetric pairs of Kac-Moody type}

\author[Ming Lu]{Ming Lu}
\address{Department of Mathematics, Sichuan University, Chengdu 610064, P.R.China}
\email{luming@scu.edu.cn}

\author[Weiqiang Wang]{Weiqiang Wang}
\address{Department of Mathematics, University of Virginia, Charlottesville, VA 22904, USA}
\email{ww9c@virginia.edu}

\subjclass[2010]{Primary 17B37, 
16E45, 18E30.}  
\keywords{Hall algebras, Quantum groups, Quantum symmetric pairs, $\imath$Serre relation}

\begin{abstract}
We extend our $\imath$Hall algebra construction from acyclic to arbitrary $\imath$quivers, where the $\imath$quiver algebras are infinite-dimensional 1-Gorenstein in general. Then we establish an injective homomorphism from the universal $\imath$quantum group of Kac-Moody type arising from quantum symmetric pairs to the $\imath$Hall algebra associated to a virtually acyclic $\imath$quiver.
\end{abstract}

\maketitle
 \setcounter{tocdepth}{1}
 \tableofcontents

\section{Introduction}

\subsection{}

In \cite{LW19a}, the authors formulated the $\imath$Hall algebras, denoted by $\tMHk$ in this paper, of $\imath$quiver algebras $\iLa$ associated to acyclic $\imath$quivers $(Q,\tau)$ over a finite field $\K= \mathbb F_q$ in the framework of semi-derived Ringel-Hall algebras of 1-Gorenstein algebras \cite{LP, Lu19}. 
 This new form of Hall algebras was motivated by the constructions of Bridgeland's Hall algebra of complexes \cite{Br} and Gorsky's semi-derived Hall algebras \cite{Gor1, Gor2} (which were in turn built on \cite{Rin, L90, Gr95}; for a survey see \cite{Sch06}). The $\imath$Hall algebras of $\imath$quiver algebras were conjectured to provide a realization of the universal $\imath$quantum groups arising from quasi-split quantum symmetric pairs of Kac-Moody type, and for finite type this was established in \cite{LW19a}.

 Bridgeland's Hall algebra construction in \cite{Br} produces the Drinfeld double $\tU$ of a quantum group $\U$, and our $\imath$Hall algebra construction produces a universal $\imath$quantum group $\tUi$. The main difference between the $\imath$quantum groups $\Ui$ (namely, the quantum symmetric pair coideal subalgebra of $\U$) \`a la G. ~Letzter \cite{Let99} and the universal $\imath$quantum groups $\tUi$ (a coideal subalgebra of $\tU$) in \cite{LW19a} is that $\Ui$ depends on various parameters while $\tUi$ admit various central elements. A central reduction of $\tUi$ recovers $\Ui$.

We view $\imath$quantum groups as a vast generalization of Drinfeld-Jimbo quantum groups, and aim at extending various fundamental constructions from quantum groups to $\imath$quantum groups \cite{BW18a} (see also \cite{BW18b, FL+20}). Bridgeland's Hall algebra realization of a quantum group \cite{Br} has been reformulated in \cite{LW19a} as $\imath$Hall algebra for $\imath$quivers of diagonal type, just as a quantum group can be viewed as an $\imath$quantum group of diagonal type.

A Serre presentation of {\em quasi-split} $\imath$quantum groups $\Ui$ of Kac-Moody type is more complicated than a Serre presentation (which is the definition) of a quantum group, and it was recently completed in full generality in our work joint with X.~Chen \cite{CLW18}. Our work was built on partial results in \cite{Ko14, BK15} in Kac-Moody setting; a complete presentation of $\Ui$ in finite type was already given earlier by Letzter \cite{Let02}. A crucial relation, known as the $\imath$Serre relation, in the final presentation for $\Ui$, involves the $\imath$divided powers which originated from the theory of canonical basis for quantum symmetric pairs \cite{BW18a, BeW18}. The $\imath$divided powers come in 2 forms, depending on a parity.

\subsection{}

In this paper, we first extend the definition of $\imath$Hall algebra from acyclic $\imath$quivers as treated in \cite{LW19a} to general $\imath$quivers (allowing oriented cycles), $(Q,\tau)$. Since the $\imath$quiver algebra $\Lambda^\imath$ associated to a non-acyclic $\imath$quiver is infinite-dimensional (and still 1-Gorenstein), the technique of Bridgeland's Hall algebras or Gorsky's semi-derived Hall algebras does not seem to apply.  However, the foundation (such as singularity categories and Hall basis) for the semi-derived Ringel-Hall algebra of 1-Gorenstein algebras in \cite{Lu19} (see also \cite{LP}) can be extended to this infinite-dimensional setting. To keep the exposition at a reasonable length, we have chosen to focus on formulating the $\imath$Hall algebra $\tMHk$ and its main properties (instead of treating general 1-Gorenstein algebras; see Remark~\ref{rem:MH}). Some of these new technical developments can be applied to shed new light to the Hall algebra realization of Drinfeld-Jimbo quantum groups via non-acyclic quivers.

Motivated by the connection to the $\imath$quantum groups $\tUi$, we formulate the notion of virtually acyclic $\imath$quivers; see Definition~\ref{def:i-acyclic}. The virtually acyclic $\imath$quivers include all acyclic $\imath$quivers, but also allow the generalized Kronecker $\imath$quivers $Q$ \eqref{eq:Kq} as new rank one $\imath$subquivers.
By the requirement of $\imath$quivers that the nontrivial involution $\tau$ preserves the generalized Kronecker quiver $Q$, the number of arrows in $Q$ is necessarily even. This translates into that the generalized Cartan matrix $C=(c_{ij})_{i,j\in \I}$ associated to $Q$ satisfies that $c_{i, \tau i} \in -2 \N$ whenever $i\neq \tau i$. (In the setting of \cite{LW19a}, the acyclic condition on $\imath$quivers imposes that $c_{i, \tau i} =0$ whenever $i\neq \tau i$.)

The $\imath$divided powers in the setting of $\tUi$ are formulated in \eqref{eq:iDPodd}--\eqref{eq:iDPev}, by suitably modifying earlier versions in $\Ui$ in various generalities (cf. \cite{BW18a, BeW18, CLW18, Li20}). These $\imath$divided powers are then used to provide a presentation of $\tUi$ with generators $B_i, \tk_i$\; $(i\in \I)$ subject to relations \eqref{relation1}--\eqref{relation6} in  Theorem~\ref{thm:Serre}, a variant of the presentation for $\Ui$ in \cite{CLW18}.

With the above constructions in place, we are ready to formulate the main result of this paper, which generalizes \cite[Theorem~7.7]{LW19a} for ADE type and, in case of acyclic $\imath$quivers, settles \cite[Conjecture~7.9]{LW19a} completely. Set $\sqq =\sqrt{q}$.

\begin{customthm}{} [Theorem~\ref{thm:main}]
 Let $(Q, \btau)$ be a virtually acyclic $\imath$quiver. Then there exists a $\Q(\sqq)$-algebra monomorphism
$ 
\widetilde{\psi}: \tUi_{|v= \sqq} \longrightarrow \tMHk,
$ 
which sends
\begin{align*}
B_j \mapsto \frac{-1}{q-1}[S_{j}],\text{ if } j\in\ci,
&\qquad\qquad
\tk_i \mapsto - q^{-1}[\bK_i], \text{ if }\btau i=i \in \I;
  \\
B_{j} \mapsto \frac{{\sqq}}{q-1}[S_{j}],\text{ if }j\notin \ci,
&\qquad\qquad
\tk_i \mapsto \sqq^{\frac{-c_{i,\btau i}}{2}}[\bK_i],\quad \text{ if }\btau i\neq i \in \I.
  \end{align*}
\end{customthm}

\subsection{}

There are 2 relations for $\tUi$ which are quite involved at this level of generality, namely the BK relation \eqref{relation5} (which goes back to \cite{BK15}) and the $\imath$Serre relation \eqref{relation6}. The main new technical difficulty in showing that $\widetilde{\psi}$ in the Main Theorem is a homomorphism is to verify the BK relation \eqref{relation5} and especially the $\imath$Serre relation \eqref{relation6} in the $\imath$Hall algebra $\tMHk$. (In contrast, in the ADE setting of \cite{LW19a}, the relation \eqref{relation5} is easy thanks to $i\neq \tau i$ and hence $c_{i, \tau i} =0$, while the $\imath$Serre relation \eqref{relation6} for $c_{ij}=-1$ is all one needs to verify.)

The proof of the relation \eqref{relation5} in the $\imath$Hall algebra $\tMHk$ requires some interesting Hall algebra computation in Section~\ref{sec:Relation5}. In particular, we are able to see clearly how the two summands in \eqref{relation5} arise from the viewpoint of Hall algebra.

The verification of the $\imath$Serre relation \eqref{relation6} in the $\imath$Hall algebra setting is highly nontrivial and occupies Sections~\ref{sec:iDP} through \ref{sec:comb}. The strategy here bears some similarities with that used in establishing the $\imath$Serre relation for $\Ui$; see \cite{CLW18}. The expansion formulas \cite{BeW18} for the $\imath$divided powers in terms of PBW basis of $\U$ are used {\em loc. cit.} to reduce the verification of the $\imath$Serre relation in $\Ui$ to some new $v$-binomial identity, which was then established after some serious work.

In the current $\imath$Hall algebra setting, we first establish closed formulas for the $\imath$divided powers in terms of an $\imath$Hall basis; see Propositions~\ref{prop:iDPev}--\ref{prop:iDPodd}. These formulas are of independent interest and have other applications; for example, they will play a basic role in our forthcoming work when we develop further the reflection functors on $\imath$Hall algebras \cite{LW19b} to establish a conjecture in \cite{CLW20} on the braid group action on $\tUi$. The existence of such closed formulas (as well as those in \cite{BeW18}) is in our view a manifestation of the basic nature of $\imath$divided powers. (In contrast, closed formulas for monomials $[S_i]^{*n}$ or $B_i^n$, for $i=\tau i$ and $n\in \N$, in terms of Hall basis or PBW basis are unknown.)

Next we convert the summation in the $\imath$Serre relation (which are defined via $\imath$divided powers) into a linear combination of the $\imath$Hall basis, and a new quantum binomial identity arises this way. We eventually reduce the proof of this identity (see Proposition~\ref{prop:T=0}) further to the following identities  \eqref{eq:km5}--\eqref{eq:kmrd}: for $p, d\ge 1$,
\begin{align*}
\sum_{k=0}^{p} v^{-k(p-k+1)} \qbinom{p}{k}  & =  \prod_{j=1}^p (1+v^{-j}),
\qquad
\sum_{\stackrel{k,m,r \in \N}{k+m+r =d}} (-1)^r \frac{v^{{r+1 \choose 2} -2(k-1)m}}{[r]^! [2k]^{!!} [2m]^{!!}} =0.
\end{align*}
(The first $v$-binomial formula here is non-standard, and as we learned from G.~Andrews, it is a variant of a known identity of Rogers-Szeg\"o polynomials, cf. \cite[Exercise~ 5, pp.49]{An98}.)

Both $\tUi$ and $\tMHk$ admit natural filtered algebra structures, whose associated graded are half a quantum group $\U^-$ and Ringel-Hall algebra $\widetilde{\ch}(\K Q)$ over a quantum torus, respectively. Once we know that $\widetilde{\psi}$ is a homomorphism, the injectivity of $\widetilde{\psi}$ can be established by applying some filtered algebra argument and reducing to the main theorem of Ringel and Green on Hall algebra realization of $\U^-$.

\subsection{}
Note that a general quiver (possibly with loops) leads to a Borcherds-Cartan matrix, Borcherds Lie algebra and its corresponding quantum group. The theory of $\imath$Hall algebras developed in Sections~\ref{sec:quiver}--\ref{sec:Hall} and a conjectural extension of Theorem~\ref{thm:main} for general $\imath$quivers call for a development of a
 theory of quantum symmetric pairs and $\imath$quantum groups associated to Borcherds-Cartan matrices, which should be of independent interest.

\subsection{}

The paper is organized as follows.
In Section~\ref{sec:quiver}, following and generalizing \cite{LW19a}, we formulate the $\imath$quiver algebras, their projective modules and singularity categories in the generality of arbitrary $\imath$quivers. This requires us to overcome various technical issues.
The $\imath$Hall algebras of $\imath$quiver algebras associated to general $\imath$quivers and their $\imath$Hall bases are established in Section~\ref{sec:Hall}.

In Section~\ref{sec:iQG}, we review and set up notations for quantum groups and $\imath$quantum groups. A Serre presentation for $\tUi$ is formulated.
The verification of the BK relation \eqref{relation5} in the $\imath$Hall algebra is taken up in Section~\ref{sec:Relation5}.

In Section~\ref{sec:iDP}, we formulate and establish the Hall basis expansion formulas for the $\imath$divided powers. These formulas are applied in Section~\ref{sec:SerreRel} to reduce the verification of the $\imath$Serre relation \eqref{relation6} in the $\imath$Hall algebra to a new $v$-binomial identity; the proof of this identity is given in section~\ref{sec:comb}.

Finally, in Section~\ref{sec:iQGHall} we verify the remaining defining relations for $\tUi$ in the $\imath$Hall algebra setting. We complete the proof of the main Theorem~\ref{thm:main}, providing a Hall algebra realization of the quasi-split $\imath$quantum groups of Kac-Moody type.

\vspace{2mm}
\noindent{\bf Acknowledgments.}
We thank Changjian Fu and Yang Han for helpful discussions on quiver algebras. ML thanks Liangang Peng for his continuing encouragement and helpful discussions on Hall algebras. ML thanks for University of Virginia for hospitality and support. We thank East China Normal University for hospitality and support which helps to facilitate this collaboration. WW is partially supported by the NSF grant DMS-1702254 and DMS-2001351.

\section{$\imath$Quiver algebras and homological properties}
 \label{sec:quiver}


In this section, we review and generalize the $\imath$quiver algebras $\Lambda^\imath$ and their homological properties from acyclic $\imath$quivers to general $\imath$quivers allowing oriented cycles. Following \cite[\S3]{LW19a}, we shall prove that $\Lambda^\imath$ is $1$-Gorenstein algebra, describe its singularity category $D_{sg}(\mod(\Lambda^\imath))$ and characterize the finite-dimensional nilpotent modules of finite projective dimensions. However, since the $\imath$quiver algebra $\Lambda^\imath$ may be infinite-dimensional, various results for $\fgm(\Lambda^\imath)$ and $D_{sg}(\fgm(\Lambda^\imath))$, known for $\iLa$  finite-dimensional, have to be reestablished for $\mod(\Lambda^\imath)$ and $D_{sg}(\mod(\Lambda^\imath))$ (see Lemma \ref{lem: resolution} and Lemma \ref{lem: iso in singularity}).

\subsection{Notations}

Let $\K$ be a field.
For a quiver algebra $A=\K Q/I$ (not necessarily finite-dimensional), a representation $V=(V_i,V(\alpha))_{i\in Q_0,\alpha\in Q_1}$ of $A$ is called {\em nilpotent} if for each oriented cycle $\alpha_m\cdots\alpha_1$ at a vertex $i$, the $\K$-linear map $V(\alpha_m)\cdots V(\alpha_1):V_i\rightarrow V_i$ is nilpotent. We denote

$\triangleright$ $\fgm(A)$ -- category of finitely generated (left) $A$-modules

$\triangleright$ $\proj(A)$ -- category of finitely generated projective $A$-modules

$\triangleright$ $\mod(A)$ -- category of finite-dimensional nilpotent $A$-modules

$\triangleright$ $K^b(\proj(A))$ -- bounded homotopy category of $\proj(A)$

$\triangleright$ $D^b(\fgm(A))$ -- bounded derived category of $\fgm(A)$, with shift functor $\Sigma$

$\triangleright$ $D^b(\mod(A))$ -- bounded derived category for $\mod(A)$

$\triangleright$ ${\rm proj.dim}_AM$ -- projective dimension of an $A$-module $M$

$\triangleright$ ${\rm inj.dim}_AM$ -- injective dimension of $M$


%
%
\subsection{The $\imath$quiver algebras}
\label{subsec:iquiveralg}

Let $Q=(Q_0,Q_1)$ be a general quiver (where oriented cycles are allowed). Throughout the paper, we shall identify $Q_0=\I$. An {\em involution} of $Q$ is defined to be an automorphism $\btau$ of the quiver $Q$ such that $\btau^2=\Id$. In particular, we allow the {\em trivial} involution $\Id:Q\rightarrow Q$. An involution $\btau$ of $Q$ induces an involution of the path algebra $\K Q$, again denoted by $\btau$.
A quiver together with an involution $\btau$, $(Q, \btau)$, will be called an {\em $\imath$quiver}.

Let $R_1$ denote the truncated polynomial algebra $\K[\varepsilon]/(\varepsilon^2)$.
Let $R_2$ denote the radical square zero of the path algebra of $\xymatrix{1 \ar@<0.5ex>[r]^{\varepsilon} & 1' \ar@<0.5ex>[l]^{\varepsilon'}}$, i.e., $\varepsilon' \varepsilon =0 =\varepsilon\varepsilon '$. Define a $\K$-algebra
\begin{equation}
  \label{eq:La}
\Lambda=\K Q\otimes_\K R_2.
\end{equation}

Associated to the quiver $Q$, the {\em double framed quiver} $Q^\sharp$ is the quiver such that
\begin{itemize}
\item
the vertex set of $Q^{\sharp}$ consists of 2 copies of the vertex set $Q_0$, $\{i,i'\mid i\in Q_0\}$;
\item
the arrow set of $Q^{\sharp}$ is
\[
\{\alpha: i\rightarrow j,\alpha': i'\rightarrow j'\mid(\alpha:i\rightarrow j)\in Q_1\}\cup\{ \varepsilon_i: i\rightarrow i' ,\varepsilon'_i: i'\rightarrow i\mid i\in Q_0 \}.
\]
\end{itemize}
Note $Q^\sharp$ admits a natural involution, denoted by $\swa$. The involution $\btau$ of a quiver $Q$ induces an involution ${\btau}^{\sharp}$ of $Q^{\sharp}$ which is basically the composition of $\swa$ and $\btau$ (on the two copies of subquivers $Q$ and $Q'$ of $Q^\sharp$), cf. \cite[\S2.1]{LW19a}.

The algebra $\Lambda$ can be described in terms of the quiver $Q^{\sharp}$ with relations \cite[\S2.2]{LW19a}. More precisely, we have $\Lambda\cong \K Q^{\sharp} \big/ I^{\sharp}$, where $I^{\sharp}$ is the admissible ideal of $\K Q^{\sharp}$ generated by
\begin{itemize}
\item
$\varepsilon_i \varepsilon_i'$, $\varepsilon_i'\varepsilon_i$ for each $i\in Q_0$;
\item
$\varepsilon_j' \alpha' -\alpha\varepsilon_i'$, $\varepsilon_j \alpha -\alpha'\varepsilon_i$ for each $(\alpha:i\rightarrow j)\in Q_1$.
\end{itemize}

By \cite[Lemma~2.4]{LW19a}, ${\btau}^{\sharp}$ on $Q^\sharp$ preserves $I^\sharp$ and hence induces an involution ${\btau}^{\sharp}$ on the algebra $\Lambda$. The {\em $\imath$quiver algebra} of $(Q, \btau)$ is defined to be the ${\btau}^{\sharp}$-fixed point subalgebra of $\Lambda$:
\begin{equation}
   \label{eq:iLa}
\iLa
= \{x\in \Lambda\mid {\btau}^{\sharp}(x) =x\}.
\end{equation}

Let $\ov{Q}$ be a new quiver obtained from $Q$ by adding a loop $\varepsilon_i$ at the vertex $i\in Q_0$ if $\btau i=i$, and adding an arrow $\varepsilon_i: i\rightarrow \btau i$ for each $i\in Q_0$ if $\btau i\neq i$. The algebra $\iLa$ can be described in terms of the quiver $\ov Q$ with relations, cf. \cite[Proposition 2.6]{LW19a}; that is, $\iLa \cong \K \ov{Q} / \ov{I}$, where $\ov{I}$ is generated by
\begin{itemize}
\item
$\varepsilon_{i}\varepsilon_{\btau i}$ for each $i\in Q_0$;
\item
$\varepsilon_i\alpha-\btau(\alpha)\varepsilon_j$ for each arrow $\alpha:j\rightarrow i$ in $Q_1$.
\end{itemize}
The algebras $\Lambda$ and $\Lambda^\imath$ are finitely generated and hence are Neotherian.
Note also that $\Lambda^\imath$ is finite dimensional if and only if $Q$ is acyclic. We call $(Q,\btau)$ an {\em acyclic $\imath$quiver} if $Q$ is acyclic.

Note that $\Lambda^{\imath}$ is an $\N$-graded algebra, $\Lambda^\imath=\Lambda^{\imath}_0 \bigoplus \Lambda^\imath_1$,  where $\Lambda^{\imath}_0= \K Q$, with the grading $|\cdot |$ given by
$|\varepsilon_i|=1,
|\alpha| =0
$,
for $i\in \I$ and $\alpha$ in $Q\subseteq \ov{Q}$. It follows that $\K Q$ is naturally a subalgebra and also a quotient algebra of $\Lambda^\imath$, cf. \cite[Corollary 2.12]{LW19a}.


%
%
\subsection{$\Lambda^\imath$ as a $1$-Gorenstein algebra}
 \label{subsec:Goren}

Similar to \cite[Remark 2.11]{LW19a}, we obtain a pushdown functor
\begin{align}  \label{eq:pi}
\Pd:\mod (\Lambda) \longrightarrow \mod (\Lambda^{\imath}).
\end{align}
In particular, $\Pd$ preserves projective modules, injective modules, and the almost split sequences. However, $\Pd$ may not be dense in general.
$\Pd$ admits a left and also right adjoint functor, i.e., the pullup functor $\pi^*: \mod(\Lambda^\imath)\longrightarrow \mod(\Lambda)$.

Viewing $\K Q$ as a subalgebra of $\Lambda^{\imath}$, we have restriction functors
\begin{equation*}
\res: \fgm\Lambda^\imath)\longrightarrow \fgm (\K Q),\qquad
\res: \mod (\Lambda^{\imath})\longrightarrow \mod (\K Q);
\end{equation*}
viewing $\K Q$ as a quotient algebra of $\Lambda^{\imath}$, we obtain pullback functors
\begin{equation}\label{eqn:rigt adjoint}
\iota:\fgm(\K Q)\longrightarrow\fgm(\Lambda^{\imath}), \qquad
\iota:\mod(\K Q)\longrightarrow\mod(\Lambda^{\imath}).
\end{equation}
In this way, we can and shall view $\fgm(\K Q)$ (respectively,  $\mod(\K Q)$) as subcategory of $\fgm(\Lambda^{\imath})$ (respectively,  $\mod(\Lambda^{\imath})$).

Let $\cc_{\Z/2}(\fgm(\K Q))$ be the category of $\Z/2$-graded complexes over $\fgm(\K Q)$.
We shall identify $\fgm(\Lambda)\cong \cc_{\Z/2}(\fgm(\K Q))$ below.
For $i\in Q_0$, we denote by $P_i$ the indecomposable projective $\K Q$-module $(\K Q)e_i$.

\begin{lemma} \cite[Proposition 3.11]{LW19a}
  \label{prop:projective module of lambdai}
A $\Lambda^{\imath}$-module $X=(X_i,X(\alpha), X(\varepsilon_i))_{i\in Q_0,\alpha\in Q_1}$ is isomorphic to an indecomposable projective $\Lambda^\imath$-module $\Lambda^\imath e_j$ if and only if
\begin{equation*}
\begin{cases}
\text{the $\K Q$-module $(X_i,X(\alpha))_{i\in Q_0}$ is equal to $P_j\oplus P_{\btau j}$}
\\
\text{and $X(\varepsilon_j)$ is a linear isomorphism, }
\end{cases}
\end{equation*}
for some $j\in Q_0$; see \eqref{eqn:rigt adjoint}. In particular, we have a short exact sequence in $\mod(\Lambda^\imath):$
\begin{align*}
0\longrightarrow P_{\btau j}\longrightarrow (\Lambda^\imath) e_j \longrightarrow P_{j}\longrightarrow0.
\end{align*}
\end{lemma}




Similarly, one can describe the injective $\Lambda^\imath$-modules.

Following \cite{Ha3,EJ}, a noetherian algebra $A$ is called {\em $d$-Gorenstein} if $\ind{}_AA\leq d$ and $\ind A_A\leq d$.

\begin{proposition}
\label{proposition of 1-Gorenstein}
For a general $\imath$quiver $(Q,\btau)$, $\Lambda$ and $\Lambda^\imath$ are $1$-Gorenstein algebras.
\end{proposition}

\begin{proof}
The proof in \cite[Proposition 3.5(1)]{LW19a} works verbatim for a general $\imath$quiver.
%
\end{proof}

\subsection{Modules of finite projective dimensions}

Let $(Q, \btau)$ be an $\imath$quiver. Recall that $\Lambda^\imath=\K\ov{Q}/\ov{I}$ with $(\ov{Q}, \ov{I})$ as defined in \S\ref{subsec:iquiveralg}.
Following \cite[(2.7)]{LW19a}, for each $i\in Q_0$, define a $\K$-algebra
\begin{align}\label{dfn:Hi}
\BH _i:=\left\{ \begin{array}{cc}  \K[\varepsilon_i]/(\varepsilon_i^2) & \text{ if }\btau i=i,
 \\
\K(\xymatrix{i \ar@<0.5ex>[r]^{\varepsilon_i} & \btau i \ar@<0.5ex>[l]^{\varepsilon_{\btau i}}})/( \varepsilon_i\varepsilon_{\btau i},\varepsilon_{\btau i}\varepsilon_i)  &\text{ if } \btau i \neq i .\end{array}\right.
\end{align}
Note that $\BH _i=\BH _{\btau i}$ for any $i\in Q_0$. 
Choose one representative for each $\btau$-orbit on $\I =Q_0$, and let
\begin{align}   \label{eq:ci}
\ci = \{ \text{the chosen representatives of $\btau$-orbits in $\I$} \}.
\end{align}
Define a subalgebra of $\Lambda^{\imath}$:
\begin{equation}  \label{eq:H}
\BH =\bigoplus_{i\in \ci }\BH _i.
\end{equation}
Note that $\BH $ is a radical square zero selfinjective algebra. Denote by
\begin{align}
\res_\BH :\mod(\Lambda^\imath)\longrightarrow \mod(\BH )
\end{align}
the natural restriction functor.
As $\BH $ is a quotient algebra of $\Lambda^\imath$, every $\BH $-module can also be viewed as a $\Lambda^\imath$-module.

Recall the algebra $\BH _i$ for $i \in \ci$ from \eqref{dfn:Hi}. For $i\in Q_0 =\I$, define the indecomposable module over $\BH _i$ (if $i\in \ci$) or over $\BH_{\btau i}$ (if $i\not \in \ci$)
\begin{align}
  \label{eq:E}
\bK_i =\begin{cases}
\K[\varepsilon_i]/(\varepsilon_i^2), & \text{ if }\btau i=i;
\\
\xymatrix{\K\ar@<0.5ex>[r]^1 & \K\ar@<0.5ex>[l]^0} \text{ on the quiver } \xymatrix{i\ar@<0.5ex>[r]^{\varepsilon_i} & \btau i\ar@<0.5ex>[l]^{\varepsilon_{\btau i}} }, & \text{ if } \btau i\neq i.
\end{cases}
\end{align}
Then $\bK_i$, for $i\in Q_0$, can be viewed as a $\Lambda^\imath$-module and will be called a {\em generalized simple} $\Lambda^\imath$-module.

For any $\K$-algebra $A$, denote by $\cp^{\leq d}(A)$ the subcategory of $\mod(A)$ formed by modules of projective dimension $\leq d$ for any $d\in\N$.
Similarly, $\cp^{<\infty}(A)$ denotes the subcategory of $\mod(A)$ formed by modules of finite projective dimensions.


\begin{lemma}
\label{lem: res proj}
We have the following.

(a) $\pd_{\Lambda^{\imath}} (\bK_i)\leq1$ and $\ind_{\Lambda^{\imath}} (\bK_i)\leq1$ for any $i\in Q_0$.

(b) For any $M\in\mod(\Lambda^\imath)$, if $\pd_{\Lambda^\imath} M<\infty$, then $\res_\BH (M)$ is projective as $\BH $-module.
\end{lemma}
(We shall see from Corollary \ref{cor: res proj} below that the converse in (b) here also holds.)

\begin{proof}
(a). The proof is the same as for \cite[Lemma 3.7]{LW19a}.

(b). It follows from Lemma \eqref{prop:projective module of lambdai} that $\res_\BH(\Lambda^\imath e_i)$ is projective for any $i\in\I$. By considering the projective resolution of $M$ and applying the exact functor $\res_{\BH}$  the result follows.
\end{proof}

As we cannot find a suitable reference for the following result below, we include a proof here.

\begin{lemma}
\label{lem: resolution}
For any $M\in\mod(\Lambda^\imath)$, there exist short exact sequences
\begin{align}
\label{resolution 1}
0\longrightarrow M\longrightarrow H^M\longrightarrow X^M\longrightarrow0\\
\label{resolution 2}
0\longrightarrow X_M\longrightarrow H_M\longrightarrow M\longrightarrow0.
\end{align}
with $H^M,H_M\in \cp^{\leq1}(\Lambda^\imath)$.
\end{lemma}

\begin{proof}
We prove it by induction on the dimension of $M$.

First, if $M=S_i$, then we have a short exact sequence
$$0\longrightarrow S_i\longrightarrow \bK_{\btau i}\longrightarrow S_{\btau i}\longrightarrow0.$$

For any nonzero $M$, we have a short exact sequence for some $i\in \I$:
$$0\longrightarrow N\longrightarrow M\longrightarrow S_i\longrightarrow0.$$
By induction, there exists a short exact sequence
$$0\longrightarrow N\longrightarrow H^N\longrightarrow X^N\longrightarrow0$$
with $H^N\in\cp^{\leq1}(\Lambda^\imath)$.
We have the following commutative pushout diagram
\begin{equation*}
\xymatrix{ N\ar[r] \ar[d] &H^N\ar[r] \ar[d]^{f_1} & X^N\ar@{=}[d] \\
M\ar[r]^{g_1} \ar[d] & X\ar[r]\ar[d]^{f_2} & X^N\\
S_i\ar@{=}[r] &S_i }
\end{equation*}

Since $\Lambda^\imath$ is $1$-Gorenstein, we have $\ind_{\Lambda^\imath} H^N\leq1$ by \cite[Theorem 9.1.10]{EJ}. Then there exists a commutative diagram of short exact sequences:
\begin{equation}
\label{eqn: diag 2}
\xymatrix{ H^N\ar@{=}[r] \ar[d]^{f_1} & H^N \ar[d] \\
X \ar[r]^{g_2} \ar[d]^{f_2} & H^M \ar[r] \ar[d] & S_{\btau i}\ar@{=}[d] \\
S_i\ar[r] & \bK_{\tau i} \ar[r]& S_{\btau i} }
\end{equation}
by noting that $\Ext^2_{\Lambda^\imath}( S_{\btau i}, H^N)=0$.
We have $H^M\in \cp^{\leq1}(\Lambda^\imath)$ by using the short exact sequence in the second column of \eqref{eqn: diag 2}, and $g_2\circ g_1:M\longrightarrow H^M$ is injective. Hence the desired short exact sequence \eqref{resolution 1} follows.

Dually, one can prove the existence of the second short exact sequence \eqref{resolution 2}.
\end{proof}

\subsection{Singularity categories}
 \label{subsec:Sing}

\subsubsection{}

The \emph{singularity category} of $\fgm(A)$ is defined to be the Verdier localization
\begin{align*}
D_{sg}(\fgm(A)):=D^b(\fgm(A))/ K^b(\proj(A)).
\end{align*}
As $D^b(\mod(A))$ is a thick subcategory of $D^b(\fgm(A))$, we define the \emph{singularity category} $D_{sg}(\mod(A))$ of $\mod(A)$ to be the subcategory of $D_{sg}(\fgm(A))$ formed by all
objects in $D^b(\mod(A))$. Then $D_{sg}(\mod(A))$ is a triangulated category.

Note that $\Lambda^\imath$ is a $1$-Gorenstein algebra. Denote by
\begin{align}
\Gproj(\Lambda^\imath):=\{X\in\fgm(\Lambda^\imath)\mid \Ext^1_{\Lambda^\imath}(X,\Lambda^\imath)=0\}
\end{align}
the category of {\em Gorenstein projective modules}. 
Buchweitz-Happel's Theorem shows that $\Gproj(\Lambda^\imath)$ is a Frobenius category with projective modules as projective-injective objects, and its stable category $\underline{\Gproj}(\Lambda^\imath)$ is triangulated equivalent to $D_{sg}(\fgm(\Lambda^\imath))$:
\begin{align}
\Phi:\underline{\Gproj}(\Lambda^\imath)\stackrel{\simeq}{\longrightarrow}D_{sg}(\fgm(\Lambda^\imath)).
\end{align}





\begin{lemma}
\label{lem: existence SES}
For any projective $\Lambda^\imath$-module $V$ and for any $N>0$, there exists a short exact sequence
$$0\longrightarrow V'\xrightarrow{f} V\longrightarrow U\longrightarrow0$$
such that $U\in\cp^{\leq1}(\Lambda^\imath)$, and $f$ is induced by paths of length $\geq N$.
\end{lemma}

\begin{proof}
We can assume $V$ is indecomposable. Then $V=\pi_*(W)$ for some indecomposable projective $\Lambda$-module $W$. By \cite[Lemma 3.10]{LW19a}, without loss of generality, we assume $W$ to be of the form
$\xymatrix{ P \ar@<0.5ex>[r]^{1}& P \ar@<0.5ex>[l]^{0}  }$ for some indecomposable projective $\K Q$-module $P$. Let $Q$ be the submodule of $P$ generated by paths of length $\geq N$, then $Q$ is projective. We have $P/Q$ is finite-dimensional nilpotent $\K Q$-module. Let $V'= \xymatrix{ Q \ar@<0.5ex>[r]^{1}& Q \ar@<0.5ex>[l]^{0}  }$ and $U=\xymatrix{ P/Q \ar@<0.5ex>[r]^{1}& P/Q \ar@<0.5ex>[l]^{0}  }$. By applying $\pi_*$, the desired short exact sequence follows.
\end{proof}

The following lemma is well known for $\fgm(\Lambda^\imath)$, 
and we need to prove it for $\mod(\Lambda^\imath)$.

\begin{lemma}
\label{lem: iso in singularity}
For any $X,Y\in\mod(\Lambda^\imath)$, we have $X\cong Y$ in $D_{sg}(\mod(\Lambda^\imath))$ if and only if
there exist two short exact sequences
\begin{align*}
0\longrightarrow U_1\longrightarrow Z \longrightarrow X\longrightarrow0,
\qquad 0\longrightarrow U_2\longrightarrow Z\longrightarrow Y\longrightarrow0
\end{align*}
with $U_1,U_2\in\cp^{<\infty}(\Lambda^\imath)$, $Z\in\mod(\Lambda^\imath)$.
\end{lemma}

\begin{proof}
The ``if part'' follows by definition.

For the ``only if part'', since $\Lambda^\imath$ is $1$-Gorenstein, by \cite[Theorem 11.5.1]{EJ}, we have two short exact sequences
\begin{align*}
0\longrightarrow P_X \stackrel{f_1}{\longrightarrow} G_X\stackrel{f_2}{\longrightarrow} X\longrightarrow0,\qquad 0\longrightarrow P_Y\stackrel{g_1}{\longrightarrow}G_Y\stackrel{g_2}{\longrightarrow} Y\longrightarrow0
\end{align*}
with $P_X,P_Y\in\proj(\Lambda^\imath)$, and $G_X,G_Y\in\Gproj(\Lambda^\imath)$.
Since $X\cong Y$ in $D_{sg}(\mod(\Lambda^\imath))$, we have $G_X\cong G_Y$ in $\underline{\Gproj}(\Lambda^\imath)$. Without loss of generality, we can assume that
$G_X=G=G_Y$.

Consider $g_2\circ f_1: P_X\longrightarrow Y$. Since $Y$ is nilpotent, there exists $N>0$ such that $g_2\circ f_1(p)=0$, for any path $p$ of length $\geq N$.
By Lemma \ref{lem: existence SES}, there exists a nilpotent finite dimensional module $U_1\in\cp^{<\infty}(\Lambda^\imath)$ and a projective resolution
$$0\longrightarrow P \stackrel{h_1}{\longrightarrow} P_X\longrightarrow U_1\longrightarrow0$$
with $h_1$ induces by paths of length $\geq N$.
Then we have the following pushout diagram:
\[
\xymatrix{ P\ar@{=}[r] \ar[d]^{h_1} & P\ar[d] \\
P_X\ar[r]^{f_1} \ar[d] & G\ar[r]^{f_2} \ar[d]^{h_2} & X\ar@{=}[d] \\
U_1\ar[r] & Z \ar[r] &X  }
\]
Clearly, $Z\in\mod(\Lambda^\imath)$, and the third row gives us the first short exact sequence in the lemma.

By assumption, $g_2f_1h_1=0$. So $g_2$ factors through $h_2$, i.e., there exists $h:Z\longrightarrow Y$ such that $g_2=hh_2$.
Note that $h$ is epic. So we have the following commutative diagram of short exact sequences:
\[
\xymatrix{ P\ar[r] \ar@{=}[d]& P_Y \ar[r] \ar[d]^{g_1} & U_2 \ar[d] \\
P\ar[r]^{f_1h_1}& G\ar[r]^{h_2} \ar[d]^{g_2} & Z\ar[d]^{h} \\
&Y \ar@{=}[r] &Y }
\]
The exact sequence in the third column shows that $U_2\in\mod(\Lambda^\imath)$; and together with the short exact sequence in the first row, we have $U_2\in\cp^{\leq1}(\Lambda^\imath)$. Then the third column gives us the second short exact sequence in the lemma.
\end{proof}

\subsubsection{}

Let $\ct$ be an algebraic triangulated category with $\Sigma$ as its shift functor. We call $T\in\ct$ a \emph{partial tilting object} if $\Hom_\ct(T,\Sigma^iT)=0$ for any $i\neq0$. In this case, we have a triangulated embedding $K^b(\proj(\End(T)^{\text{op}}))\longrightarrow \ct$; see \cite{Ke2}.

Recall that $\Lambda^\imath$ is positively graded by $\deg\varepsilon_i=1$, $\deg\alpha=0$ for any $i\in Q_0$, $\alpha\in Q_1$. Note that $\Lambda^\imath_0=\K Q$.
Let $\fgmz(\Lambda^\imath)$ be the category of finitely generated graded $\iLa$-modules. One can define $D_{sg}(\fgmz(\Lambda^\imath))$ similarly; see, e.g., \cite[\S3.5]{LW19a}.

\begin{lemma} 
   \label{prop:tilting object}
The $\underline{T}= \Lambda^{\imath}_0$ is a partial tilting object in $D_{sg}(\fgmz(\Lambda^{\imath}))$, and its (opposite) endomorphism algebra is isomorphic to $\K Q$. In particular, we have the following triangulated embedding
\[
D^b(\fgm(\K Q))\longrightarrow D_{sg}(\fgmz(\Lambda^{\imath})).
\]
\end{lemma}

\begin{proof}
The proof is the same as for \cite[Proposition 3.14]{LW19a}, and hence omitted here.
%
%
\end{proof}

The triangulated embedding in Lemma~\ref{prop:tilting object}, denoted by $G$, is given by the composition of functors:
\begin{align}\label{eqn:G}
G:D^b(\fgm(\K Q))\xrightarrow{\underline{T}\otimes_{\K Q}^\L-} D^b(\fgmz (\Lambda^{\imath}))\stackrel{\pi}{\longrightarrow}  D_{sg}(\fgmz (\Lambda^{\imath})).
\end{align}
On the other hand, $\underline{T}$ is isomorphic to $\K Q$ as a $\Lambda^{\imath}$-$\K Q$-bimodule, so $( \underline{T}\otimes_{\K Q}-)\simeq \iota$, where $\iota$ is defined in \eqref{eqn:rigt adjoint}. So $G$ is equivalent to the composition
\[
D^b(\fgm(\K Q))\xrightarrow{D^b(\iota)} D^b(\fgmz (\Lambda^{\imath}))\stackrel{\pi}{\longrightarrow}  D_{sg}(\fgmz (\Lambda^{\imath})),
\]
where $D^b(\iota)$ is the derived functor of $\iota$ since $\iota$ is exact.
Moreover, the restriction of $G$ to $D^b(\mod(\K Q))$ gives a triangulated embedding
\[
D^b(\mod(\K Q))\xrightarrow{D^b(\iota)} D^b(\fdmz (\Lambda^{\imath}))\stackrel{\pi}{\longrightarrow}  D_{sg}(\fdmz (\Lambda^{\imath})),
\]

Let $\widehat{\btau}$ be the triangulated auto-equivalence of $D^b(\mod(\K Q))$ induced by $\btau$. Similar to \cite[Theorem 3.18]{LW19a}, we have
\begin{align}
\label{eqn: shift}
(1)\circ G\simeq G\circ \Sigma\circ\widehat{\btau}.
\end{align}

\begin{lemma}
(cf. \cite[Lemma 4.3]{Ha3})
  \label{lem: dense}
The natural functor $U:\mod(\Lambda^\imath)\longrightarrow D_{sg}(\mod(\Lambda^\imath))$ is dense.
\end{lemma}

\begin{proof}
Let $\cc^b(\mod(\Lambda^\imath))$ be the category of bounded complexes. Let $\cc^{-,b}(\cp^{\leq1}(\Lambda^\imath))$ be the category of complexes bounded above with bounded cohomology over $\cp^{\leq1}(\Lambda^\imath)$. Using Lemma~ \ref{lem: resolution}, similar to \cite[Lemma 4.1]{Ke1} (see also \cite[Proposition 5.6]{LP}), one can prove that for any bounded complex $X^\bullet\in \cc^b(\mod(\Lambda^\imath))$, there exists $P^\bullet\in\cc^{-,b}(\mod(\cp^{\leq1}(\Lambda^\imath)))$ and an epimorphism $P^\bullet \longrightarrow X^\bullet$ which is a quasi-isomorphism.

Let $D^{-,b}(\cp^{\leq1}(\Lambda^\imath))$ be the derived category of $\cc^{-,b}(\cp^{\leq1}(\Lambda^\imath))$. Then we have
$D^b(\mod(\Lambda^\imath))\simeq D^{-,b}(\cp^{\leq1}(\Lambda^\imath))$, and we shall identify them.
The remaining part of the proof is the same as in \cite[Lemma~ 4.3]{Ha3} by using Lemma~ \ref{lem: resolution}, and will be omitted here.
\end{proof}

\begin{corollary}
\label{cor: triangulated equiv}
The restriction of $G$ to $D^b(\mod(\K Q))$ gives a triangulated equivalence
\[D^b(\mod(\K Q))\stackrel{\simeq}{\longrightarrow} D_{sg}(\fdmz(\Lambda^\imath)).\]
\end{corollary}

\begin{proof}
By Lemma \ref{lem: dense}, it suffices to check that all graded $\Lambda^\imath$-modules are in $G(D^b(\mod(\K Q)))$.
Similar to \cite[Lemma 3.2]{LZ}, we only need to check that $S(i)\in G(D^b(\mod(\K Q)))$ for any simple $\K Q$-module $S$ and $i\in\Z$.
From \eqref{eqn: shift}, it is equivalent to that $\Sigma^i S\in G(D^b(\mod(\K Q)))$ for any $i\in\Z$, which is clear.
\end{proof}

\begin{theorem}\label{thm:sigma}
Let $(Q, \btau)$ be an $\imath$quiver. Then $D^b(\mod(\K Q))/\Sigma \circ \widehat{\btau}$ is a triangulated orbit category \`a la Keller \cite{Ke2}, and we have the following triangulated equivalence
\[
D_{sg}(\mod(\Lambda^{\imath}))\simeq D^b(\mod(\K Q))/\Sigma \circ \widehat{\btau}.
\]
\end{theorem}
\begin{proof}
The proof is the same as for \cite[Lemma 3.17, Theorem 3.18]{LW19a} by using now Corollary~ \ref{cor: triangulated equiv} and \eqref{eqn: shift}.
\end{proof}


\begin{corollary} (cf. \cite[Corollary 3.21]{LW19a})
   \label{corollary for stalk complexes}
For any $M\in D_{sg}(\mod(\Lambda^{\imath}))$, there exists a unique (up to isomorphisms) module $N\in \mod(\K Q)\subseteq \mod(\Lambda^{\imath})$ such that
$M\cong N$ in $D_{sg}(\mod(\Lambda^{\imath}))$.
\end{corollary}

\begin{proof}
The proof is the same as for \cite[Corollary 3.21]{LW19a}.
\end{proof}

\begin{corollary}
\label{cor: res proj}
For any $M\in\mod(\Lambda^{\imath})$ the following are equivalent.
\begin{itemize}
\item[(i)] $\pd M<\infty$;
\item[(ii)] $\ind M<\infty$;
\item[(iii)] $\pd M\leq1$;
\item[(iv)] $\ind M\leq1$;
\item[(v)] $\res_\BH (M)$ is projective as an $\BH $-module.
\end{itemize}
\end{corollary}

\begin{proof}
Proposition \ref{proposition of 1-Gorenstein} states that $\Lambda^{\imath}$ is $1$-Gorenstein, and then the equivalence of (i)--(iv) follows by \cite[Theorem 9.1.10]{EJ}.

(i)$\Rightarrow$(v) follows from Lemma \ref{lem: res proj}, so it remains to prove (v)$\Rightarrow$(i).

Assume $\res_\BH (M)$ is projective as $\BH $-module.
By Corollary \ref{corollary for stalk complexes}, there exists $N\in \mod(\K Q)$ such that $M\cong N$ in $D_{sg}(\mod(\Lambda^{\imath}))$.
Together with Lemma \ref{lem: iso in singularity}, we have
\begin{align*}
0\longrightarrow U_1\longrightarrow Z \longrightarrow M\longrightarrow0,
\qquad 0\longrightarrow U_2\longrightarrow Z\longrightarrow N\longrightarrow0
\end{align*}
in $\mod(\Lambda^\imath)$ with $U_1,U_2\in\cp^{<\infty}(\Lambda^\imath)$.

By applying $\res_\BH$ to the first short exact sequence, since $\res_\BH(U_1)$ is projective, so is $\res_\BH(Z)$. Then $\res_\BH(N)$ is projective by applying $\res_\BH$ to
the second one, which implies that $N=0$. So $\pd_{\Lambda^\imath} M<\infty$.
\end{proof}






\section{The $\imath$Hall algebras}
  \label{sec:Hall}


In this section, we take the field $\K=\mathbb F_q$, a finite field of $q$ elements. We formulate the $\imath$Hall algebra $\tMHk$ as a twisted semi-derived Hall algebra for the $\imath$quiver algebra $\iLa$ and study its properties.

\subsection{Euler forms}

As in \cite[(A.1)-(A.2)]{Lu19}, we define the Euler form $\langle\cdot,\cdot\rangle =\langle\cdot,\cdot\rangle_{\Lambda^\imath}$ for $\Lambda^\imath$:
\begin{align}
\label{eq:Euler1}
\langle\cdot,\cdot\rangle: K_0(\cp^{\leq 1}(\Lambda^\imath))\times K_0(\mod(\Lambda^\imath))\longrightarrow \Z,
\\
\label{eq:Euler2}
\langle\cdot,\cdot\rangle: K_0(\mod(\Lambda^\imath))\times K_0(\cp^{\leq 1}(\Lambda^\imath))\longrightarrow \Z.
\end{align}

Denote by $\langle\cdot,\cdot\rangle_Q$ the Euler form of $\K Q$. Denote by $S_i$ the simple $\K Q$-module (respectively, $\Lambda^{\imath}$-module) corresponding to vertex $i\in Q_0$ (respectively, $i\in\ov{Q}_0$).

\begin{lemma}
\label{lem:Euler}
For $K,K'\in \cp^{\leq1}(\Lambda^\imath)$, $M\in\mod(\Lambda^\imath)$, $i,j\in\I$, we have
\begin{align}
\label{Eform1}
\langle K,M\rangle =\langle \res_{\BH}(K),M\rangle,& \qquad \langle M,K\rangle =\langle M,\res_{\BH}(K)\rangle,
\\
\label{Eform2}
\langle \bK_i,S_j\rangle=\langle S_i,S_j\rangle_Q,&\qquad \langle S_j,\bK_i\rangle =\langle S_j, S_{\btau i} \rangle_Q,
\\
\label{Eform3}
\langle K,K'\rangle=&\frac{1}{2}\langle \res(K),\res(K')\rangle_Q.
\end{align}
\end{lemma}

\begin{proof}
The proof of \eqref{Eform2}--\eqref{Eform3} is the same as for \cite[Lemma 4.3]{LW19a}.

It remains to prove \eqref{Eform1}. Since $K_0(\mod(\Lambda^\imath))=\langle S_i\mid i\in\I\rangle\cong\Z\I$, without loss of generality, we assume $M=S_i$ for some $i\in\I$.
For any $K= \big(K_i, K(\alpha),K(\varepsilon_i) \big)_{i\in Q_0,\alpha\in Q_1} \in \mod(\Lambda^\imath)$, define a $\Lambda^\imath$-module
\[
\Phi(K):= \big(K_i,K(\alpha),-K(\varepsilon_i) \big)_{i\in Q_0,\alpha\in Q_1},
\]
which lies in $\mod(\Lambda^\imath)$. This defines an involution $\Phi$ of $\mod(\Lambda^\imath)$.

For any $K\in\cp^{\leq1}(\Lambda^\imath)$, we have $\Phi(K),\pi_*\pi^*(K)\in\cp^{\leq1}(\Lambda^\imath)$ by Corollary \ref{cor: res proj}. Note that $\pi_*\pi^*(K)_i= K_i\oplus K_i$ for $i\in\I$.
We have the following short exact sequence
$$0\longrightarrow K\xrightarrow{(f_i)_{i\in\I}} \pi_*\pi^*(K)\xrightarrow{(g_i)_{i\in\I}} \Phi(K) \longrightarrow0,$$
where $f_i=(\Id_{K_i},\Id_{K_i})^t$ and $g_i=(\Id_{K_i},-\Id_{K_i})$.


Clearly $\Phi$ preserves the Euler form \eqref{eq:Euler1}--\eqref{eq:Euler2}. Since $\Phi(S_i)=S_i$ for any $i\in\I$, it follows that
\begin{align}
  \label{eq:KS}
\langle K,S_i\rangle= \langle \Phi(K),S_i\rangle,\qquad \langle S_i,K\rangle= \langle S_i,\Phi(K)\rangle.
\end{align}

By the proof of \cite[Proposition 2.3]{LP}, we have $\widehat{\pi^*(K)}\in K_0(\cp^{\leq1}(\Lambda))=\langle \widehat{\bK}_i,\widehat{\bK}_{i'}\mid i\in\I\rangle$. Since $\pi_*$ preserves the exactness, we have $\widehat{\pi_*\pi^*}(K)\in \langle \widehat{\bK}_i \mid i\in\I\rangle\subseteq K_0(\cp^{\leq1}(\Lambda^\imath))$. So one can show that \eqref{Eform1} with $K$ replaced by $\pi_*\pi^*(K)$ holds. Then \eqref{Eform1} follows from this, using \eqref{eq:KS} and the fact that $\res_\BH(K)\cong \res_\BH(\Phi(K))$.
\end{proof}

\subsection{Semi-derived Hall algebras for $\imath$quiver algebras}

We shall follow \cite{Lu19} with some suitable modification to define semi-derived Hall algebra of $\Lambda^\imath$ for an arbitrary $\imath$quiver $(Q,\btau)$. Let
\[
\sqq=\sqrt{q}.
\]
Let $\ch(\Lambda^\imath)$ be the Ringel-Hall algebra of $\mod(\Lambda^\imath)$ over $\Q(\sqq)$.
Define $J$ to be the linear subspace of $\ch(\Lambda^\imath)$ spanned by
\begin{align}
  \label{eq:ideal}
&\{[K]-[K'] \mid \res_\BH(K)\cong\res_\BH(K'),  K,K'\in\cp^{\leq1}(\Lambda^\imath)\} \bigcup
\\\notag
&\{[L]-[K\oplus M]\mid \exists \text{ exact sequence } 0 \longrightarrow K \longrightarrow L \longrightarrow M \longrightarrow 0, K\in\cp^{\leq1}(\Lambda^\imath)\}.
\end{align}
Let $I$ be the two-sided ideal of $\ch(\Lambda^\imath)$ generated by $J$.


Consider the following subset of $\ch(\ca)/I$:
\begin{equation}
  \label{eq:Sca}
\cs_{\Lambda^\imath} := \{ a[K]  \mid a\in \Q^\times, K\in \cp^{\leq1}(\ca)\}.
\end{equation}
For any $\imath$quiver $(Q,\btau)$, $\mod(\Lambda^\imath)$ satisfies \cite[\S A.2, (E.a)--(E.d)]{Lu19}, where (E.d) holds thanks to Lemma~ \ref{lem: resolution}. Thus, we can define the semi-derived Ringel-Hall algebra of $\iLa$ as
\[
\utMH:=(\ch(\Lambda^\imath)/I)[\cs_{\Lambda^\imath}^{-1}].
\]

The quantum torus $\ct(\Lambda^\imath)$ is defined to be $(\ch(\cp^{\leq1}(\Lambda^\imath))/I_{ac})[\cs_{\Lambda^\imath}^{-1}]$, where $I_{ac}$ is the ideal generated by
$$\{[K]-[K'] \mid \res_\BH(K)\cong\res_\BH(K'),  K,K'\in\cp^{\leq1}(\Lambda^\imath)\}.$$
Then
\begin{align}
\ct(\Lambda^\imath)=\langle [\bK_i]\mid i\in\I\rangle.
\end{align}

For any $\alpha=\sum\limits_{i\in\I} a_i\widehat{S_i}\in K_0(\mod(\K Q))=\Z\I$, define
\[
\bK_\alpha:=q^{-\langle \widehat{X}-\widehat{Y},\widehat{Y}\rangle} [X]\diamond [Y]^{-1} \in \utMH,
\]
where $X= \bigoplus\limits_{i\in\I:a_i\geq0} \bK_i^{\oplus a_i}$ and $Y=\bigoplus\limits_{i\in\I:a_i<0} \bK_i^{\oplus (-a_i)}$.
In this way, we have $\ct(\Lambda^\imath)=\{\bK_\alpha\mid \alpha\in \Z\I\}$.

\begin{lemma}
\label{lem: Groth 1}
$\{\bK_\alpha\mid \alpha\in\Z\I\}$ forms a basis of $\ct(\Lambda^\imath)$.
\end{lemma}

\begin{proof}

Consider the group $\ov{K}_0(\cp^{\leq1}(\Lambda^\imath)):=K_0(\cp^{\leq1}(\Lambda^\imath))/ (\widehat{K}-\widehat{K'} \mid \res_{\BH}(K)= \res_{\BH}(K'))$. Clearly we have
\begin{align}
\label{eq:quotientK}
\ov{K}_0(\cp^{\leq1}(\Lambda^\imath))=\{\widehat{\bK}_\alpha \mid \alpha\in\Z\I \}.
\end{align}
For any $0\longrightarrow K\longrightarrow K'\longrightarrow K''\longrightarrow0$ in $\cp^{\leq}(\Lambda^\imath)$, we have $\res_\BH(K')=\res_\BH(K\oplus K'')$.
So $\ct(\Lambda^\imath)$ is the group algebra of $\ov{K}_0(\cp^{\leq1}(\Lambda^\imath))$ over $\Q(\sqq)$ with its multiplication twisted by $q^{-\langle \cdot,\cdot \rangle}$.
By  Corollary \ref{cor: res proj}, there is a morphism $\ov{K}_0(\cp^{\leq1}(\Lambda^\imath))\longrightarrow K_0(\proj(\BH))=\Z\I$ induced by $K\mapsto \res_\BH(K)$, which is surjective. Together with \eqref{eq:quotientK}, we have $\ov{K}_0(\cp^{\leq1}(\Lambda^\imath))\cong K_0(\proj(\BH))=\Z\I$, which is a free abelian group. So $\{\bK_\alpha\mid \alpha\in\Z\I\}$ is a basis of $\ct(\Lambda^\imath)$.
\end{proof}

The following Hall multiplication endows $\utMH$ a $\ct(\Lambda^\imath)$-bimodule structure:
\begin{eqnarray}
[M]\diamond [K]=q^{-\langle M,K\rangle} [M\oplus K]
   \label{right module structure},
\quad
[K]\diamond[M]&=&q^{-\langle K,M\rangle} [K\oplus M]
\end{eqnarray}
for any $K\in\cp^{\leq1}(\Lambda^\imath)$, $M\in\mod(\Lambda^\imath)$.





With the help of \eqref{Eform1}, similar to \cite[\S A.3]{Lu19}, we can define a $\ct(\Lambda^\imath)$-bimodule $\cm(\Lambda^\imath):=\ct(\Lambda^\imath)\otimes_{\ch(\cp^{\leq1}(\Lambda^\imath))/I_{ac}} \big(\ch(\Lambda^\imath)/J\big)\otimes_{\ch(\cp^{\leq1}(\Lambda^\imath))/I_{ac}} \ct(\Lambda^\imath)$ via the action given by \eqref{right module structure}. The proof of \cite[Lemmas~ A.11-A.12]{Lu19} proceeds in the same way with the help of \eqref{resolution 1}. Therefore, $\utMH$ is isomorphic to $\cm(\Lambda^\imath)$ as $\ct( \Lambda^\imath)$-bimodules by \cite[Proposition A.13]{Lu19}.

\subsection{An $\imath$Hall basis}
  \label{subsec:Hall basis}

In this subsection, we shall construct a Hall basis for $\utMH$ via a new approach; compare \cite[Lemma A.17]{Lu19}. For \cite[Lemma A.17]{Lu19} (in the setting of a finite-dimensional 1-Gorenstein algebra), we argue that $\utMH$ is isomorphic to the semi-derived Hall algebra $\cs\cd\ch(\Gproj(\Lambda^\imath))$ of $\Gproj(\Lambda^\imath)$ defined in \cite{Gor1}, and then a basis of $\cs\cd\ch(\Gproj(\Lambda^\imath))$ gives rise to a Hall basis of $\utMH$. For arbitrary (non-acyclic) $\imath$quiver, Gorenstein projective $\Lambda^\imath$-modules may be infinite-dimensional, and then its semi-derived Hall algebra is not well defined.

For $[X]\in\Iso(\mod(\K Q))\subseteq \Iso(\mod(\Lambda^\imath))$, by Corollary~ \ref{corollary for stalk complexes}, we define $\ch(\Lambda^\imath)_{[X]}$ to be the subspace of $\ch(\Lambda^\imath)$ spanned by
$\{[M]\in\Iso(\mod(\Lambda^\imath))\mid M\cong X \text{ in }D_{sg}(\mod(\Lambda^\imath))\}$. One can decompose $\ch(\Lambda^\imath)$ into a direct sum
\begin{align*}
\ch(\Lambda^\imath)=\bigoplus_{[X]\in \Iso(\mod(\K Q))} \ch(\Lambda^\imath)_{[X]}.
\end{align*}
Then $\ch(\Lambda^\imath)$ is an $\Iso(\mod(\K Q))$-graded vector space.

For a short exact sequence $0\longrightarrow K\longrightarrow L\longrightarrow M \longrightarrow0$ in $\mod(\Lambda^\imath)$ with $K$ of finite projective dimension, we have $L\cong M\cong K\oplus M$ in $D_{sg}(\mod(\Lambda^\imath))$. It follows from \eqref{eq:ideal} that $\ch(\Lambda^\imath)/J$, and then $\cm(\Lambda^\imath)$, are $\Iso(\mod(\K Q))$-graded vector spaces, that is,
\begin{align*}
\cm(\Lambda^\imath)=\bigoplus_{[X]\in \Iso(\mod(\K Q))} \cm(\Lambda^\imath)_{[X]}.
\end{align*}

\begin{lemma}
\label{lem: basis 1}
We have
$\cm(\Lambda^\imath)_{[X]}=[X]\diamond \ct(\Lambda^\imath)$
 for any $[X]\in \Iso(\mod(\K Q))\subseteq \Iso(\mod(\Lambda^\imath))$.
\end{lemma}

\begin{proof}
For any $M\in\mod(\Lambda^\imath)$ such that $M\cong X$ in $D_{sg}(\mod(\Lambda^\imath))$, by Lemma \ref{lem: iso in singularity} we have the following short exact sequences
\begin{align*}
0\longrightarrow U_1\longrightarrow Z\longrightarrow X\longrightarrow0,\qquad 0\longrightarrow U_2\longrightarrow Z\longrightarrow M\longrightarrow0
\end{align*}
with $U_1,U_2\in\cp^{\leq1}(\Lambda^\imath)$.
Then $[X]=q^{-\langle X,U_1\rangle}[Z]\diamond [U_1]^{-1}$, and $[M]= q^{-\langle M,U_2\rangle}[Z]\diamond [U_2]^{-1}$ in $\cm(\Lambda^\imath)$.
Therefore,
\begin{align*}
[M]=& q^{-\langle M,U_2\rangle}[Z]\diamond [U_2]^{-1}  \\
=&q^{-\langle M,U_2\rangle+\langle X,U_1\rangle}[X]\diamond [U_1]\diamond [U_2]^{-1} \in[X]\diamond \ct(\Lambda^\imath).
\end{align*}
The lemma is proved.
\end{proof}

It is well known that $K_0(\mod(\K Q))\cong K_0(\mod(\Lambda^\imath))\cong\Z\I$ are free abelian groups with a basis $\{\widehat{S}_i\mid i\in \I\}$. For any $M=(M_i, M(\alpha),M(\varepsilon_i))_{i\in\I,\alpha\in Q_1}$ in $\mod(\Lambda^\imath)$, we denote
\begin{align}
\widehat{\Im}(M(\varepsilon))=\sum_{i\in\I} \dim_\K(\Im(\varepsilon_i))\widehat{S}_i\in K_0(\mod(\K Q)).
\end{align}

\begin{lemma}
\label{lem: grade 1}
For any short exact sequence $0\longrightarrow K\longrightarrow L\longrightarrow M\longrightarrow0$ in $\mod(\Lambda^\imath)$ with $K$ of finite projective dimension, we have
$\widehat{\Im}(L(\varepsilon))=\widehat{\Im}(K(\varepsilon))+\widehat{\Im}(M(\varepsilon))$.
\end{lemma}

\begin{proof}
It suffices to show that $\dim_\K \Im(L(\varepsilon_i))= \dim_\K \Im(K(\varepsilon_i))+ \dim_\K \Im(M(\varepsilon_i))$ for any $i\in \I$. It is equivalent to consider it in ${\rm mod}(\BH_i)$. 
For $i\neq \btau i$, it follows from \cite[Lemma~ 3.12]{LP} by using Corollary \ref{cor: res proj}. A similar proof for $i=\btau i$ will be omitted here.
\end{proof}

Consider the following set
\begin{align}
\mathbb G := \big\{(\alpha,[X])\mid \alpha\in K_0(\mod(\K Q)), [X]\in \Iso(\mod(\K Q)) \big\}.
\end{align}
Then $\ch(\Lambda^\imath)$ is a $\mathbb G$-graded vector space, that is,
\begin{align}
\label{eqn: grade 1}
\ch(\Lambda^\imath) =\bigoplus_{(\alpha,[X])\in \mathbb G} \Big(\bigoplus_{\stackrel{\widehat{ \Im  }(M(\varepsilon))=\alpha}{
M\cong X \text{ in }D_{sg}(\mod(\Lambda^\imath))}} \Q(\sqq)[M]\Big).
\end{align}

\begin{lemma}
\label{lem: grade 2}
$\cm(\Lambda^\imath)$ is a $\mathbb G$-graded vector space with grading induced by \eqref{eqn: grade 1}.
\end{lemma}

\begin{proof}
The proof is the same as for \cite[Lemma 3.13]{LP} with the help of Lemma \ref{lem: grade 1}, and hence omitted here.
\end{proof}

\begin{theorem}[$\imath$Hall basis]
\label{thm:utMHbasis}
Let $(Q,\btau)$ be an $\imath$quiver. Then
\begin{equation}
  \label{eq:Hall basis}
\big\{ [X]\diamond \bK_\alpha ~\big |~ [X]\in\Iso(\mod(\K Q))\subseteq \Iso(\mod(\Lambda^{\imath})), \alpha\in \Z\I \big\}
\end{equation}
is a basis of $\utMH$.
\end{theorem}

\begin{proof}
Our proof here is inspired by that of \cite[Theorem 3.7]{Gor2}.

By Lemma \ref{lem: basis 1}, we have the following surjective morphism
\begin{align*}
\ct(\Lambda^\imath)\longrightarrow \cm(\Lambda^\imath)_{[X]}=[X]\diamond \ct(\Lambda^\imath),\qquad [K]\mapsto [X]\diamond [K].
\end{align*}

Let $K_0^{{\rm split}}(\BH)$ be the split Grothendieck group of ${\rm mod}(\BH)$. Then we have the following composition of natural maps
$$\zeta:  \ct(\Lambda^\imath)\longrightarrow \cm(\Lambda^\imath)_{[X]}\longrightarrow \cm(\Lambda^\imath)\stackrel{\xi}{\longrightarrow} \Q(\sqq)[K_0^{{\rm split}}(\BH)]$$
where $\xi$ maps $M$ to $\res_\BH(M)$. Note that $\xi$ is well defined. Indeed, applying $\res_{\BH}$ to a short exact sequence
$0\longrightarrow K\longrightarrow L\longrightarrow M\longrightarrow0$ make it split in ${\rm mod}(\BH)$ since $\res_\BH(K)$ is injective by Corollary \ref{cor: res proj}.

We claim that $\zeta$ is injective. Indeed, any $M\in {\rm mod}(\BH)$ can be decomposed in a unique way (up to a permutation of factors) into a direct sum of indecomposables:
$M=\bigoplus_{i\in\I}
(S_i^{\oplus n_i}\oplus \bK_i^{\oplus m_i})$, for some $m_i,n_i\in \N$. Then the linear map
$$\zeta': \Q(\sqq)[K_0^{{\rm split}}(\BH)]\longrightarrow \ct(\Lambda^\imath),\quad [M]\mapsto [\bigoplus_{i\in\I}\bK_i^{\oplus m_i}]$$
is well defined. Note that $\ct(\Lambda^\imath)=\langle [\bK_i]\mid i\in\I \rangle$. Then $\zeta'\circ \zeta=\Id$.
So $\zeta$ is injective.

It follows that the map
\begin{align}
\label{eqn: isomor}
\ct(\Lambda^\imath) \longrightarrow  \cm(\Lambda^\imath)_{[X]}=[X]\diamond \ct(\Lambda^\imath),\quad [K] \mapsto [X]\diamond [K]
\end{align}
is an isomorphism.

Assume that
$$\sum_{\alpha\in \Z\I,
[X]\in \Iso(\mod(\K Q))
} a_{X,\alpha} [X]\diamond \bK_\alpha=0$$
in $\cm(\Lambda^\imath)$, where $a_{X,\alpha}\in\Q(\sqq)$. It follows from \eqref{eqn: grade 1} that
$\sum\limits_{\alpha\in \Z\I} a_{X,\alpha} [X]\diamond\bK_\alpha=0$ for any $[X]\in\Iso(\mod(\K Q))$ in $\cm(\Lambda^\imath)$. Together with \eqref{eqn: isomor}, we have $\sum_{\alpha\in \Z\I} a_{X,\alpha} \bK_\alpha=0$ in $\ct(\Lambda^\imath)$, and then
$a_{X,\alpha}=0$ by Lemma \ref{lem: Groth 1}.
 So \eqref{eq:Hall basis} is a basis of $\cm(\Lambda^\imath)$.

The lemma follows since $\utMH$ is isomorphic to $\cm(\Lambda^\imath)$ as $\ct( \Lambda^\imath)$-bimodules.
\end{proof}

\begin{remark}
Using the ideal $I$ in \cite{Lu19} to define $\utMH$, one can prove that $\utMH$ is a (left) free $\ct(\Lambda^\imath)$-module with a basis given by
$\{[X]\mid X\in \mod(\K Q)\subseteq \mod(\Lambda^\imath)\}$. However, it is not clear if $\ct(\Lambda^\imath)$ defined there is generated by $\{\bK_i\mid i\in\I\}$. 
\end{remark}

\begin{remark}
  \label{rem:MH}
For an infinite-dimensional finitely generated 1-Gorenstein algebra $A$, we can still define the semi-derived Ringel-Hall algebra $\cs\cd\ch (A)$ using the ideal $I$ in \cite{Lu19}. Then $\cs\cd\ch(A)$ is a (left) free $\ct(A)$-module with $\{[X] \in \text{Iso}(D_{sg} (\mod (A))) \mid X\in \mod(A)\}$ as its basis. However, we have chosen to focus on $\utMH$ and making it more explicit in this section.
\end{remark}

Via the restriction functor $\res: \mod(\Lambda^{\imath})\rightarrow\mod (\K Q)$, we define the twisted semi-derived Ringel-Hall algebra to be the $\Q(\sqq)$-algebra on the same vector space as $\utMH$ with twisted multiplication given by
\begin{align}
   \label{eqn:twsited multiplication}
[M]* [N] =\sqq^{\langle \res(M),\res(N)\rangle_Q} [M]\diamond[N].
\end{align}
We shall denote this algebra $(\utMH, *)$ by 
 $\tMHk$, and call it the {\em Hall algebra associated to the $\imath$quiver $(Q, \btau)$}, (or an {\em $\imath$Hall algebra}, for short). The {\em twisted quantum torus} $\tTL$ is defined to be the subalgebra of $\tMHk$ generated by $\bK_\alpha$, $\alpha\in \Z\I$.
By Lemma~ \ref{lem: Groth 1},  $\tTL$ is a Laurent polynomial algebra generated by $[\bK_i]$, for $i\in \I$; and $[\bK_\alpha]*[\bK_\beta]=[\bK_{\alpha+\beta}]$ for any $\alpha,\beta\in\Z\I$.

\subsection{$\imath$Hall algebras for $\imath$subquivers}

Let $(Q, \btau)$ be an $\imath$quiver and $\Lambda^\imath$ be its $\imath$quiver algebra. Let $'Q$ be a full subquiver of $Q$ preserved by $\btau$. 
Hence we obtain an {\em $\imath$subquiver} $('Q, \btau)$ of $(Q, \btau)$, and denote by $'\Lambda^\imath$ the $\imath$quiver algebra of $('Q, \btau)$. Clearly, $'\Lambda^\imath$ is a quotient algebra (also a subalgebra) of $\Lambda^\imath$. Then we can view $\mod('\Lambda^\imath)$ as a full subcategory of $\mod(\Lambda^\imath)$.

\begin{lemma}\cite[Lemma 4.12]{LW19a}
\label{lem:subalgebra}
Retain the notation as above. Then 
$\widetilde{\ch}(\K\, {{}'Q}, \btau)$ is naturally a subalgebra of $\tMHk$, with the inclusion morphism induced by
$\mod('\Lambda^\imath)\subseteq \mod(\Lambda^\imath)$.
\end{lemma}

\subsection{An $\imath$Hall multiplication formula}

Let $\ce$ be an exact category. For any short exact sequence $0\rightarrow A \xrightarrow{f} B\xrightarrow{g}C\rightarrow0$ in $\ce$, we denote by $\ov{(f,g)}$ the corresponding element in $\Ext^1_\ce(C,A)$.
Below, we present a fairly general multiplication formula in the $\imath$Hall algebra $\tMHk$ with $\tau=\Id$.
In concrete situations (see Proposition~\ref{prop:SSSM} for example), the items appearing in RHS \eqref{Hallmult1} below are computable, and this makes Proposition~\ref{prop:iHallmult} useful and applicable.  

\begin{proposition}
\label{prop:iHallmult}
Let $(Q,\tau)$ be an $\imath$quiver with $\tau=\Id$.
For any $A,B\in\mod(\K Q)\subset \mod(\Lambda^\imath)$, we have
\begin{align}
\label{Hallmult1}
[A]*[B]=&
\sum_{[L],[M],[N]\in\Iso({\mod(\K Q)})} \sqq^{\langle A,B\rangle_Q}  q^{\langle N,B\rangle_Q-\langle N,A\rangle_Q +\langle N,N\rangle_Q -\langle A,B\rangle_Q}\frac{|\Ext^1(N, L)_{M}|}{|\Hom(N,L)| }
\\
\notag
&\cdot |\{s\in\Hom(A,B)\mid \Ker s\cong N, \coker s\cong L
\}|\cdot [M]*[\bK_{\widehat{A}-\widehat{N}}]
\end{align}
in $\tMHk$.
\end{proposition}

\begin{proof}
A $1$-periodic complex over $\mod(\K Q)$ is a pair $M^\bullet=(M,d)$ such that $M\in \mod(\K Q)$ and $d:M\rightarrow M$ is a morphism of $\K Q$-modules with $d^2=0$.
Let $\cc_1(\mod(\K Q))$ be the category of $1$-periodic complexes over $\mod(\K Q)$. It is well known that $\cc_1(\mod(\K Q))\simeq \mod(\Lambda^\imath)$, and we identify them in the following.

By definition,
\begin{align}
\label{eq:basic}
[A]* [B]=&\sum_{[C^\bullet]\in\Iso(\cc_1({\mod(\K Q)}))}\sqq^{\langle A,B\rangle_Q}\frac{|\Ext^1_{\cc_1({\mod(\K Q)})}(A,B)_{C^\bullet}|}{|\Hom_{\cc_1({\mod(\K Q)})}(A,B)|}[C^\bullet].
\end{align}
For any $C^\bullet=(C,d)\in\cc_1({\mod(\K Q)})$ such that $|\Ext^1_{\cc_1({\mod(\K Q)})}(A,B)_{C^\bullet}|\neq0$, denote by $M=H(C^\bullet)$ the homology group of $C^\bullet$, i.e., $\Ker d/\Im d$.
Then we have
$[C^\bullet]= [M\oplus \bK_{\Im d}]=[M]*[\bK_{\widehat{\Im d}}]$; see, e.g., \cite[Lemma 2.10]{LRW20}. 
Note that
$$\widehat{\Im d}=\frac{\widehat{A}+\widehat{B}-\widehat{M}}{2}\in K_0({\mod(\K Q)}).$$
Denote by $\mathbb{S}_{[M]}:=\{ [\xi]\in \Ext^1_{\Lambda^\imath}(A,B)_{C^\bullet} \mid H(C^\bullet)\cong M\}$. Then
we have
\begin{align}
\label{eq:basic1}
[A]*[B]=&\sum_{[M]\in\Iso({\mod(\K Q)})}\sqq^{\langle A,B\rangle_Q}\frac{|\mathbb{S}_{[M]}|}{|\Hom(A,B)|}[M]*[K_{\frac{\widehat{A}+\widehat{B}-\widehat{M}}{2}}].
\end{align}

Let $C^\bullet=(C,d)\in\cc_{1}({\mod(\K Q)})$ such that $|\Ext^1_{\Lambda^\imath}(A,B)_{C^\bullet}|\neq0$. Then we have the following short exact sequence
$$0\longrightarrow B\stackrel{f}{\longrightarrow} C\stackrel{g}{\longrightarrow} A\longrightarrow0$$
such that $df=0=gd$. Denote by $U=\Ker d$ and $V=\Im d$. Then we have the following short exact sequences
\begin{align*}
0\longrightarrow U\stackrel{l_2}{\longrightarrow} C\stackrel{d_1}{\longrightarrow} V\longrightarrow0,
\qquad
0\longrightarrow V\stackrel{p_1}{\longrightarrow} U \stackrel{p_2}{\longrightarrow} M\longrightarrow0.
\end{align*}
By definition, $d=l_2p_1d_1$. Then there exist the following two commutative diagrams, which are both push-outs and pull-backs

\begin{equation}
\label{eq:diagrams}
\xymatrix{ B \ar@{.>}[r]^{h_1} \ar@{=}[d] &  U \ar@{.>}[r]^{h_2} \ar[d]^{l_2} & N \ar@{.>}[d]^{t_2} \\
B  \ar[r]^f & C \ar[r]^g\ar[d]^{d_1}  & A\ar@{.>}[d]^{s_1}  \\
& V \ar@{=}[r] & V }
\qquad\qquad \xymatrix{ V \ar@{.>}[r]^{s_2} \ar@{=}[d] & B \ar@{.>}[r]^{s_3} \ar@{.>}[d]^{h_1} & L \ar@{.>}[d]^{t_0} \\
V \ar[r]^{p_1} & U \ar[r]^{p_2} \ar@{.>}[d]^{h_2} & M \ar@{.>}[d]^{t_1} \\
& N \ar@{=}[r]  & N   }
\end{equation}

Conversely, denote by
$$\G_{[L],[N]}:=\{s\in\Hom(A,B)\mid \Ker s\cong N, \coker s\cong L\}.$$
For any $s\in \G_{[L],[N]}$, denote by $V=\Im s$. Then there exist $s_1:A\rightarrow V$ and $s_2:V\rightarrow B$ such that $s=s_2s_1$. In fact, $s_1,s_2$ are unique up to a group action of $\Aut(V)$. Then we have two short exact sequences
\begin{align}
\label{eq:sesG}
0\longrightarrow N\stackrel{t_2}{\longrightarrow} A\stackrel{s_1}{\longrightarrow} V\longrightarrow0,\qquad 0\longrightarrow V\stackrel{s_2}{\longrightarrow} B\stackrel{s_3}{\longrightarrow} L\longrightarrow0.
\end{align}
 For any $s\in \G_{[L],[N]}$, denote by
 \begin{align}
 \mathbb{S}_{s,[M]}:=\{[\eta]\in \Ext^1(A,B)_C\mid \Ext^1_{\K Q}(N,s_3)\circ\Ext^1_{\K Q}(t_2,B) ([\eta]) \in \Ext^1(N,L)_M \}.
 \end{align}

From above, we define a map
\begin{align*}
\Xi: \mathbb{S}_{[M]}&\longrightarrow \bigsqcup_{s\in\G_{[L],[N]}} \mathbb{S}_{s,[M]}
\\
\overline{(f,g)}&\mapsto \overline{(f,g)}_{s_2s_1}\in\mathbb{S}_{s_2s_1,[M]}
\end{align*}
by using diagrams \eqref{eq:diagrams}.

\vspace{2mm}
{\bf Claim ($\star$).} $\Xi$ is a bijection.

First, we prove that $\Xi$ is surjective. For any $\overline{(f,g)}_{s}\in\mathbb{S}_{s,[M]}$ and $s\in \G_{[L],[N]}$, we have a short exact sequence $0\rightarrow B\stackrel{f}{\rightarrow} C\stackrel{g}{\rightarrow} A\rightarrow0$. Define
$C^\bullet=(C,d)$ where $d=fsg$. Then there exists a short exact sequence
$0\rightarrow B\stackrel{f}{\rightarrow} C^\bullet\stackrel{g}{\rightarrow} A\rightarrow0$ in $\cc_1({\mod(\K Q)})$. One can check $\Xi(\overline{(f,g)})= \overline{(f,g)}_{s}$ by definition.

Next, we prove that $\Xi$ is injective. Consider two short exact sequences $0\rightarrow B\stackrel{f}{\rightarrow} (C,d)\stackrel{g}{\rightarrow} A\rightarrow0$ and
$0\rightarrow B\stackrel{f'}{\rightarrow} (C',d')\stackrel{g'}{\rightarrow} A\rightarrow0$.
Assume that $\overline{(f,g)},\overline{(f',g')}\in \mathbb{S}_{[M]}$ with $\Xi(\overline{(f,g)})\in \mathbb{S}_{s,M}$ and $\Xi(\overline{f',g'})\in \mathbb{S}_{s',M}$.
In place of the notations in the diagrams \eqref{eq:diagrams} associated to $\overline{(f,g)}$, we shall use the corresponding prime notations for all objects and maps in the counterpart diagrams associated to $\overline{(f',g')}$.
If $\Xi(\overline{(f,g)})=\Xi(\overline{(f',g')})$, then $s=s'$. Without losing generality, we assume that $s_1=s_1'$ and $s_2=s_2'$ in the diagrams \eqref{eq:diagrams}.
Since $\overline{(f,g)}=\overline{(f',g')}$ in $\Ext^1_{\K Q}(A,B)$, there exits an isomorphism $\beta :C\rightarrow C'$ such that $\beta f=f'$ and $g'\beta=g$.
Note that $d=fsg$ and $d'=f'sg'$. We have $\beta d=\beta fsg= f'sg'\beta= d'\beta$. So $\beta:(C,d)\rightarrow (C',d')$ is an isomorphism in $\cc_1({\mod(\K Q)})$, and then there is a commutative diagram of short exact sequences in $\cc_1({\mod(\K Q)})$
\[\xymatrix{ 0\ar[r] & B \ar[r]^f \ar@{=}[d] &(C,d) \ar[r]^g \ar[d]^{\beta} &A\ar@{=}[d] \ar[r] &0
\\
0\ar[r] & B \ar[r]^{f'}  &(C',d') \ar[r]^{g'}  &A\ar[r] &0 }\]
Therefore,
$\overline{(f,g)}=\overline{(f',g')}\in \mathbb{S}_{[M]}$. The injectivity of $\Xi$ and hence the Claim ($\star$) is proved.

By Claim ($\star$), we have $|\mathbb{S}_{[M]}|= \sum_{s\in \G_{[L],[N]}} |\mathbb{S}_{s,[M]}|$.
For any $s\in \G_{[L],[N]}$, keep the notations as in \eqref{eq:sesG}. Then  we have two long exact sequences
\begin{align}
\label{longses1}
0\longrightarrow \Hom_{\K Q}(V,B)\longrightarrow \Hom_{\K Q}(A,B)\longrightarrow\Hom_{\K Q}(N,B) \longrightarrow \Ext^1_{\K Q}(V,B)\\\notag
\longrightarrow \Ext^1_{\K Q}(A,B)\stackrel{\varphi_1}\longrightarrow \Ext^1_{\K Q}(N,B)\longrightarrow0;
\\
\label{longses2}
0\longrightarrow \Hom_{\K Q}(N,V)\longrightarrow \Hom_{\K Q}(N,B)\longrightarrow\Hom_{\K Q}(N,L) \longrightarrow \Ext^1_{\K Q}(N,V)\\\notag
\longrightarrow \Ext^1_{\K Q}(N,B)\stackrel{\varphi_2}\longrightarrow \Ext^1_{\K Q}(N,L)\longrightarrow0,
\end{align}
where $\varphi_1=\Ext^1_{\K Q}(t_2,B)$ and $\varphi_2=\Ext^1_{\K Q}(N,s_3)$.
We have a map
$\varphi_2\circ\varphi_1: \mathbb{S}_{M,s}\rightarrow \Ext^1_{\K Q}(N,L)_M.$
By using \eqref{eq:diagrams}, one can easily see that $\varphi_2\circ\varphi_1$ is surjective, and
$$(\varphi_2\circ\varphi_1)^{-1}(\Ext^1_{\K Q}(N,L)_M)=\mathbb{S}_{s,[M]}.$$
Then it follows from \eqref{longses1}--\eqref{longses2} that
\begin{align*}
|\mathbb{S}_{s,[M]}|
=&|\varphi_1^{-1}\big(\varphi_2^{-1}(\Ext^1_{\K Q}(N,L)_M)\big)|
\\
=&|\varphi_2^{-1}(\Ext^1_{\K Q}(N,L)_M)|\cdot |\Ker \varphi_1|
\\
=&|\varphi_2^{-1}(\Ext^1_{\K Q}(N,L)_M)|\cdot \frac{|\Ext^1_{\K Q}(V,B)|\cdot |\Hom_{\K Q}(A,B)|}{|\Hom_{\K Q}(V,B)|\cdot |\Hom_{\K Q}(N,B)|}
\\
=&|\Ext^1_{\K Q}(N,L)_M|\cdot \frac{|\Ext^1_{\K Q}(N,V)|\cdot |\Hom_{\K Q}(N,B)|}{|\Hom_{\K Q}(N,V)|\cdot |\Hom_{\K Q}(N,L)|} \cdot \frac{|\Ext^1_{\K Q}(V,B)|\cdot |\Hom_{\K Q}(A,B)|}{|\Hom_{\K Q}(V,B)|\cdot |\Hom_{\K Q}(N,B)|}
\\
=&q^{-\langle N,V\rangle -\langle V,B\rangle}\cdot |\Hom_{\K Q}(A,B)| \frac{|\Ext^1_{\K Q}(N,L)_M|}{|\Hom_{\K Q}(N,L)| }
\\
=&q^{\langle N,B\rangle-\langle N,A\rangle +\langle N,N\rangle -\langle A,B\rangle}\cdot |\Hom_{\K Q}(A,B)| \frac{|\Ext^1_{\K Q}(N,L)_M|}{|\Hom_{\K Q}(N,L)| }.
\end{align*}
Note that $|\mathbb{S}_{s,[M]}|$ depends only on $[L],[N]$ for any $s\in\G_{[L],[N]}$. So we have
\begin{align}
\label{eq:SM}
|\mathbb{S}_{[M]}|=& \sum_{s\in \G_{[L],[N]}} |\mathbb{S}_{s,[M]}|
\\\notag
=& \sum_{[L],[N]\in\Iso({\mod(\K Q)})} q^{\langle N,B\rangle-\langle N,A\rangle +\langle N,N\rangle -\langle A,B\rangle}|\Hom_{\K Q}(A,B)| \frac{|\Ext^1_{\K Q}(N,L)_M|}{|\Hom_{\K Q}(N,L)| }  |\G_{[L],[N]}|.
\end{align}
Note that $\widehat{\Im d}=\widehat{V}=\widehat{A}-\widehat{N}$.
Then \eqref{Hallmult1} follows from \eqref{eq:basic1} and \eqref{eq:SM}.
\end{proof}

\section{Quantum symmetric pairs and $\imath$quantum groups}
  \label{sec:iQG}

In this section, we review and set up notations for quantum symmetric pairs $(\U, \Ui)$ and universal $\imath$quantum groups $\tUi$. 
We formulate a Serre presentation for $\tUi$.

\subsection{Quantum groups}
  \label{subsection Quantum groups}

Let $Q$ be a quiver (without loops) with vertex set $Q_0= \I$.
Let $n_{ij}$ be the number of edges connecting vertex $i$ and $j$. Let $C=(c_{ij})_{i,j \in \I}$ be the symmetric generalized Cartan matrix of the underlying graph of $Q$, defined by $c_{ij}=2\delta_{ij}-n_{ij}.$ Let $\fg$ be the corresponding Kac-Moody Lie algebra. Let $\alpha_i$ ($i\in\I $) be the simple roots of $\fg$.

Let $\bv$ be an indeterminant. Write $[A, B]=AB-BA$. Denote, for $r,m \in \N$,
\[
 [r]=\frac{\bv^r-\bv^{-r}}{\bv-\bv^{-1}},
 \quad
 [r]!=\prod_{i=1}^r [i], \quad \qbinom{m}{r} =\frac{[m][m-1]\ldots [m-r+1]}{[r]!}.
\]
Then $\tU := \tU_\bv(\fg)$ is defined to be the $\Q(\bv)$-algebra generated by $E_i,F_i, \tK_i,\tK_i'$, $i\in \I$, where $\tK_i, \tK_i'$ are invertible, subject to the following relations:
\begin{align}
[E_i,F_j]= \delta_{ij} \frac{\tK_i-\tK_i'}{\bv-\bv^{-1}},  &\qquad [\tK_i,\tK_j]=[\tK_i,\tK_j']  =[\tK_i',\tK_j']=0,
\label{eq:KK}
\\
\tK_i E_j=\bv^{c_{ij}} E_j \tK_i, & \qquad \tK_i F_j=\bv^{-c_{ij}} F_j \tK_i,
\label{eq:EK}
\\
\tK_i' E_j=\bv^{-c_{ij}} E_j \tK_i', & \qquad \tK_i' F_j=\bv^{c_{ij}} F_j \tK_i',
 \label{eq:K2}
\end{align}
 and the quantum Serre relations, for $i\neq j \in \I$,
\begin{align}
& \sum_{r=0}^{1-c_{ij}} (-1)^r  E_i^{(r)} E_j  E_i^{(1-c_{ij}-r)}=0,
  \label{eq:serre1} \\
& \sum_{r=0}^{1-c_{ij}} (-1)^r   F_i^{(r)} F_j  F_i^{(1-c_{ij}-r)}=0.
  \label{eq:serre2}
\end{align}
Here \[
F_i^{(n)} =F_i^n/[n]!, \quad E_i^{(n)} =E_i^n/[n]!, \quad \text{ for } n\ge 1, i\in \I.
\]
Note that $\tK_i \tK_i'$ are central in $\tU$ for all $i$.
The comultiplication $\Delta: \widetilde{\U} \longrightarrow \widetilde{\U} \otimes \widetilde{\U}$ is defined as follows:
\begin{align}  \label{eq:Delta}
\begin{split}
\Delta(E_i)  = E_i \otimes 1 + \tK_i \otimes E_i, & \quad \Delta(F_i) = 1 \otimes F_i + F_i \otimes \tK_{i}', \\
 \Delta(\tK_{i}) = \tK_{i} \otimes \tK_{i}, & \quad \Delta(\tK_{i}') = \tK_{i}' \otimes \tK_{i}'.
 \end{split}
\end{align}
The Chevalley involution $\omega$ on $\tU$ is given by
\begin{align}  \label{eq:omega}
\omega(E_i)  = F_i,\quad  \omega(F_i) = E_i,\quad \omega(\tK_{i}) = \tK_{i}' , \quad \omega(\tK_{i}') =\tK_{i}, \quad \forall i\in \I.
\end{align}

Analogously as for $\tU$, the quantum group $\bU$ is defined to be the $\Q(v)$-algebra generated by $E_i,F_i, K_i, K_i^{-1}$, $i\in \I$, subject to the  relations modified from \eqref{eq:KK}--\eqref{eq:serre2} with $\tK_i$ and $\tK_i'$ replaced by $K_i$ and $K_i^{-1}$, respectively. The comultiplication $\Delta$ and Chevalley involution $\omega$ on $\U$ are obtained by modifying \eqref{eq:Delta}--\eqref{eq:omega} with $\tK_i$ and $\tK_i'$ replaced by $K_i$ and $K_i^{-1}$, respectively (cf. \cite{L93}; beware that our $K_i$ has a different meaning from $K_i \in \U$ therein.)

The algebra $\U$ is isomorphic to a quotient algebra of $\tU$ by the ideal $( \tK_i \tK_i'- 1 \mid \forall i\in \I )$.

Let $\widetilde{\bU}^+$ be the subalgebra of $\widetilde{\bU}$ generated by $E_i$ $(i\in \I)$, $\widetilde{\bU}^0$ be the subalgebra of $\widetilde{\bU}$ generated by $\tK_i, \tK_i'$ $(i\in \I)$, and $\widetilde{\bU}^-$ be the subalgebra of $\widetilde{\bU}$ generated by $F_i$ $(i\in \I)$, respectively.
The subalgebras $\bU^+$, $\bU^0$ and $\bU^-$ of $\bU$ are defined similarly. Then both $\widetilde{\bU}$ and $\bU$ have triangular decompositions:
$
\widetilde{\bU} =\widetilde{\bU}^+\otimes \widetilde{\bU}^0\otimes\widetilde{\bU}^-,
\, 
\bU =\bU^+\otimes \bU^0\otimes\bU^-.
$ 
Clearly, ${\bU}^+\cong\widetilde{\bU}^+$, ${\bU}^-\cong \widetilde{\bU}^-$, and ${\bU}^0 \cong \widetilde{\bU}^0/(\tK_i \tK_i' -1 \mid   i\in \I)$.

\subsection{The $\imath$quantum groups $\Ui$ and $\tUi$}
  \label{subsec:iQG}

For a generalized Cartan matrix $C=(c_{ij})$, let $\Aut(C)$ be the group of all permutations $\btau$ of the set $\I$ such that $c_{ij}=c_{\btau i,\btau j}$. An element $\btau\in\Aut(C)$ is called an \emph{involution} if $\btau^2=\Id$.

Let $\btau$ be an involution in $\Aut(C)$. We define $\widetilde{\bU}^\imath$ 
 to be the $\Q(v)$-subalgebra of $\tU$ generated by
\[
B_i= F_i +  E_{\btau i} \tK_i',
\qquad \tk_i = \tK_i \tK_{\btau i}', \quad \forall i \in \I.
\]
Let $\tU^{\imath 0}$ be the $\Q(v)$-subalgebra of $\tUi$ generated by $\tk_i$, for $i\in \I$.
The elements
\begin{align}
\label{eq:central}
\tk_i\; (i= \btau i)
\qquad
\tk_i \tk_{\btau i}\;  (i\neq \btau i)
\end{align}
are central in $\tUi$.



Let $\bvs=(\vs_i)\in  (\Q(\bv)^\times)^{\I}$ be such that $\vs_i=\vs_{\btau i}$ for each $i\in \I$ which satisfies $c_{i, \btau i}=0$.
Let $\Ui:=\Ui_{\bvs}$ be the $\Q(v)$-subalgebra of $\bU$ generated by
\[
B_i= F_i+\vs_i E_{\btau i}K_i^{-1},
\quad
k_j= K_jK_{\btau j}^{-1},
\qquad  \forall i \in \I, j \in \ci.
\]
It is known \cite{Let99, Ko14} that $\bU^\imath$ is a right coideal subalgebra of $\bU$, i.e., $\Delta (\Ui) \subset \Ui\otimes \U$; and $(\bU,\Ui)$ is called a quasi-split \emph{quantum symmetric pair}, as they specialize at $v=1$ to $(U(\fg), U(\fg^\theta))$, where $\theta=\omega \circ \btau$, and $\btau$ is understood here as an automorphism of $\fg$.

We call $\Ui$ an $\imath$quantum group and $\tUi$ a universal $\imath$quantum group.
The algebras $\Ui_{\bvs}$, for $\bvs \in  (\Q(\bv)^\times)^{\I}$, are obtained from $\tUi$ by central reductions.

\begin{proposition}\cite[Propositon 6.2]{LW19a}
  \label{prop:QSP12}
(1) The $\Q(v)$-algebra $\Ui$ is isomorphic to the quotient of $\tUi$ by the ideal generated by
$
\tk_i - \vs_i \; (\text{for } i =\btau i)$ and
$\tk_i \tk_{\btau i} - \vs_i \vs_{\btau i}  \;(\text{for } i \neq \btau i).
$ 
The isomorphism is given by sending $B_i \mapsto B_i, k_j \mapsto \vs_{\btau j}^{-1} \tk_j, k_j^{-1} \mapsto \vs_{ j}^{-1} \tk_{\btau j}, \forall i\in \I, j\in \ci$.

(2) The algebra $\widetilde{\bU}^\imath$ is a right coideal subalgebra of $\widetilde{\bU}$. 
\end{proposition}

\subsection{A Serre presentation of $\Ui$}\label{subsection:iserre presentation}

For  $i\in \I$ with $\btau i\neq i$,  we define the {\em $\imath${}divided power} of $B_i$ as
\begin{align}
  \label{eq:iDP1}
  B_i^{(m)}:=B_i^{m}/[m]!, \quad \forall m\ge 0, \qquad (\text{if } i \neq \btau i).
\end{align}

For $i\in \I$ with $\btau i= i$, generalizing the constructions in \cite{BW18a, BeW18}, we define the {\em $\imath${}divided powers} of $B_i$ to be (see also \cite{CLW20})
\begin{eqnarray}
&&\ff_{i,\odd}^{(m)}=\frac{1}{[m]_{v}!}\left\{ \begin{array}{ccccc} B_i\prod_{s=1}^k (B_i^2-v\tk_i[2s-1]_{v}^2 ) & \text{if }m=2k+1,\\
\prod_{s=1}^k (B_i^2-v\tk_i[2s-1]_{v}^2) &\text{if }m=2k; \end{array}\right.
  \label{eq:iDPodd}\\
&&\ff_{i,\ev}^{(m)}= \frac{1}{[m]_{v}!}\left\{ \begin{array}{ccccc} B_i\prod_{s=1}^k (B_i^2-v\tk_i[2s]_{v}^2 ) & \text{if }m=2k+1,\\
\prod_{s=1}^{k} (B_i^2-v\tk_i[2s-2]_{v}^2) &\text{if }m=2k. \end{array}\right.
 \label{eq:iDPev}
\end{eqnarray}

Denote
\[
(a;x)_0=1, \qquad (a;x)_n =(1-a)(1-ax)  \cdots (1-ax^{n-1}), \quad   n\ge 1.
\]

The following theorem is an upgrade of (and can be derived from) \cite[Theorem~3.1]{CLW18} for $\Ui$ to the setting of a universal $\imath$quantum group $\tUi$; it generalizes \cite[Proposition~6.4]{LW19a} for $\tUi$ of ADE type.

\begin{theorem}
 \label{thm:Serre}
Fix $\ov{p}_i\in \Z/2\Z$ for each $i\in \I$. The $\Q(v)$-algebra $\tUi$ has a presentation with generators $B_i$, $\tk_i$ $(i\in \I)$ and the relations \eqref{relation1}--\eqref{relation6} below: for $\ell \in \I$, and $i\neq j \in \I$,
\begin{align}
\tk_i \tk_\ell =\tk_\ell \tk_i,
\quad
\tk_i B_\ell & = v^{c_{\btau i,\ell} -c_{i \ell}} B_\ell \tk_i,
   \label{relation1}
\\
B_iB_{j}-B_jB_i &=0, \quad \text{ if }c_{ij} =0 \text{ and }\btau i\neq j,\label{relation2}
\\
\sum_{n=0}^{1-c_{ij}} (-1)^nB_i^{(n)}B_jB_i^{(1-c_{ij}-n)} &=0, \quad \text{ if } j \neq \btau i\neq i, \label{relation3}
\\
\sum_{n=0}^{1-c_{i,\btau i}} (-1)^nB_i^{(n)}B_{\btau i}B_i^{(1-c_{i,\btau i}-n)}& =\frac{1}{v-v^{-1}} \times
\label{relation5}     \\
    \left(v^{c_{i,\btau i}} (v^{-2};v^{-2})_{-c_{i,\btau i}} \right.  &
  \left. B_i^{(-c_{i,\btau i})} \tk_i
   -(v^{2};v^{2})_{-c_{i,\btau i}}B_i^{(-c_{i,\tau i})} \tk_{\btau i}  \right),
\text{ if } \btau i \neq i,
 \notag \\
\sum_{n=0}^{1-c_{ij}} (-1)^n  B_{i, \overline{p_i}}^{(n)}B_j B_{i,\overline{c_{ij}}+\overline{p}_i}^{(1-c_{ij}-n)} &=0,\quad   \text{ if }\btau i=i.
\label{relation6}
\end{align}
(This presentation is called a {\em Serre presetation} of $\tUi$.)
\end{theorem}

\begin{proof}
Recall the main differences between $\Ui$ and $\tUi$ are as follows.  Let $\mathbb K$ be a field which contains parameters $\vs_i$, for $i\in \I_\tau$, such that $\vs_{\tau i} =\vs_i$ for all $i$. The $\mathbb K$-algebra $\Ui$ (cf.  \cite[Theorem~3.1]{CLW18}) does not contain the central elements \eqref{eq:central} as in $\tUi$; additionally, instead of $\tk_i$\; $(i\in\I)$ in $\tUi$, $\Ui$ contains generators $k_j\; (j \in \I_\btau)$; note $k_j$ here corresponds to the notation $\widetilde{K}_j \widetilde{K}_{\btau j}^{-1}$ in \cite{CLW18}.

Let us now fix the field $\mathbb K =\Q(v) \big(\vs_i\mid i\in \I_\tau \big)$, where the $\vs_i$'s are algebraically independent over $\Q(v)$. Fixing a square root $(\tk_j \tk_{\btau j})^{1/2}$ and identifying it with $\vs_i$, for $j \in \I_\tau$,
we consider the base change $\tUi_{\mathbb K} ={\mathbb K} \otimes \tUi$.
Then, over $\mathbb K$, $\tUi_{\mathbb K}$ is isomorphic to the $\mathbb K$-algebra $\Ui$ (with $\Ui$ in \cite[Theorem~3.1]{CLW18}) by sending $B_i \mapsto B_i\; (i\in \I)$, $\tk_j \mapsto \vs_{j} k_j\; (j \in \I_\btau)$ (and it follows that $\tk_{\btau j} \mapsto \vs_j k_j^{-1})$.

Now the presentation of $\Ui$ in \cite[Theorem~3.1]{CLW18} translates into the presentation for $\tUi$ in the statement.
\end{proof}

The relation~\eqref{relation5} in the setting of $\Ui$ originates in \cite{BK15}, and will be referred to as the BK relation. The $\imath$Serre relation~\eqref{relation6} first appeared in \cite{CLW18} and higher order $\imath$Serre relations have been formulated in \cite{CLW20}.

\begin{remark}
All constructions and results in this section (in particular, Theorem~\ref{thm:Serre}) are valid for $\Ui$ and $\tUi$ associated to symmetrizable generalized Cartan matrices, with various $v$-powers in $v$-binomials, $\imath$divided powers and \eqref{relation1}--\eqref{relation6} replaced by suitable $v_i$-powers.
\end{remark}

\subsection{Virtually acyclic $\imath$quivers}

To facilitate the computations in $\imath$Hall algebras in connection to $\imath$quantum groups $\tUi$, we shall consider a distinguished class of $\imath$quivers. Recall an oriented cycle of $Q$ is called \emph{minimal} if
it does not contain any proper oriented cycle. A minimal cycle of length $m$ is called an \emph{$m$-cycle}. 

\begin{definition}
\label{def:i-acyclic}
An $\imath$quiver $(Q,\btau)$ is called {\em virtually acyclic} if its only possible cycles are $2$-cycles between $i$ and $\tau i$ for $\tau i \neq i \in Q_0$.
\end{definition}
Note that acyclic $\imath$quivers are virtually acyclic $\imath$quivers.

\begin{lemma}
Let $(Q,\btau)$ be a virtually acyclic $\imath$quiver and $i \in Q_0$ such that $\btau i\neq i$. Then $\#\{\alpha:i\rightarrow j \mid \alpha\in Q_1\}=\#\{ \alpha:j\rightarrow i\mid \alpha\in Q_1\}$, and the number of edges between $i$ and $\tau i$ is even.
\end{lemma}

\begin{proof}
Follows by the definition.
\end{proof}

\begin{example}
\label{ex:Kquiver}
Denote by
\begin{align}
  \label{eq:Kq}
Q=\xymatrix{ {\Large \text{1}} \ar@<3ex>[rr]|-{\alpha_r}  \ar@<0.75ex>[rr]|-{\alpha_1}^{\cdots}   && {\Large \text{2}}\ar@<0.75ex>[ll]|-{\beta_1}  \ar@<3ex>[ll]|-{\beta_r}_{\cdots}  },
\qquad\qquad
\ov{Q}= \xymatrix{ {\Large \text{1}}\ar@/^1.5pc/@<2.5ex>@[purple][rr]^{\textcolor{purple}{\varepsilon_1}} \ar@<3ex>[rr]|-{\alpha_r}   \ar@<0.75ex>[rr]|-{\alpha_1}^{\cdots}   && {\Large \text{2}}\ar@<0.75ex>[ll]|-{\beta_1} \ar@/^1.5pc/@<2.5ex>@[purple][ll]^{\textcolor{purple}{\varepsilon_2}} \ar@<3ex>[ll]|-{\beta_r}_{\cdots}  }
\end{align}
Then $Q$ is a generalized Kronecker quiver, with involution $\btau$ given by $\btau 1=2$. Note that the $\imath$quiver $(Q,\btau)$ is virtually acyclic but not acyclic, for $r\ge 1$; moreover, $\Lambda^\imath =\K \ov{Q}/\ov{I}$, where
\[
\ov{I}=(\varepsilon_1\varepsilon_2, \varepsilon_2\varepsilon_1, \alpha_i\varepsilon_2-\varepsilon_1\beta_i,\beta_i\varepsilon_1-\varepsilon_2\alpha_i\mid 1\le i\le r).
\]
This is a new rank one $\imath$quiver algebra which did not appear in \cite{LW19a}. (The rank here refers to the number of $\tau$-orbits on the vertex set of $Q$.)
\end{example}


In the remainder of this paper, we shall restrict ourselves to the $\imath$Hall algebras $\tMHk$ associated to virtually acyclic $\imath$quivers. This suffices for the $\imath$Hall algebra realization of the $\imath$quantum groups $\tUi$ which we shall develop. The generalized Cartan matrix of $\U$ has to satisfy $c_{i,\tau i}\in -2\N$ whenever $\tau i\neq i \in \I$ (and no other  conditions), a condition imposed from the $\imath$quivers; see Example~\ref{ex:Kquiver}.


\section{The BK relation in $\imath$Hall algebra}
\label{sec:Relation5}

In this section, we shall establish an identity in $\imath$Hall algebra $\tMHk$ which corresponds to the BK relation \eqref{relation5} in $\tUi$. By Lemma~\ref{lem:subalgebra}, we are reduced to considering the rank one generalized Kronecker $\imath$quiver.

\subsection{First computation in $\tMHk$}

Let $(Q, \btau)$ be generalized Kronecker $\imath$quivers as in Example~\ref{ex:Kquiver}.
Recall $\Lambda^\imath =\K \ov{Q}/\ov{I}$ where
$\ov{I}=(\varepsilon_1\varepsilon_2, \varepsilon_2\varepsilon_1, \alpha_i\varepsilon_2-\varepsilon_1\beta_i,\beta_i\varepsilon_1-\varepsilon_2\alpha_i\mid 1\le i\le r)$ and see \eqref{eq:Kq} for $\ov Q$. A $\iLa$-module $M$ is a tuple of the form
\[
M =(M_i, M(\alpha_j), M(\beta_j), M(\varepsilon_i))_{i=1,2;1\leq j\leq r}.
\]
Recall $\K Q$ is a subalgebra (and also a quotient algebra) of $\Lambda^\imath$. Recall
\[
q =\sqq^2.
\]

For a $\iLa$-module $S$, we shall write
\[
[l S] =[\underbrace{S\oplus \cdots \oplus S}_{l}],
\qquad
[S]^{*l} = \underbrace{[S]* \cdots * [S]}_{l}.
\]
The following formula follows by definitions.
\begin{lemma}
For $l\ge 1$ and $i=1, 2$, we have
\begin{align}
[S_i]^{*l}= \sqq^{-\frac{l(l-1)}{2}}[l S].
 \label{eq:DPS}
\end{align}
\end{lemma}

Corresponding to the $\imath$divided powers in \eqref{eq:iDP1}, we define the divided powers, for $i=1, 2$,
\begin{align}
[S_i]^{(l)}:=\frac{[S_i]^{*l}}{[l]^!_\sqq}=\sqq^{-\frac{l(l-1)}{2}}\frac{[l S]}{[l]^!_\sqq}.
\end{align}

Our goal in this section is to verify the relation \eqref{relation5} for $i=1$ and $\btau i=2$, see \eqref{eq:Kq}; the other case when $i=2$ follows by symmetry.

For any $\Lambda^\imath$-module $M=(M_i, M(\alpha_j), M(\beta_j), M(\varepsilon_i))_{i=1,2;1\leq j\leq r}$ such that $\widehat{M}=(2r+1)\widehat{S_1}+\widehat{S_2}$ in $K_0(\mod(\Lambda^\imath))$, we define
\begin{align}
  \label{eq:UW}
U_M:= \bigcap_{1\leq j\leq r} \Ker M(\alpha_j) \bigcap \Ker M(\varepsilon_1),\qquad W_M:= \Im M(\varepsilon_2) +\sum_{j=1}^r \Im M(\beta_j),
\end{align}
and let
\begin{align}  \label{eq:uw}
u_M:=\dim U_M,\qquad w_M:=\dim W_M.
\end{align}

Since $\Lambda^\imath$ is  a quotient algebra of $\K\ov{Q}$, we can view each $\Lambda^\imath$-module as a $\K\ov{Q}$-module naturally. Let $\widetilde{\ch}(\K\ov{Q})$ be the Hall algebra of $\mod(\K\ov{Q})$ with its Hall multiplication twisted by $\sqq^{\langle \cdot,\cdot\rangle_Q}$.
\begin{lemma}
\label{lem: quotient}
There exists an algebra epimorphism
\begin{align*}
\phi:\widetilde{\ch}(\K\ov{Q})\longrightarrow \tMHk
\end{align*}
defined by letting
\begin{align}
   \label{def:quotient}
[M]\mapsto\left\{ \begin{array}{cc} [M], & \text{ if } M\in\mod(\Lambda^\imath), \\
0,& \text{ otherwise}. \end{array} \right.
\end{align}
\end{lemma}

\begin{proof}
Let $\widetilde{\phi}:\widetilde{\ch}(\K\ov{Q})\longrightarrow \widetilde{\ch}(\Lambda^\imath)$ be the linear map defined by \eqref{def:quotient}.
It suffices to check that $\widetilde{\phi}$ is a homomorphism.

For any $L,M,N\in \mod(\K \ov{Q})$, if there exists a short exact sequence
$$0\longrightarrow M\longrightarrow N\longrightarrow L\longrightarrow 0$$
such that $N\in \mod(\Lambda^\imath)$, then $L,M\in\mod(\Lambda^\imath)$.
So $\widetilde{\phi}([L]*[M])=0=\widetilde{\phi}([L])*\widetilde{\phi}([M])$ if one of $L,M$ is not in $\mod(\Lambda^\imath)$.

Now let $L,M\in\mod(\Lambda^\imath)$. Then, for any $N\in\mod(\Lambda^\imath)$, we have
\[
\big| \Hom_{\K\ov{Q}}(L,M) \big | =\big| \Hom_{\Lambda^\imath}(L,M) \big|,
\qquad
\big| \Ext^1_{\K\ov{Q}}(L,M)_N|= |\Ext^1_{\Lambda^\imath}(L,M)_N \big|.
\]
So $\widetilde{\phi}([L]*[M])=\widetilde{\phi}([L])*\widetilde{\phi}([M])$ by definition.

The lemma follows.
\end{proof}


For any three objects $X,Y,Z$, let
\begin{align}
 \label{eq:Fxyz}
F_{XY}^Z:= \big |\{L\subseteq Z \mid L \cong Y,  Z/L\cong X\} \big |.
\end{align}

\begin{lemma}[Riedtman-Peng formula]
\label{lem: Ried-P}
For any three objects $X,Y,Z$, we have
\[
F_{XY}^Z= \frac{|\Ext^1(X,Y)_Z|}{|\Hom(X,Y)|} \cdot \frac{|\aut(Z)|}{|\aut(X)| |\aut(Y)|}.
\]
\end{lemma}

Similar to \cite{Rin,Gr95}, a direct computation in $\widetilde{\ch}(\K\ov{Q})$ using Lemma~ \ref{lem: Ried-P} shows that
\begin{align}
  \label{eq:single}
[S_1]^{(l)}*[S_2]*[S_1]^{(t)}=\sqq^{-r(2r+1)+tl+l(l-1)+t(t-1)}\sum_{[M]\in\Iso(\mod(\K\ov{Q}))}\bp_{M,t}\frac{(q-1)^{2r+2}}{|\aut(M)|}[M],
\end{align}
where $\bp_{M,t}=0$ unless $W_M\subseteq U_M$ (see \eqref{eq:UW}); if  $W_M\subseteq U_M$, then we have
\begin{align}
\bp_{M,t}=& |\Gr(t-w_M, u_M-w_M)|=\sqq^{(u_M-t)(t-w_M)}\qbinom{u_M-w_M}{t-w_M},
 \label{eq:pMt}
 \end{align}
 where $\Gr(a,N)$ denotes the Grassmannian of $a$-subspaces in $\K^N$, and
 \begin{align}
|\aut(M)|=&(q-1)(q^{u_M-w_M}-1)\cdots(q^{u_M-w_M}-q^{u_M-w_M-1})
  \label{eq:autM} \\
&\quad \cdot q^{w_M(u_M-w_M)+w_M(2r+1-u_M)+(u_M-w_M)(2r+1-u_M)}.
\notag
\end{align}
Here one should note the difference between Ringel's Hall multiplication and Bridgeland's; see \cite[\S2.3]{Br}.
By Lemma \ref{lem: quotient} and \eqref{eq:single}, we have the following identity in $\tMHk$:
\begin{align}
 \label{eq:single2}
[S_1]^{(l)}*[S_2]*[S_1]^{(t)}=\sqq^{-r(2r+1)+tl+l(l-1)+t(t-1)}\sum_{[M]\in\Iso(\mod(\Lambda^\imath))}\bp_{M,t}\frac{(q-1)^{2r+2}}{|\aut(M)|}[M].
\end{align}

Summing up \eqref{eq:single2}, we obtain
\begin{align}
  \label{eq:SumSerre}
\sum_{l+t =2r+1} (-1)^l [S_1]^{(l)}*[S_2]*[S_1]^{(t)}=\sum_{[M]: W_M\subseteq U_M} \bp_M[M],
\end{align}
where, thanks to \eqref{eq:pMt},
\begin{align}
 \label{eq:pM}
\bp_M=&(q-1)^{(2r+2)}\sum_{l=0}^{2r+1} (-1)^l\sqq^{-r(2r+1)+tl+l(l-1)+t(t-1)}\frac{\bp_{M,t}}{|\aut(M)|}
\\
=&  \sqq^{r(2r+1)-u_Mw_M} \frac{(q-1)^{(2r+2)}}{|\aut(M)|} \sum_{t=0}^{2r+1} (-1)^{2r+1-t}\sqq^{-(2r+1-u_M-w_M)t}\qbinom{u_M-w_M}{t-w_M}.
\notag
\end{align}

For any $M$ such that $\widehat{M}=(2r+1)\widehat{S_1}+\widehat{S_2}$ in $K_0(\mod(\Lambda^\imath))$, clearly, either $M(\varepsilon_1)=0$ or $M(\varepsilon_2)=0$.
To complete the computation of \eqref{eq:SumSerre}, we proceed by dividing into 3 cases below:
\begin{enumerate}
\item
$M(\varepsilon_1)=0= M(\varepsilon_2)$,
\item
$M(\varepsilon_1)=0\neq M(\varepsilon_2)$,
\item
$M(\varepsilon_2)=0\neq M(\varepsilon_1)$.
\end{enumerate}
We shall need the following specializations of the quantum binomial formula.
\begin{lemma}
\label{lem:qbinom1}
Let $p \in \Z_{\geq1}$.
Let $d\in \Z$ be such that $|d| \le p-1$ and $d\equiv p-1\pmod 2$. Then,
\begin{enumerate}
\item
$\sum_{t=0}^{p}(-1)^tv^{-dt}\qbinom{p}{t}=0;$
\item
$\sum_{t=0}^{p}(-1)^{t}v^{-(p+1)t}\qbinom{p}{t}= (v^{-2}; v^{-2})_p$;
\item
$\sum_{t=0}^{p}(-1)^{-t}v^{(p+1)t}\qbinom{p}{t}=  (v^2;v^2)_{p}$.
\end{enumerate}
\end{lemma}

\begin{proof}
Recall the quantum binomial formula (cf., e.g., \cite[1.3.1(c)]{L93})
\begin{align}
   \label{eq:BF}
\sum_{t=0}^p v^{t(1-p)}\qbinom{p}{t} z^t = \prod_{j=0}^{p-1} (1+ v^{-2j} z).
\end{align}
Then Formula (1) follows from it by letting $z=v^{p-1-d}$.
Formula (2) follows by letting $z=v^{-2}$, and (3) follows from (2) by applying the bar involution $v \mapsto v^{-1}$.
\end{proof}

\subsection{Case $M(\varepsilon_1)=0= M(\varepsilon_2)$} In this case, we may regard $M\in\mod(\K Q) \subseteq\mod(\Lambda^\imath)$.
Recall $\bp_M$ from \eqref{eq:pM}.

\begin{lemma}
\label{lem: gen serre 1}
We have $\bp_M=0$, for any $M\in\mod(\K Q)$ such that $\widehat{M}=(2r+1)\widehat{S}_1 +\widehat{S}_2$.
\end{lemma}

\begin{proof}
For any such $M$ with $M(\varepsilon_i)=0$, for $i=1,2$, we have $u_M\ge r+1> w_M$. We deduce that
$1-u_M-w_M \le 2r+1-u_M-w_M\le u_M+w_M-1$.

By a change of variables $s=t-w_M$ and Lemma~\ref{lem:qbinom1}(1), we have
\begin{align*}
\sum_{t=0}^{2r+1}  & (-1)^{2r+1-t}\sqq^{-(2r+1-u_M-w_M)t}\qbinom{u_M-w_M}{t-w_M} \\
& = (-1)^{2r+1+w_M} \sqq^{-(2r+1-u_M-w_M)w_M}\sum_{s=0}^{u_M-w_M} (-1)^{s}\sqq^{-(2r+1-u_M-w_M)s}\qbinom{u_M-w_M}{s}
=0.
\end{align*}
Then by \eqref{eq:pM}, we obtain $\bp_M =0$.
\end{proof}

\subsection{Case $M(\varepsilon_2)=0\neq M(\varepsilon_1)$}
In this case, we have $M(\beta_i)=0$, for $1\le i\le r$, by noting that $M(\beta_i)M(\varepsilon_1)=M(\varepsilon_2)M(\alpha_i)=0$. Then $w_M=0$ (recall $u_M, w_M$ from \eqref{eq:uw}).
Denote
\begin{align*}U'_M =\bigcap_{1\leq i\leq r} \Ker M(\alpha_i),
\qquad
u'_M =\dim U'_M.
\end{align*}

\begin{lemma}
\label{lem:ker not inclusion}
Retain the notations and assumptions as above. Then
there exists a short exact sequence $0\longrightarrow \bK_1 \longrightarrow M\longrightarrow S_1^{\oplus 2r} \longrightarrow0$ if 
$M(\varepsilon_1)|_{U'_M}\neq0$.
\end{lemma}
\begin{proof}
It follows by the definition of morphisms of quiver representations.
\end{proof}

We now proceed by dividing into 2 subcases.

\vspace{2mm}

\underline{Subcase (a): $M(\varepsilon_1)|_{U'_M}=0$}. Clearly, $U'_M\subseteq \ker M(\varepsilon_1)$.
Then we have $U_M=U'_M$ by definition. It follows that $u_M\ge r+1> w_M=0$, and $1-u_M-w_M\le 2r+1-u_M-w_M\le u_M+w_M-1$.
Similar to Lemma \ref{lem: gen serre 1}, we deduce that $\bp_M=0$.

\vspace{2mm}

\underline{Subcase (b): $M(\varepsilon_1)|_{U'_M}\neq0$}. Note by Lemma \ref{lem:ker not inclusion} that
$[M]=[\bK_1\oplus S_1^{\oplus2r}]$. In this case $u_M=u_M'-1$ by noting that $\Ker M(\varepsilon_1)$ is a hyperplane. Note that $u_M'\geq r+1$, and hence $u_M\geq r$. In case $u_M\geq r+1$, we have $\bp_M=0$ by arguments similar to the above.

It remains to consider the subcase when $u_M=r$. In this case, there is a {\em unique $M$ (up to isomorphism) such that $u_M=r$}, and note that
\begin{align}   \label{eq:autMr}
|\aut(M)|= (q-1)(q^{r}-1)\cdots(q^{r}-q^{r-1}) q^{r(r+1)}.
\end{align}
Thus, 
applying Lemma~\ref{lem:qbinom1}(2) (with $p=r$) we have
\begin{align}
\bp_M=& \sqq^{r(2r+1)-u_Mw_M} \frac{(q-1)^{(2r+2)}}{|\aut(M)|} \sum_{t=0}^{2r+1} (-1)^{2r+1-t}\sqq^{-(2r+1-u_M-w_M)t}\qbinom{u_M-w_M}{t-w_M}_\sqq
 \label{eq:pM1}
\\
=&- \sqq^{r(2r+1)}\frac{(q-1)^{(2r+2)}}{|\aut(M)|} \sum_{t=0}^{r}(-1)^{t}\sqq^{-(r+1)t}\qbinom{r}{t}_\sqq
 \notag \\
=&- \sqq^{r(2r+1)}\frac{(q-1)^{(2r+2)}}{|\aut(M)|}(\sqq^{-2}; \sqq^{-2})_r
 \notag \\
=&-(q-1)^{(2r+1)}\sqq^{r(-2r-1)}. \notag
\end{align}
We also note that
\begin{align}
[2r  S_1]*[\bK_1]&=\sqq^{\langle   S_1^{\oplus 2r},\res\bK_1\rangle_Q} q^{-\langle   S_1^{\oplus 2r},\bK_1\rangle} [2r  S_1 \oplus\bK_1]
 \label{eq:SE1}
\\
&=q^{r^2+r}[2r  S_1 \oplus\bK_1].
\notag
\end{align}
Therefore, using \eqref{eq:DPS} and \eqref{eq:pM1}-\eqref{eq:SE1} we obtain
\begin{align}
   \label{eqn: gen serre 2}
\sum_{[M]: M(\varepsilon_1)\neq0} \bp_M[M]
=&\sum_{[M]:  u_M=u_M'-1} \bp_M [2r S_1 \oplus \bK_1]
  \\
=&\sum_{u_M=r} \sqq^{-3r}[2r]_\sqq^!\bp_M [S_1]^{(2r)}*[\bK_1]
\notag \\
=&-q^{-r^2-2r}(q-1)^{(2r+1)}[2r]_\sqq^![S_1]^{(2r)}*[\bK_1].\notag
\end{align}

\subsection{Case $M(\varepsilon_1)=0\neq M(\varepsilon_2)$}

In this case, we have $M(\alpha_i)=0$, for $1\le i\le r$.
So $u_M=2r+1$, and $w_M \leq r+1$. In case $w_M < r+1$, similar to Lemma \ref{lem: gen serre 1},  we have
$\bp_M=0$.

It remains to consider the subcase $w_M=r+1$. In this case, there is a {\em unique $M$ (up to isomorphism) such that $w_M=r+1$}. When $w_M=r+1$ (recall $u_M=2r+1$), $|\aut(M)|$ is given again as in \eqref{eq:autMr}. By applying Lemma~\ref{lem:qbinom1}(3) (with $p=r$) and changing variables $t' =t-r-1$, we have
\begin{align}
\bp_M=& \sqq^{r(2r+1)-u_Mw_M} \frac{(q-1)^{(2r+2)}}{|\aut(M)|} \sum_{t=0}^{2r+1} (-1)^{2r+1-t}\sqq^{-(2r+1-u_M-w_M)t}\qbinom{u_M-w_M}{t-w_M}_\sqq
  \label{eq:pM2}
\\
=& \sqq^{r(2r+1)-(2r+1)(r+1)}\frac{(q-1)^{(2r+2)}}{|\aut(M)|} \sum_{t=0}^{2r+1}(-1)^{2r+1-t}\sqq^{(r+1)t}\qbinom{r}{t-(r+1)}_\sqq
\notag \\
=& \sqq^{r^2}\frac{(q-1)^{(2r+2)}}{|\aut(M)|}\sum_{t'=0}^{r}(-1)^{r-t'}\sqq^{(r+1)t'}\qbinom{r}{t'}_\sqq
\notag \\
=& \sqq^{r^2}\frac{(q-1)^{(2r+2)}}{|\aut(M)|}(-1)^r(\sqq^{2}; \sqq^{2})_r
\notag \\
=&\sqq^{-2r^2-r}(q-1)^{2r+1}.  \notag
\end{align}
Note also that
\begin{align}
  \label{eq:SE2}
[2r S_1 ]*[\bK_2]&=\sqq^{\langle S_1^{\oplus  2r },\res\bK_2\rangle_Q} q^{-\langle  S_1^{\oplus 2r},\bK_2\rangle} [2r S_1 \oplus\bK_2]
\\
&=q^{-r^2-r}[2r S_1 \oplus\bK_2].
\notag
\end{align}
Therefore, using \eqref{eq:DPS} and \eqref{eq:pM2}--\eqref{eq:SE2} we obtain
\begin{align}
\label{eqn: gen serre 3}
\sum_{[M]: M(\varepsilon_2)\neq0} \bp_M[M]
=&\sum_{[M]: w_M=r+1} \bp_M[ 2r S_1 \oplus \bK_2]\\
=&\sum_{[M]: w_M=r+1} q^{r^2+r}\sqq^{\frac{2r(2r-1)}{2}}[2r]^!_\sqq\bp_M [S_1]^{(2r)}*[\bK_2]
\notag \\
=&q^{r^2}(q-1)^{(2r+1)}[2r]^!_\sqq[S_1]^{(2r)}*[\bK_2].\notag
\end{align}

\subsection{Relation \eqref{relation5} in $\tMHk$}

Now we are ready to establish the following identity in the $\imath$Hall algebra corresponding to the BK relation \eqref{relation5}.

\begin{proposition}
   \label{prop:serre 1}
The following identity holds in $\imath$Hall algebra $\tMHk$:
\begin{align*}
\sum_{t=0}^{2r+1} (-1)^t & [S_1]^{(t)}*[S_2]*[S_1]^{(2r+1-t)}
\\
=&-\sqq^{-r}(q-1)(q^{-1};q^{-1})_{2r} [S_1]^{(2r)}*[\bK_1]+\sqq^r(q-1)(q;q)_{2r} [S_1]^{(2r)}*[\bK_2].
\end{align*}
\end{proposition}

\begin{proof}
Combining Lemma~\ref{lem: gen serre 1}, \eqref{eqn: gen serre 2} and \eqref{eqn: gen serre 3}, we finish the computation in \eqref{eq:SumSerre} as follows:
\begin{align*}
\sum_{t=0}^{2r+1} & (-1)^t[S_1]^{(t)}*[S_2]*[S_1]^{(2r+1-t)}
\\
=&-q^{-r^2-2r}(q-1)^{(2r+1)}[2r]^!_\sqq[S_1]^{(2r)}*[\bK_1]+q^{r^2}(q-1)^{(2r+1)}[2r]^!_\sqq [S_1]^{(2r)}*[\bK_2]
\\
=&-\sqq^{-r}(q-1)(q^{-1};q^{-1})_{2r} [S_1]^{(2r)}*[\bK_1]+\sqq^r(q-1)(q;q)_{2r} [S_1]^{(2r)}*[\bK_2].
\end{align*}
The proposition is proved.
\end{proof}

\section{$\imath$Divided powers in $\imath$Hall algebra}
  \label{sec:iDP}

In this section, we establish closed formulas for the $\imath$divided powers in terms of $\imath$Hall basis for the $\imath$Hall algebra $\tMHk$.

To that end, by Lemma~\ref{lem:subalgebra}, it suffices to consider the $\imath$quiver which consists of a single vertex with a trivial involution, and the associated $\imath$quiver algebra given by $\Lambda^\imath = \K[x]/(x^2)$. The corresponding split $\imath$quantum group $\tUi$ of rank one is the algebra $\Q(v)[\ff,\tk^{\pm1}]$.
The following is the special case of \cite[Proposition 7.5]{LW19a} at rank one.

\begin{lemma}
\label{lem: iso rank 1}
There exists a $\Q(\sqq)$-algebra isomorphism
$\widetilde{\psi}:\tUi|_{v=\sqq}\longrightarrow \tMHk$ (of rank one)
which sends
\begin{align*}
\ff\mapsto \frac{-1}{q-1}[S], \qquad \tk\mapsto -\frac{\bK}{q}.
\end{align*}
\end{lemma}

\begin{lemma}
The following identity holds in $\tMHk$, for $m\in\N$:
\begin{align}  \label{eq:SmS}
[S]*[mS]= \sqq^{-m}[(m+1)S]+(\sqq^m-\sqq^{-m})[(m-1)S]*[\bK].
\end{align}
\end{lemma}

\begin{proof}
The required Euler form is given by $\langle S,S^{\oplus m}\rangle_Q=m=\dim_\K\Hom_{\Lambda^\imath}(S,S^{\oplus m})$. For any non-split short exact sequence
$0\rightarrow S^{\oplus m} \rightarrow M\rightarrow S\rightarrow0$
in $\mod(\Lambda^\imath)$, we have $M\cong S^{\oplus(m-1)}\oplus \bK$. Note that $\Ext^1_{\Lambda^\imath}(S,S^{\oplus m})=m$. Then we have
\begin{align*}
[S]*[mS]=&\sqq^{-m}[(m+1)S]+\sqq^{-m}(q^m-1) [(m-1)S\oplus \bK]
\\
 = &\sqq^{-m}[(m+1)S]+(\sqq^m-\sqq^{-m})[(m-1)S]*[\bK].
\end{align*}
The lemma is proved.
\end{proof}

Inspiring by \eqref{eq:iDPodd}--\eqref{eq:iDPev}, we define the $\imath$-divided power of $[S]$ in $\tMHk$ as follows:
\begin{align*}
&[S]_{\odd}^{(m)}:=\frac{1}{[m]_{\sqq}!}\left\{ \begin{array}{ll} [S]\prod_{j=1}^k ([S]^2+\sqq^{-1}(\sqq^2-1)^2[2j-1]_{\sqq}^2 [\bK] ) & \text{if }m=2k+1,\\
\prod_{j=1}^k ([S]^2+\sqq^{-1}(\sqq^2-1)^2[2j-1]_{\sqq}^2[\bK]) &\text{if }m=2k; \end{array}\right.
 \\
&[S]_{\ev}^{(m)}:= \frac{1}{[m]_{\sqq}!}\left\{ \begin{array}{ll} [S]\prod_{j=1}^k ([S]^2+\sqq^{-1}(\sqq^2-1)^2[2j]_{\sqq}^2[\bK] ) &\text{if }m=2k+1,\\
\prod_{j=1}^{k} ([S]^2+\sqq^{-1}(\sqq^2-1)^2[2j-2]_{\sqq}^2[\bK]) &\text{if }m=2k. \end{array}\right.
\end{align*}

These $\imath$-divided powers satisfy the following recursive relations:
\begin{align}
[S]*[S]^{(2m)}_{\odd} &=[2m+1] [S]^{(2m+1)}_{\odd},
 \label{rec:2modd}
\\
[S]*[S]^{(2m+1)}_{\odd} &=[2m+2] [S]^{(2m+2)}_{\odd} -\sqq (\sqq-\sqq^{-1})^2[2m+1] [S]^{(2m)}_{\odd}*[\bK],
 \label{rec:2m+1odd}
\\
[S]*[S]^{(2m-1)}_{\ev} &=[2m][S]^{(2m)}_{\ev},
 \label{rec:2m-1ev}
\\
[S]*[S]^{(2m)}_{\ev} &=[2m+1] [S]^{(2m+1)}_{\ev} -\sqq (\sqq-\sqq^{-1})^2 [2m] [S]^{(2m-1)}_{\ev}*[\bK].
 \label{rec:2mev}
\end{align}


\begin{lemma}
The isomorphism $\widetilde{\psi}$ in Lemma~\ref{lem: iso rank 1} satisfies that, for $m\in \N$,
\begin{align}
\widetilde{\psi}(\ff_{\odd}^{(m)})=\frac{[S]_{\odd}^{(m)}}{(1-\sqq^2)^m},
\qquad
\widetilde{\psi}(\ff_{\ev}^{(m)})=\frac{[S]_{\ev}^{(m)}}{(1-\sqq^2)^m}.
\end{align}
\end{lemma}

\begin{proof}
Follows by definitions.
\end{proof}

We denote by $[0]_\sqq ^{^{!!}}=1$, and for any $k\in \Z_{\ge 1}$,
\[
[2k]_\sqq^{^{!!}}=[2k]_\sqq[2k-2]_\sqq \cdots [4]_\sqq[2]_\sqq.
\]
We denote by $\lfloor x \rfloor$ the largest integer not exceeding $x$, for $x\in \mathbb R$.

\begin{proposition}
 \label{prop:iDPev}
The following identity holds in $\tMHk$, for $n\in \N$:
\begin{align}
  \label{eq:nev}
&[S]^{(n)}_{\ev}=\sum_{k=0}^{\lfloor \frac{n}2\rfloor} \frac{\sqq^{k(k - (-1)^n)-\binom{n-2k}{2}}\cdot(\sqq-\sqq^{-1})^k}{[n-2k]_{\sqq}^{!}[2k]_\sqq^{!!}}[(n-2k)S]*[\bK]^k.
\end{align}
\end{proposition}

\begin{proof}
We prove the formula by induction on $n$; the cases when $n=0,1$ are clear.
To facilitate the induction, let us rewrite \eqref{eq:nev} depending on the parity of $n$:
for $c \in \N$,
\begin{align}
&[S]^{(2c)}_{\ev}=\sum_{k=0}^c \frac{\sqq^{k(k-1)-\binom{2c-2k}{2}}\cdot(\sqq-\sqq^{-1})^k}{[2c-2k]_{\sqq}^{!}[2k]_\sqq^{!!}}[(2c-2k)S]*[\bK]^k,
 \label{eq:2mev}
 \\
&[S]^{(2c+1)}_{\ev}=\sum_{k=0}^c \frac{\sqq^{k(k+1) -\binom{1+2c-2k}{2}}\cdot(\sqq-\sqq^{-1})^k}{[1+2c-2k]_{\sqq}^{!}[2k]_\sqq^{!!}}[(1+2c-2k)S]*[\bK]^k.
 \label{eq:2m+1ev}
\end{align}

First we shall prove \eqref{eq:2mev} for $[S]^{(2c)}_{\ev}$ by assuming the formula holds  for $[S]^{(2c-1)}_{\ev}$ with $c\ge 1$. Using the inductive assumption \eqref{eq:2m+1ev} (with $c$ replaced by $c-1$), \eqref{rec:2m-1ev} and \eqref{eq:SmS}, we have
\begin{align*}
& [2c]_\sqq [S]^{(2c)}_{\ev} =[S]*[S]^{(2c-1)}_{\ev}
\\
& = \sum_{k=0}^{c-1} \frac{\sqq^{k(k+1) -\binom{2c-2k-1}{2}}\cdot(\sqq-\sqq^{-1})^k}{[2c-2k-1]_{\sqq}^{!}
  [2k]_\sqq^{!!}}  [S]*[(2c-2k-1)S]*[\bK]^k
\\
& = \sum_{k=0}^{c-1} \frac{\sqq^{k(k+1) -\binom{2c-2k-1}{2}}\cdot(\sqq-\sqq^{-1})^k}{[2c-2k-1]_{\sqq}^{!}[2k]_\sqq^{!!}} \times
\\ &\qquad
\Big(\sqq^{1+2k-2c} [(2c-2k)S]*[\bK]^k + (\sqq^{2c-2k-1} -\sqq^{1+2k-2c})  [(2c-2k-2)S]*[\bK]^{k+1} \Big)
\\
& \stackrel{(a)}{=}\sum_{k=0}^c \frac{1}{[2c-2k]_{\sqq}^{!}[2k]_\sqq^{!!}}
  \left( \sqq^{k(k+1)-\binom{2c-2k-1}{2}+ (1+2k-2c)}\cdot(\sqq-\sqq^{-1})^k [2c-2k]_{\sqq}
  \right.
  \\
    &\qquad\qquad \left.+  \frac{\sqq^{k(k-1)-\binom{2c-2k+1}{2}}\cdot(\sqq-\sqq^{-1})^{k-1} (\sqq^{2c-2k+1} -\sqq^{2k-2c-1}) [2k]}{[2c-2k+1]_{\sqq}}
  \right)
   [(2c-2k)S]*[\bK]^k
   \\
& = \sum_{k=0}^c \frac{\sqq^{k(k-1)-\binom{2c-2k}{2}}\cdot(\sqq-\sqq^{-1})^k}{[2c-2k]_{\sqq}^{!}[2k]_\sqq^{!!}}
  \left( \sqq^{2k} [2c-2k]_{\sqq}  + \sqq^{2k-2c} [2k]_{\sqq}
  \right)
   [(2c-2k)S]*[\bK]^k
\\
& = [2c]_\sqq \sum_{k=0}^c \frac{\sqq^{k(k-1)-\binom{2c-2k}{2}}\cdot(\sqq-\sqq^{-1})^k}{[2c-2k]_{\sqq}^{!}[2k]_\sqq^{!!}}[(2c-2k)S]*[\bK]^k.
\end{align*}
In the equation (a) above, we have shifted the index $k \mapsto k-1$ in the second summand on the LHS of (a). This proves \eqref{eq:2mev}.

We now prove \eqref{eq:2m+1ev} for $[S]^{(2c+1)}_{\ev}$ by assuming \eqref{eq:2mev} for $[S]^{(2c)}_{\ev}$ and the formula for $[S]^{(2c-1)}_{\ev}$ (i.e., \eqref{eq:2m+1ev} with $c$ replaced by $c-1$). Together with \eqref{rec:2mev} and \eqref{eq:SmS}, we compute
\begin{align*}
& [2c+1]_\sqq [S]^{(2c+1)}_{\ev}
\\
 &=[S]*[S]^{(2c)}_{\ev} +\sqq(\sqq -\sqq^{-1})^2 [2c]_\sqq [S]^{(2c-1)}_{\ev} *[\bK]
\\
&= \sum_{k=0}^c \frac{\sqq^{k(k-1)-\binom{2c-2k}{2}}\cdot(\sqq-\sqq^{-1})^k}{[2c-2k]_{\sqq}^{!}[2k]_\sqq^{!!}}
[S]*[(2c-2k)S]*[\bK]^k
+\sqq(\sqq -\sqq^{-1})^2 [2c]_\sqq [S]^{(2c-1)}_{\ev} *[\bK]
\\%
&= \sum_{k=0}^c \frac{\sqq^{k(k-1)-\binom{2c-2k}{2}}\cdot(\sqq-\sqq^{-1})^k}{[2c-2k]_{\sqq}^{!}[2k]_\sqq^{!!}} \times
\\
& \qquad
\Big(\sqq^{2k-2c} [(1+2c-2k)S]*[\bK]^k + (\sqq^{2c-2k} -\sqq^{2k-2c})  [(2c-2k-1)S]*[\bK]^{k+1} \Big)
\\
&\qquad +  [2c]_\sqq \sum_{k=0}^{c-1} \frac{\sqq^{k(k+1) -\binom{2c-2k-1}{2} +1}\cdot(\sqq-\sqq^{-1})^{k+2}}{[2c-2k-1]_{\sqq}^{!}[2k]_\sqq^{!!}}[(2c-2k-1)S]*[\bK]^{k+1},
\\ %
\end{align*}
which can be reorganized as
\begin{align*}
&= \sum_{k=0}^c \frac{\sqq^{k(k-1)-\binom{2c-2k}{2} +2k-2c}\cdot(\sqq-\sqq^{-1})^k}{[2c-2k]_{\sqq}^{!}[2k]_\sqq^{!!}} [(1+2c-2k)S]*[\bK]^k
\\
& \qquad
+ \left( \sum_{k=0}^{c-1} \frac{\sqq^{k(k-1)-\binom{2c-2k}{2}}\cdot(\sqq-\sqq^{-1})^k}{[2c-2k]_{\sqq}^{!}[2k]_\sqq^{!!}} (\sqq^{2c-2k} -\sqq^{2k-2c})   \right.
\\
&\qquad\qquad \left.
+  [2c]_\sqq \sum_{k=0}^{c-1} \frac{\sqq^{k(k+1) -\binom{2c-2k-1}{2} +1}\cdot(\sqq-\sqq^{-1})^{k+2}}{[2c-2k-1]_{\sqq}^{!}[2k]_\sqq^{!!}}
\right)
[(2c-2k-1)S]*[\bK]^{k+1}
\\%
& = \sum_{k=0}^c \frac{\sqq^{k(k-1)-\binom{2c-2k}{2} +2k-2c}\cdot(\sqq-\sqq^{-1})^k}{[2c-2k]_{\sqq}^{!}[2k]_\sqq^{!!}} [(1+2c-2k)S]*[\bK]^k
\\
& \qquad
+ \sum_{k=0}^{c-1} \frac{\sqq^{k(k-1) -\binom{2c-2k}{2} +4c}\cdot(\sqq-\sqq^{-1})^{k+1}}{[2c-2k-1]_{\sqq}^{!}[2k]_\sqq^{!!}}
[(2c-2k-1)S]*[\bK]^{k+1}
\\%
&\stackrel{(b)}{=} \sum_{k=0}^c \frac{\sqq^{k(k+1)-\binom{1+2c-2k}{2} }\cdot(\sqq-\sqq^{-1})^k}{[1+2c-2k]_{\sqq}^{!}[2k]_\sqq^{!!}}
\\&\qquad
\left( \sqq^{-2k +(2c-2k) +2k-2c } [1+2c-2k]_\sqq
  +\sqq^{2-4k -(1+2c-2k) +4c} [2k]_\sqq
\right)
[(1+2c-2k)S]*[\bK]^k
\\%
& =  [2c+1]_\sqq\sum_{k=0}^c \frac{\sqq^{k(k+1)-\binom{1+2c-2k}{2} }\cdot(\sqq-\sqq^{-1})^k}{[1+2c-2k]_{\sqq}^{!}[2k]_\sqq^{!!}}
  [(1+2c-2k)S]*[\bK]^k.
\end{align*}
In the equation (b) above, we have shifted the index $k \mapsto k-1$ in the second summand on the LHS of (b).

The proposition is proved.
\end{proof}

\begin{proposition}
 \label{prop:iDPodd}
%
The following identity holds in $\tMHk$, for $n\in \N$:
\begin{align}
&[S]^{(n)}_{\odd}=\sum_{k=0}^{\lfloor \frac{n}2\rfloor}  \frac{\sqq^{k(k+(-1)^n)-\binom{n-2k}{2}}\cdot(\sqq-\sqq^{-1})^k}{[n-2k]_{\sqq}^{!}[2k]_\sqq^{!!}}[(n-2k)S]*[\bK]^k.
\end{align}
\end{proposition}

\begin{proof}
The proof is entirely similar to the one for Proposition~\ref{prop:iDPev}, using now the recursions \eqref{rec:2modd}--\eqref{rec:2m+1odd}; it will be skipped.
\end{proof}

\section{The $\imath$Serre relation in $\imath$Hall algebra}
  \label{sec:SerreRel}

In this section, we shall establish an identity in the $\imath$Hall algebra $\tMHk$ corresponding to the $\imath$Serre relation \eqref{relation6} in $\tUi$ (where $j=\tau j$), modulo a combinatorial identity which will be established in Section~\ref{sec:comb}. By Lemma~\ref{lem:subalgebra}, we are reduced to considering a rank two $\imath$quiver.

%
%
\subsection{Identities in $\imath$Hall algebra}

Consider the $\imath$quiver
\begin{align}
  \label{diag:split2}
Q=(\xymatrix{1\ar@<1ex>[r]^{\alpha_1}_{\cdots} \ar@<-1ex>[r]_{\alpha_{a}}  & 2}),
\quad
\btau=\Id,
\qquad \text{where } a=-c_{12}.
\end{align}
Then the corresponding $\imath$quiver algebra $\Lambda^\imath$ has its quiver $\ov{Q}$ as
\begin{center}\setlength{\unitlength}{0.7mm}
\begin{equation}
\label{diag:split KM}
 \begin{picture}(50,3)(0,0)
\put(0,-3){$1$}
\put(4,1){\vector(1,0){14}}
\put(4,-4){\vector(1,0){14}}
\put(8,-2.5){$\cdots$}
\put(9,1){$^{\alpha_1}$}
\put(9,-6.5){$_{\alpha_a}$}
\put(20,-3){$2$}
\color{purple}
\qbezier(-1,1)(-3,3)(-2,5.5)
\qbezier(-2,5.5)(1,9)(4,5.5)
\qbezier(4,5.5)(5,3)(3,1)
\put(3.1,1.4){\vector(-1,-1){0.3}}
\qbezier(19,1)(17,3)(18,5.5)
\qbezier(18,5.5)(21,9)(24,5.5)
\qbezier(24,5.5)(25,3)(23,1)
\put(23.1,1.4){\vector(-1,-1){0.3}}
\put(1,10){$\varepsilon_1$}
\put(21,10){$\varepsilon_2$}
\end{picture}
\vspace{5mm}
\end{equation}
\end{center}

We shall prove the following identity corresponding to the $\imath$Serre relation \eqref{relation6}, where $j=\tau j$.

\begin{theorem}
\label{thm: iserre relation}
Let $\Lambda^\imath$ be the $\imath$quiver algebra associated with the $\imath$quiver \eqref{diag:split2}.
Then the following identity holds in $\tMHk$, for any $\overline{p}\in\Z/2$:
\begin{align}
\sum_{n=0}^{1+a} (-1)^n  [S_1]_{\overline{p}}^{(n)}*[S_2] *[S_1]_{\overline{a}+\overline{p}}^{(1+a-n)} =0,
\label{eqn:iserre1}
\\
\sum_{n=0}^{1+a} (-1)^n  [S_2]_{\overline{p}}^{(n)}*[S_1] *[S_2]_{\overline{a}+\overline{p}}^{(1+a-n)} =0.
\label{eqn:iserre2}
\end{align}
\end{theorem}

\begin{remark}
The identity in the $\imath$Hall algebra corresponding to $\imath$Serre relation \eqref{relation6} for $j\neq \tau j$ can be proved similarly to or simply derived from Theorem~\ref{thm: iserre relation}; see Proposition~\ref{prop:relation62} below.
\end{remark}

\subsection{A building block}

We denote
\begin{align}
  \label{eq:Ik}
\mathcal{I}_{k} &=\big \{[M]\in\Iso(\mod(\K Q))\mid \exists N\subseteq M \text{ such that }N\cong S_2, M/N\cong k S_1 \big \}.
\end{align}
We also introduce the following polynomial in 4 variables:
\begin{align}
p(a,r,s,t) &=-s(a+t)+2ra+(u_M-t+2s-r)(t-r) +(s-r)^2
 \label{eq:p}
\\
&\qquad\qquad +{s-r \choose 2} + (t-r)^2 +{t-r \choose 2} +r(s+t) -{r+1 \choose 2}+1.
\notag
\end{align}
Any $\K Q$-module $M$ with dimension vector $n\widehat{S_1}+\widehat{S_2}$ can be decomposed as
\begin{align}  \label{eq:MNu}
M\cong N\oplus S_1^{\oplus u_M}
\end{align}
with $N$ indecomposable (unique up to isomorphism), for a unique $u_M \in \N$.

The following formula is a basic building block in the subsequent computations of the $\imath$Serre relation in an $\imath$Hall basis.

\begin{proposition}
  \label{prop:SSSM}
The following identity holds in $\tMHk$, for $s, t \ge 0$:
\begin{align}
&[s S_1]*[S_2]*[t S_1]
  \label{eq:SSS}\\
&=\sum_{r=0}^{\min\{ s,t\}} \sum_{[M]\in \mathcal{I}_{s+t-2r}}
\sqq^{p(a,r,s,t)} (\sqq -\sqq^{-1})^{s+t-r+1} \frac{[s]_\sqq^{!} [t]_\sqq^{!}}{[r]_\sqq^{!}}
\qbinom{u_M}{t-r}_\sqq \frac{[M]}{|\aut(M)|}*[\bK_1]^r.
\notag
\end{align}
\end{proposition}

\begin{proof}
By definition, we have
\begin{align} \label{eq:S1S2S1}
[s S_1 ]*[S_2]*[t S_1 ]= [s S_1 ]*[S_2\oplus S_1^{\oplus t}].
\end{align}
Observe that for any morphism $f: S_1^{\oplus s}\rightarrow S_2\oplus S_1^{\oplus t}$, we have $\Im f\cong S_1^{\oplus r}$, $\Ker f\cong S_1^{\oplus(s-r)}$ and $\coker f\cong S_2\oplus S_1^{\oplus (t-r)}$ for some $r$. In particular, $0\leq r\leq \min\{s,t\}$ in this case. Using a standard linear algebra fact (cf., e.g., \cite{M06}), we obtain
\begin{align}   \label{eq:C1}
& | \{f: S_1^{\oplus s}\rightarrow S_2\oplus S_1^{\oplus t}\mid \Ker f\cong S_1^{\oplus(s-r)}, \coker f\cong S_2\oplus S_1^{\oplus(t-r)}\}|
\\
=& \big| \{A\in M_{s\times t}(\K)\mid \rank A=r\} \big|
= \prod_{j=0}^{r-1}\frac{(q^s-q^j)(q^t-q^j)}{q^r-q^j}.
\notag
\end{align}

By \eqref{eq:S1S2S1} and applying \eqref{Hallmult1}, we have
\begin{align}
[s S_1 ]*[S_2]&*[t S_1 ] =\sum_{r=0}^{\min\{ s,t\}} \sum_{[M]\in \Iso(\mod(\Lambda^\imath))} \sqq^{ -sa+st} q^{-(s-r)a+t(s-r)-(s-r)s+(s-r)^2+sa-st}
 \label{eq:basic} \\
&\frac{|\Ext^1(S_1^{\oplus(s-r)},S_2\oplus S_1^{\oplus(t-r)})_{M}|}{|\Hom(S_1^{\oplus(s-r)},S_2\oplus S_1^{\oplus(t-r)})| } \prod_{j=0}^{r-1}\frac{(q^s-q^j)(q^t-q^j)}{q^r-q^j}[M]*[\bK_{1}]^r
 \notag \\
=&\sum_{r=0}^{\min\{ s,t\}} \sum_{[M]\in \mathcal{I}_{s+t-2r}} \sqq^{ (2r-s)(a-t)-2r(s-r) -2(s-r)(t-r)}\prod_{j=0}^{r-1}\frac{(q^s-q^j)(q^t-q^j)}{q^r-q^j}
 \notag \\
&\qquad\cdot|\Ext^1(S_1^{\oplus(s-r)},S_2\oplus S_1^{\oplus(t-r)})_{M}|\cdot [M]*[\bK_{1}]^r,
\notag
\end{align}
since $|\Ext^1(S_1^{\oplus(s-r)},S_2\oplus S_1^{\oplus(t-r)})_{M}|\neq 0$ implies $[M]\in \mathcal{I}_{s+t-2r}$ by \eqref{eq:Ik}.
Recall $M\cong N\oplus S_1^{\oplus u_M}$ for some indecomposable $\K Q$-module $N$ from \eqref{eq:MNu}. From \cite{Rin} or \cite[Theorem 3.16]{Sch06} and its proof, recalling \eqref{eq:Fxyz} we have
$$F_{S_1^{\oplus (s-r)},S_2\oplus S_1^{\oplus(t-r)}}^{ M}= \sqq^{(u_M-(t-r))(t-r)} \qbinom{u_M}{t-r}_\sqq.$$
Using the Riedtman-Peng formula in Lemma \ref{lem: Ried-P}, one obtains that
\begin{align}  \label{eq:C3}
&\big|\Ext^1( S_1^{\oplus(s-r)}, S_2\oplus S_1^{\oplus(t-r)} )_{M} \big|\\
=& \frac{ \prod_{i=0}^{s-r-1}(q^{s-r}-q^i)\prod_{i=0}^{t-r-1}(q^{t-r}-q^i)}{|\aut(M)|}(q-1)\sqq^{(u_M-(t-r))(t-r)+2(s-r)(t-r)} \qbinom{u_M}{t-r}_\sqq.
\notag
\end{align}

Thus using \eqref{eq:C3}, we rewrite \eqref{eq:basic} as
\begin{align}
&[s S_1] *[S_2]*[ t S_1 ]
\label{eq:SSS2}
\\\notag
=&\sum_{r=0}^{\min\{ s,t\}}\sum_{[M]\in \mathcal{I}_{s+t-2r}}\sqq^{-s(a+t)+2ra+(u_M-t+2s-r)(t-r) }(q-1)   \prod_{j=0}^{r-1}\frac{(q^{s}-q^j)(q^{t}-q^j)}{q^r-q^j} \\
&\qquad\qquad\cdot
\frac{ \prod_{i=0}^{s-r-1}(q^{s-r}-q^i)\prod_{i=0}^{t-r-1}(q^{t-r}-q^i)}{|\aut(M)|}\qbinom{u_M}{t-r}_\sqq[M]*[\bK_1]^r.\notag
\end{align}

Recall $q=\sqq^2$. Note that
\begin{align*}
\prod_{j=0}^{r-1} (q^r -q^j) &=\sqq^{r^2 +{r \choose 2}} (\sqq -\sqq^{-1})^r [r]_\sqq^{!},
\\
\prod_{j=0}^{r-1} (q^s -q^j) &=\sqq^{rs +{r \choose 2}} (\sqq -\sqq^{-1})^r [s]_\sqq [s-1]_\sqq \ldots [s-r+1]_\sqq,
\\
\prod_{i=0}^{s-r-1} (q^{s-r} -q^i) &=\sqq^{(s-r)^2 +{s-r \choose 2}} (\sqq -\sqq^{-1})^{s-r} [s-r]_\sqq^{!}.
\end{align*}
These identities (and the counterparts of the last 2 identities with $s$ replaced by $t$) allow us to convert the formula \eqref{eq:SSS2} to \eqref{eq:SSS} by a direct computation. This proves the proposition.
\end{proof}

\subsection{$\imath$Serre relation in $\tMHk$}

It is well known that $\widetilde{\ch}(\K Q^{\text{op}}, \btau) =\widetilde{\ch}(\K Q, \btau)^{\text{op}}$. Hence the identity \eqref{eqn:iserre2} is equivalent to \eqref{eqn:iserre1}. It remains to prove \eqref{eqn:iserre1}.

In this subsection, we shall prove \eqref{eqn:iserre1} (and hence Theorem~\ref{thm: iserre relation}), modulo the validity of a combinatorial identity (which will be established in Section~\ref{sec:comb}). Note that the identity \eqref{eqn:iserre1} can be rewritten as
\begin{align}
\sum_{n=0}^{1+a} (-1)^n  [S_1]_{\ev}^{(n)}*[S_2] *[S_1]_{\ov{a}}^{(1+a-n)} &=0,\label{eqn:iserre1ev}
\\
\sum_{n=0}^{1+a} (-1)^n  [S_1]_{\odd}^{(n)}*[S_2] *[S_1]_{\odd +\ov{a}}^{(1+a-n)} &=0.\label{eqn:iserre1odd}
\end{align}
We will provide a detailed proof of \eqref{eqn:iserre1ev}, which will be modified to give a proof of \eqref{eqn:iserre1odd}.

\subsubsection{Proof of \eqref{eqn:iserre1ev} }
  \label{subsec:proofev}

We divide the computation of the LHS of \eqref{eqn:iserre1ev} into 2 cases.

\underline{\em Case (I): $n$ even}.
By Proposition~\ref{prop:iDPev} and Proposition~\ref{prop:SSSM} we have
\begin{align*}
 &[S_1]_{\ev}^{(n)}*[S_2] *[S_1]_{\ov{a}}^{(1+a-n)} \\
=&\sum_{k=0}^{\frac{n}{2}}\sum_{m=0}^{\lfloor \frac{a+1-n}{2} \rfloor} \frac{\sqq^{k(k-1)+m(m+1)-\binom{n-2k}{2} -\binom{1+a-n-2m}{2}}\cdot (\sqq-\sqq^{-1})^{k+m} }{[n-2k]_\sqq![1+a-n-2m]_\sqq![2k]_\sqq^{!!}[2m]_\sqq^{!!}}\\
&\qquad\qquad \times [(n-2k)S_1]*[S_2]*[(1+a-n-2m)S_1]*[\bK_1]^{k+m}
\\
=&\sum_{k=0}^{\frac{n}{2}}\sum_{m=0}^{\lfloor \frac{a+1-n}{2} \rfloor} \sum_{r=0}^{\min\{n-2k,1+a-n-2m\}} \sum_{[M]\in\mathcal{I}_{1+a-2k-2m-2r}}
\frac{\sqq^{k(k-1)+m(m+1)-\binom{n-2k}{2} -\binom{1+a-n-2m}{2}} }{[n-2k]_\sqq^![1+a-n-2m]_\sqq^! [2k]_\sqq^{!!}[2m]_\sqq^{!!}}\\
&\qquad\times (\sqq-\sqq^{-1})^{k+m} \sqq^{p(a,r,n-2k, 1+a -n -2m)}
(\sqq -\sqq^{-1})^{2+a -2k -2m -r } \frac{[n-2k]_\sqq^{!} [1+a -n -2m]_\sqq^{!}}{[r]_\sqq^{!}}
\\ &
\qquad \times \qbinom{u_M}{1+a -n -2m-r}_\sqq \frac{[M]}{|\aut(M)|}*[\bK_1]^{r+k+m}.
\end{align*}
This can be simplified to be, for $n$ even,
\begin{align}
&[S_1]_{\ev}^{(n)}*[S_2] *[S_1]_{\ov{a}}^{(1+a-n)}
=\sum_{k=0}^{\frac{n}{2}}\sum_{m=0}^{\lfloor \frac{a+1-n}{2} \rfloor} \sum_{r=0}^{n-2k} \sum_{[M]\in\mathcal{I}_{1+a-2k-2m-2r}}
\label{sum:nev}
\\
& \qquad\qquad
\frac{\sqq^{z}
(\sqq-\sqq^{-1})^{2+a -k -m -r} }{[r]_\sqq^{!} [2k]_\sqq^{!!}[2m]_\sqq^{!!}}
\qbinom{u_M}{1+a -n -2m-r}_\sqq \frac{[M] *[\bK_1]^{r+k+m}}{|\aut(M)|}
\notag
\end{align}
where we denote (recall the polynomial $p$ from \eqref{eq:p})
\begin{align}
\label{eq:z}
z&= k(k-1)+m(m+1)-\binom{n-2k}{2} -\binom{1+a-n-2m}{2}
\\
&\qquad\qquad\quad
+p(a,r,n-2k, 1+a -n -2m).
\notag
\end{align}

\underline{\em Case (II):  $n$  odd}. By Proposition~\ref{prop:iDPev} and Proposition~\ref{prop:SSSM} we have
\begin{align*}
&[S_1]_{\ev}^{(n)}*[S_2] *[S_1]_{\ov{a}}^{(1+a-n)} \\
=&\sum_{k=0}^{\frac{n-1}{2}}\frac{\sqq^{k(k+1)-\binom{n-2k}{2}} \cdot (\sqq-\sqq^{-1})^k }{[n-2k]_\sqq![2k]_\sqq^{!!}} [(n-2k)S_1]*[\bK_1]^k*[S_2]
\\
&\;\; * \sum_{m=0}^{\lfloor \frac{a+1-n}{2} \rfloor} \frac{\sqq^{m(m-1)-\binom{1+a-n-2m}{2}}\cdot (\sqq-\sqq^{-1})^{m}}{[1+a-n-2m]_\sqq^{!}[2m]_\sqq^{!!}} [(1+a-n-2m)S_1]*[\bK_1]^m
\\
=&\sum_{k=0}^{\frac{n-1}{2}}\sum_{m=0}^{\lfloor \frac{a+1-n}{2} \rfloor} \sum_{r=0}^{\min\{n-2k,1+a-n-2m\}} \sum_{[M]\in\mathcal{I}_{1+a-2k-2m-2r}}
\\
&\quad\quad \frac{\sqq^{k(k+1)+m(m-1)-\binom{n-2k}{2} -\binom{1+a-n-2m}{2}}\cdot (\sqq-\sqq^{-1})^{k+m} }{[n-2k]_\sqq![1+a-n-2m]_\sqq![2k]_\sqq^{!!}[2m]_\sqq^{!!}}\\
&\quad\quad \times \sqq^{p(a,r,n-2k, 1+a -n -2m)}
(\sqq -\sqq^{-1})^{2+a -2k -2m -r } \frac{[n-2k]_\sqq^{!} [1+a -n -2m]_\sqq^{!}}{[r]_\sqq^{!}}
\\ &
\quad\quad \times \qbinom{u_M}{1+a -n -2m-r}_\sqq \frac{[M]}{|\aut(M)|}*[\bK_1]^{r+k+m}.
\end{align*}
This can be simplified to be, for $n$ odd,
\begin{align}
&[S_1]_{\ev}^{(n)}*[S_2] *[S_1]_{\ov{a}}^{(1+a-n)}
= \sum_{k=0}^{\frac{n-1}{2}}\sum_{m=0}^{\lfloor \frac{a+1-n}{2} \rfloor} \sum_{r=0}^{n-2k} \sum_{[M]\in\mathcal{I}_{1+a-2k-2m-2r}}
 \label{sum:nodd} \\
& \qquad\qquad \frac{\sqq^{z+2k-2m}
 (\sqq-\sqq^{-1})^{2+a -k -m -r} }{[r]_\sqq^{!} [2k]_\sqq^{!!}[2m]_\sqq^{!!}}
\qbinom{u_M}{1+a -n -2m-r}_\sqq \frac{[M] *[\bK_1]^{r+k+m}}{|\aut(M)|}.
\notag
\end{align}

Summing up \eqref{sum:nev} and \eqref{sum:nodd} above, we obtain
\begin{align}
&\sum_{n=0}^{a+1} (-1)^n  [S_1]_{\ev}^{(n)}*[S_2] *[S_1]_{\ov{a}}^{(1+a-n)}
\label{eq:1-2} \\
=& \sum_{n=0,2\mid n}^{a+1}
\sum_{k=0}^{\frac{n}{2}}\sum_{m=0}^{\lfloor \frac{a+1-n}{2} \rfloor} \sum_{r=0}^{n-2k} \sum_{[M]\in\mathcal{I}_{1+a-2k-2m-2r}}
\notag \\ &\qquad\qquad
\frac{\sqq^{z} (\sqq-\sqq^{-1})^{2+a -k -m -r} }{[r]_\sqq^{!} [2k]_\sqq^{!!}[2m]_\sqq^{!!}}
\qbinom{u_M}{1+a -n -2m-r}_\sqq
 \frac{[M] *[\bK_1]^{r+k+m}}{|\aut(M)|}
\notag \\
&-\sum_{n=0,2\nmid n}^{a+1}
\sum_{k=0}^{\frac{n-1}{2}}\sum_{m=0}^{\lfloor \frac{a+1-n}{2} \rfloor} \sum_{r=0}^{n-2k} \sum_{[M]\in\mathcal{I}_{1+a-2k-2m-2r}}
\notag \\ &\qquad\qquad
\frac{\sqq^{z+2k-2m}  (\sqq-\sqq^{-1})^{2+a -k -m -r} }{[r]_\sqq^{!} [2k]_\sqq^{!!}[2m]_\sqq^{!!}}
\qbinom{u_M}{1+a -n -2m-r}_\sqq
\frac{[M] *[\bK_1]^{r+k+m}}{|\aut(M)|}.
\notag
\end{align}

Set
\[
d =r+k+m.
\]
Now we have reduced the proof of \eqref{eqn:iserre1ev} to proving that the coefficient of $\frac{[M] *[\bK_1]^{d}}{|\aut(M)|}$ in the RHS of \eqref{eq:1-2} is zero, for any given $[M]   \in\mathcal{I}_{1+a-2d}$ and any $d\in \N$. Note the powers of $(\sqq-\sqq^{-1})$ in all terms are the same (and $= 2 +a -d$).
Denote
\begin{align}
T(a,d,u)
&= \sum_{n=0,2\mid n}^{a+1}
\sum_{k=0}^{\frac{n}{2}}\sum_{m=0}^{\lfloor \frac{a+1-n}{2} \rfloor} \delta\{0\le r \le n-2k\}
\frac{\sqq^{\brown{z}} }{[r]_\sqq^{!} [2k]_\sqq^{!!}[2m]_\sqq^{!!}}
\qbinom{u}{1+a -n -2m-r}_\sqq
  \label{def:T} \\
&-\sum_{n=0,2\nmid n}^{a+1}
\sum_{k=0}^{\frac{n-1}{2}}\sum_{m=0}^{\lfloor \frac{a+1-n}{2} \rfloor} \delta\{0\le r \le n-2k\}
\frac{\sqq^{\brown{z+2k-2m}} }{[r]_\sqq^{!} [2k]_\sqq^{!!}[2m]_\sqq^{!!}}
\qbinom{u}{1+a -n -2m-r}_\sqq,
\notag
\end{align}
where we set $\delta\{X\}=1$ if the statement $X$ holds and $\delta\{X\}=0$ if $X$ is false. We note $r=d-k-m \geq0$; see \eqref{eq:z} for $z$, and also see \eqref{eq:p} for the polynomial $p$.

 Then the coefficient of $\frac{[M] *[\bK_1]^{d}}{|\aut(M)|}$ in the RHS of \eqref{eq:1-2} is equal to $(\sqq-\sqq^{-1})^{2 +a -d} T(a,d,u)$. Summarizing the above discussions, we have established the following.

\begin{proposition}
The identity \eqref{eqn:iserre1ev} is equivalent to the identity
$
T(a,d,u) =0,
$ 
for any integers $a, d, u$ subject to the constrains
\begin{equation}
  \label{eq:adu}
a \geq0, \quad 0\leq d\leq (a+1)/2, \quad 0\leq u\leq a+1-2d, \quad d \text{ and } u \text{ not both zero}.
\end{equation}
\end{proposition}

\subsubsection{Proof of \eqref{eqn:iserre1odd} }
  \label{subsec:proofodd}

Note the differences on the formulas for $[S]^{(n)}_{\ev}$ versus $[S]^{(n)}_{\odd}$ in Proposition~\ref{prop:iDPev}--\ref{prop:iDPodd} merely lie in the powers of $\sqq$. Going through the same computations in \S\ref{subsec:proofev}, we see that the identity \eqref{eqn:iserre1odd} is equivalent to the following identity
\[
T_1(a, d, u)=0,
\]
for $a, d, u$ satisfying \eqref{eq:adu}, where $T_1$ is modified from $T$ in \eqref{def:T} by changing the power of $\sqq$ in the first summand from $\brown{z}$ to $\blue{z+2k-2m}$ and the power of $\sqq$ in the second summand from $\brown{z+2k-2m}$ to $\blue{z}$.

We shall establish the identities $T(a, d, u) =0$ and $T_1(a, d, u)=0$ in the next section.

\section{Combinatorial identities}
  \label{sec:comb}

The goal of this section is to prove the following identities (and hence complete the proof of Theorem~\ref{thm: iserre relation}). In the process, we establish some interesting $v$-binomial identities, which are of independent interest.

\begin{proposition}
  \label{prop:T=0}
For integers $a, d, u$ satisfying \eqref{eq:adu}, the following identities hold:
\begin{align}
T(a,d,u) &=0, \label{eq:T=0}
\\
T_1(a,d,u) &=0, \label{eq:T1=0}
\end{align}
where $T$ is defined in \eqref{def:T} and $T_1$ is defined in \S\ref{subsec:proofodd}.
\end{proposition}

\subsection{Some $v$-binomial identities}

We first establish some identities which will be used later.

\begin{lemma}
 \label{lem:binomial}
The following (equivalent) identities hold, for $p \in \N$:
\begin{align}
[p]^!\sum_{\stackrel{k,m \in \N}{k+m =p}} \frac{v^{-2(k-1)m -\frac{p(3-p)}2} }{[2k]^{!!} [2m]^{!!}} &= 1,
  \label{eq:km1}
\\
\sum_{k=0}^{p} v^{\frac{p(p+1)}2-2k(p-k+1)} \qbinom{p}{k}_{v^2} & =\frac{[2p]^{!!}}{[p]^!}.
  \label{eq:km3}
\end{align}
\end{lemma}

\begin{proof}
Clearly the 2 identities \eqref{eq:km1}--\eqref{eq:km3} are equivalent, by noting that $[2k]^{!!} =[2]^k [k]_{\sqq^2}^{!}$ and $\qbinom{p}{k}_{v^2} = \frac{[2p]^{!!}}{[2k]^{!!} [2m]^{!!}}$ with $m=p-k$.

By switching $k$ to $p-k$ and noting $\qbinom{p}{p-k}_{v^2} =\qbinom{p}{k}_{v^2} $, we see that the identity \eqref{eq:km3} is equivalent to
\begin{align}\sum_{k=0}^{p} v^{\frac{p(p+1)}2-2(k+1)(p-k)} \qbinom{p}{k}_{v^2} & =\frac{[2p]^{!!}}{[p]^!}.
  \label{eq:km4}
\end{align}

It remains to prove \eqref{eq:km3} by induction on $p$. It is clear when $p=0$. Assuming the statement for $p$ \eqref{eq:km3} (and its equivalent \eqref{eq:km4}), we shall prove
\[
\sum_{k=0}^{p+1} v^{\frac{(p+1)(p+2)}2-2k(p-k+2)} \qbinom{p+1}{k}_{v^2} =\frac{[2(p+1)]^{!!}}{[p+1]^!}.
\]

Indeed, using the $v$-binomial identity
$\qbinom{p+1}{k}_{v^2} =v^{2k} \qbinom{p}{k}_{v^2} +v^{-2(p+1-k)} \qbinom{p}{k-1}_{v^2},$
we have
\begin{align*}
& \sum_{k=0}^{p+1} v^{\frac{(p+1)(p+2)}2-2k(p-k+2)} \qbinom{p+1}{k}_{v^2}
\\
&= \sum_{k=0}^{p+1} v^{\frac{(p+1)(p+2)}2-2k(p-k+2)}  \left(v^{2k} \qbinom{p}{k}_{v^2} +v^{-2(p-k+1)} \qbinom{p}{k-1}_{v^2} \right)
\\
&= \sum_{k=0}^{p} v^{\frac{(p+1)(p+2)}2-2k(p-k+1)} \qbinom{p}{k}_{v^2} +
\sum_{k=1}^{p+1} v^{\frac{(p+1)(p+2)}2-2(k+1)(p-k+2)+2}  \qbinom{p}{k-1}_{v^2}
\\
&\stackrel{(*)}{=} \sum_{k=0}^{p} v^{\frac{(p+1)(p+2)}2-2k(p-k+1)} \qbinom{p}{k}_{v^2} +
\sum_{k=0}^{p} v^{\frac{(p+1)(p+2)}2-2(k+2)(p-k+1)+2}  \qbinom{p}{k}_{v^2}
\\
&= v^{p+1} \sum_{k=0}^{p} v^{\frac{p(p+1)}2-2k(p-k+1)} \qbinom{p}{k}_{v^2} +
 v^{-p-1} \sum_{k=0}^{p} v^{\frac{p(p+1)}2-2(k+1)(p-k)}  \qbinom{p}{k}_{v^2}
\\
&\stackrel{(**)}{=} v^{p+1} \frac{[2p]^{!!}}{[p]^!} + v^{-p-1} \frac{[2p]^{!!}}{[p]^!}
= \frac{[2(p+1)]^{!!}}{[p+1]^!},
\end{align*}
where the identity $(*)$ is obtained by shifting the index $k$ in the second summand on the LHS to $k+1$, and $(**)$ uses the inductive assumption \eqref{eq:km3}--\eqref{eq:km4}.
\end{proof}
Identity \eqref{eq:km3} (after a rescaling $v^2 \mapsto v$) can be further reformulated as the following identity (also compare \cite[Ex. 5, pp.49]{An98}):
\begin{align}
  \label{eq:km5}
\sum_{k=0}^{p} v^{-k(p-k+1)} \qbinom{p}{k}  & =  \prod_{j=1}^p (1+v^{-j}).
\end{align}

\begin{lemma}
 \label{lem:multinomial}
The following identity holds, for $d \ge 1$:
\begin{align}  \label{eq:kmrd}
\sum_{\stackrel{k,m,r \in \N}{k+m+r =d}} (-1)^r \frac{v^{{r+1 \choose 2} -2(k-1)m}}{[r]^! [2k]^{!!} [2m]^{!!}} =0.
\end{align}
\end{lemma}

\begin{proof}
Using \eqref{eq:km1} we have
\begin{align*}
\sum_{\stackrel{k,m,r \in \N}{k+m+r =d}} (-1)^r \frac{v^{{r+1 \choose 2} -2(k-1)m}}{[r]^! [2k]^{!!} [2m]^{!!}}
&= \sum_{r=0}^d (-1)^r \frac{v^{{r+1 \choose 2}} \cdot v^{\frac{(d-r)(3-d+r)}{2}}}{[r]^![d-r]^!}
\\
& = \frac{v^{\frac{3d-d^2}2}}{[d]^!} \sum_{r=0}^d (-1)^r  v^{(d-1)r} \qbinom{d}{r} =0.
\end{align*}
In the last step above, we have used the standard $v$-binomial formula \eqref{eq:BF}.
\end{proof}

\subsection{Proof of Identity \eqref{eq:T=0} }
 \label{subsec:proofT}

It is crucial for our purpose to introduce a new variable
\[
w=n +m -k -d
\]
in place of $n$ in \eqref{def:T}.
Hence we have $\qbinom{u}{1+a -n -d-m+k}_\sqq =\qbinom{u}{1+a -2d-w}_\sqq$ and
\begin{equation}  \label{eq:n-w}
n=w-m +k+d \equiv w+r \pmod 2.
\end{equation}
The condition $r \le n-2k$ in $T(a, d, u)$ in \eqref{def:T} is transformed into the condition $w\ge 0$.

By a direct computation we can rewrite $z$ in \eqref{eq:z} as
\begin{equation}  \label{eq:z2}
 z={r+1 \choose 2} -2(k-1)m + L,
\end{equation}
where
\[
L 
 =d(d-1)
 -a w  +(u-a+2d+2w)(1+a-2d-w) +w^2  +1.
\]
(We do not need the precise formula for $L$ except noting that $L$ is independent of $k,m,r$, and only depends on $a, d, w, u$.)
Hence, for fixed $a, w, d, u$, using \eqref{eq:n-w}--\eqref{eq:z2} and Lemma~\ref{lem:multinomial}, we calculate that the contribution to the coefficient of $\qbinom{u}{1+a -2d-w}_\sqq$ in $T(a, d, u)$ in \eqref{def:T}, for $d>0$, is equal to
\begin{align*}
&(-1)^w \sqq^L \left( \sum_{\stackrel{k,m,r \in \N, r \; {\rm even}}{k+m+r =d}} (-1)^r \frac{v^{{r+1 \choose 2} -\brown{2(k-1)m}}}{[r]^! [2k]^{!!} [2m]^{!!}}
+  \sum_{\stackrel{k,m,r \in \N, r \; {\rm odd}}{k+m+r =d}}  (-1)^r \frac{v^{{r+1 \choose 2} -\brown{2k(m-1)}}}{[r]^! [2k]^{!!} [2m]^{!!}} \right)
\\
&\stackrel{(*)}{=}  (-1)^w \sqq^L \sum_{\stackrel{k,m,r \in \N}{k+m+r =d}} (-1)^r \frac{v^{{r+1 \choose 2} -2(k-1)m}}{[r]^! [2k]^{!!} [2m]^{!!}}
=0.
\end{align*}
Note that the identity $(*)$ above is obtained by switching notation $k \leftrightarrow m$ in the second summand on the LHS of $(*)$. Therefore, we have obtained that $T(a, d, u)=0$, for $d>0$.

It remains to determine the contributions of the terms with $d=0$ to $T(a, 0, u)$ in \eqref{def:T}, for fixed $a,  u$;  recall from \eqref{eq:adu} that $u>0$ when $d=0$. In this case, we have $k=m=r=0$, and a direct computation shows that the power $z$ can be simplified to be $z =(1-u)w +(1+a)u +1$. Then, for $0<u \leq a+1$, we have
\begin{align*}
T(a, 0, u)
&= \sqq^{(1+a)u +1} \sum_{w \ge 0}  (-1)^w \sqq^{(1-u)w} \qbinom{u}{1+a -w}
\\
&\stackrel{(1)}{=} (-1)^{1+a} \sqq^{(1+a)+1}\sum_{x\ge 0} (-1)^x \sqq^{(u-1)x}  \qbinom{u}{x}
\stackrel{(2)}{=}  0,
\end{align*}
 where we have changed variables $x=1+a-w$ in the identity $(1)$, and used the $v$-binomial formula \eqref{eq:BF} in $(2)$ above.

Therefore, we have established the identity \eqref{eq:T=0}.

\subsection{Proof of Identity \eqref{eq:T1=0} }
 \label{subsec:proofT1}

 The proof is essentially the same as the proof in \S\ref{subsec:proofT} for the identity \eqref{eq:T=0}, with some   modification of details below.

 Going through \S\ref{subsec:proofT}, we calculate that the contribution to the coefficient of $\qbinom{u}{1+a -2d-w}_\sqq$ in $T_1(a, d, u)$ (see \S\ref{subsec:proofodd} for definition of $T_1$), for $d>0$, is equal to
\begin{align*}
&(-1)^w \sqq^L \left( \sum_{\stackrel{k,m,r \in \N, r \; {\rm even}}{k+m+r =d}} (-1)^r \frac{v^{{r+1 \choose 2} - \blue{2k(m-1)}} }{[r]^! [2k]^{!!} [2m]^{!!}}
+  \sum_{\stackrel{k,m,r \in \N, r \; {\rm odd}}{k+m+r =d}}  (-1)^r \frac{v^{{r+1 \choose 2} -\blue{2(k-1)m} }}{[r]^! [2k]^{!!} [2m]^{!!}} \right)
\\
&= (-1)^w \sqq^L \sum_{\stackrel{k,m,r \in \N}{k+m+r =d}} (-1)^r \frac{v^{{r+1 \choose 2} -2(k-1)m }}{[r]^! [2k]^{!!} [2m]^{!!}}
=0.
\end{align*}
Therefore, we obtain that $T_1(a, d, u)=0$, for $d>0$.
In exactly the same way as in \S\ref{subsec:proofT}, we see $T_1(a, 0, u)=0$, for $0<u \leq a+1$.
This proves the identity \eqref{eq:T1=0}.

Hence the proofs of Proposition~\ref{prop:T=0} and then of Theorem~\ref{thm: iserre relation} are completed.

\section{$\imath$Hall algebras and $\imath$quantum groups}
  \label{sec:iQGHall}

In this section, we establish several more identities in the $\imath$Hall algebras corresponding to the relations \eqref{relation1}, \eqref{relation3} and a remaining part of \eqref{relation6}. Then we prove the main theorem which provides a Hall algebra realization of the $\imath$quantum groups.
\subsection{Relation \eqref{relation1} }

Recall the Euler form $\langle \cdot, \cdot \rangle_Q$ is used in the twisted product of the $\imath$Hall algebra $\tMHk$.


We have the following identities in $\tMHk$ corresponding to the relation \eqref{relation1} in $\tUi$.

\begin{proposition}
\label{prop:relation1}
Let $(Q,\btau)$ be an $\imath$quiver. Then the following identities hold in $\tMHk$, for $i,j \in \I$:
\begin{align*}
[\bK_i]*[S_j] &= \sqq^{ c_{\btau i,j}-c_{ij}}  [S_j]*[\bK_i],\\
[\bK_i]*[\bK_j] &=[\bK_j]*[\bK_i].
\end{align*}
\end{proposition}

\begin{proof}
By Lemma \ref{lem:Euler}, we have
\begin{align*}
[\bK_i]*[S_{j}]&=\sqq^{\langle \res(\bK_i),\res (S_{j})\rangle_Q} q^{-\langle \bK_i, S_{j}\rangle}[\bK_i\oplus S_{j}]\\
&=\sqq^{\langle S_{\btau i},S_{j}\rangle_Q-\langle S_i,S_{j}\rangle_Q}[\bK_i\oplus S_j],
\\
[S_{j}]*[\bK_i] &=\sqq^{\langle S_{j},S_i\rangle_Q-\langle S_{j},S_{\btau i}\rangle_Q}[\bK_i\oplus S_{j}].
\end{align*}
Hence we have
\begin{align*}
[\bK_i]*[S_{j}]&= \sqq^{\langle S_{\btau i},S_{j}\rangle_Q-\langle S_i,S_{j} \rangle_Q-\langle S_{j},S_i\rangle_Q+\langle S_{j},S_{\btau i}\rangle_Q}[S_{j}]*[\bK_i]\\
&= \sqq^{(S_{\btau i},S_{j})-(S_i,S_{j})} [S_{j}]*[\bK_i]\\
&= \sqq^{ c_{\btau i,j}-c_{ij}}  [S_j]*[\bK_i].
\end{align*}
This proves the first formula. The second formula follows from \eqref{Eform3}.
\end{proof}

\subsection{Relation \eqref{relation3} }

We first recall the usual Serre relation in the twisted Ringel-Hall algebra associated to $Q$ over $\K$, denoted by $(\widetilde{\ch}(\K Q), *)$.

\begin{lemma}[\cite{Rin, Gr95}]
\label{lem: Rin KM}
Let $Q= \xymatrix{1\ar@<1ex>[r]^{\alpha_1}_{\cdots} \ar@<-1ex>[r]_{\alpha_{a}}  & 2}$.
The following identity holds in $\widetilde{\ch}(\K Q)$, for $i\neq j \in \I$:
\begin{align*}
&\sum_{r+s=a+1}(-1)^r[S_i]^{(r)} * [S_j] *[S_i]^{(s)}=0.
\end{align*}
\end{lemma}

Recall the definition of virtually acyclic $\imath$quivers from Definition~\ref{def:i-acyclic}.

\begin{proposition}
\label{prop:SerreU}
Let $(Q,\btau)$ be a virtually acyclic $\imath$quiver. The following identity holds in $\tMHk$, for any $i\neq j\in\I$ such that $\btau i\neq i$:
\begin{align*}
&\sum_{r+s= 1-c_{ij}}(-1)^r[S_i]^{(r)} * [S_j] *[S_i]^{(s)}=0.
\end{align*}
\end{proposition}

\begin{proof}
We have
$[S_i]^{(r)} * [S_j] *[S_i]^{(s)}=\sum_{[M]\in\mod(\Lambda^\imath)} \bp_M[M]$. If $\bp_M\neq0$, then
$M\in\mod(\K Q)$. So it boils down to the same computation as computing $[S_i]^{(r)} * [S_j] *[S_i]^{(s)}$ in $\widetilde{\ch}(\K Q)$.
Therefore the proposition follows from Lemma~ \ref{lem: Rin KM}.
\end{proof}

\subsection{Relation \eqref{relation6} for $j\neq \btau j$}

Let $Q= \xymatrix{1\ar@<1ex>[r]^{\alpha_1}_{\cdots} \ar@<-1ex>[r]_{\alpha_{a}}  & 2 &3\ar@<-1ex>[l]_{\beta_1}^{\cdots} \ar@<1ex>[l]^{\beta_{a}}  }$ with involution $\btau$ given by $\btau 1=3$ and $\btau 2=2$, where $a=-c_{12}$. Then the quiver $\ov{Q}$ of $\Lambda^\imath$ is
\begin{center}\setlength{\unitlength}{0.8mm}
 \begin{equation}
 \label{diag:qsplit KM}
 \begin{picture}(50,10)(0,-10)
\put(0,-2){$1$}
\put(20,-2){$3$}
\put(-1,-11){$_{\alpha_a}$}
\put(7,-9){$_{\alpha_1}$}
\put(12,-9){$_{\beta_1}$}
\put(19,-11){$_{\beta_a}$}
\put(3.7,-10){$\cdot$}
\put(4.2,-11){$\cdot$}
\put(4.7,-12){$\cdot$}

\put(0,-3){\vector(1,-2){8}}
\put(2.5,-2){\vector(1,-2){8}}
\put(19.5,-2){\vector(-1,-2){8}}
\put(16.7,-10){$\cdot$}
\put(16.2,-11){$\cdot$}
\put(15.7,-12){$\cdot$}
\put(22,-3){\vector(-1,-2){8}}
\put(10,-22){$2$}
\color{purple}
\put(3,1){\vector(1,0){16}}
\put(19,-1){\vector(-1,0){16}}
\put(10,1){$^{\varepsilon_1}$}
\put(10,-4){$_{\varepsilon_3}$}
\put(10,-28){$_{\varepsilon_2}$}
\begin{picture}(50,23)(-10,19)
\color{purple}
\qbezier(-1,-1)(-3,-3)(-2,-5.5)
\qbezier(-2,-5.5)(1,-9)(4,-5.5)
\qbezier(4,-5.5)(5,-3)(3,-1)
\put(3.1,-1.4){\vector(-1,1){0.3}}
\end{picture}
\end{picture}
\vspace{1.4cm}
\end{equation}
\end{center}

The following is a variant of Theorem~\ref{thm: iserre relation} and will be derived from it.

\begin{proposition}
\label{prop:relation62}
Let $\Lambda^\imath$ be the $\imath$quiver algebra with its  quiver (or opposite quiver) given by \eqref{diag:qsplit KM}.
Then the following identities hold in $\tMHk$, for any $\overline{p}\in\Z/2\Z$ and $j=1,3$:
\begin{align}
\sum_{n=0}^{1+a} (-1)^n  [S_2]_{\overline{p}}^{(n)}*[S_j] *[S_2]_{\overline{a}+\overline{p}}^{(1+a-n)} =0.
\label{eqn:iserre3}
\end{align}
\end{proposition}

\begin{proof}
Set $i=2$. It suffices to prove the case when $j=1$.
Consider the full subquiver $Q'$ of $\ov{Q}$ formed by vertices $1$ and $2$. Let $'\Lambda^\imath:= \K Q'/(\varepsilon_2^2)$. We have
$[S_i]^{(r)} * [S_j] *[S_i]^{(s)}=\sum_{[M]} \bp_M[M]$. For any $[M]\in\mod(\Lambda^\imath)$ such that $\bp_M\neq0$, we have
$M\in\mod('\Lambda^\imath)$. In this proof, we shall denote the opposite quiver in \eqref{diag:split2} by $Q''$ and its $\imath$quiver algebra (i.e., the one in Theorem \ref{thm: iserre relation}) by $''\Lambda^\imath$. Then $'\Lambda^\imath$ can be viewed as a quotient algebra (and also a subalgebra) of $''\Lambda^\imath$ naturally. So it is the same computation as computing $[S_i]^{(r)} * [S_j] *[S_i]^{(s)}$ in ${\widetilde{\ch}(\K Q'',\Id)}$. Therefore the proposition follows from Theorem~ \ref{thm: iserre relation}.
%
\end{proof}

\subsection{$\imath$Hall algebra realization of $\tUi$}

Let $\tU^{\imath 0}$ be the $\Q(v)$-subalgebra of $\tUi$ generated by $\tk_i$, for $i\in \I$.
By the Serre presentation of $\tUi$ (see Theorem~\ref{thm:Serre}), letting $\deg B_i=\alpha_i$ and $\deg \tk_i=0$, for $i\in\I$, endows $\tUi$ a $\N\I$-filtered algebra structure. Let $\tU^{\imath,\gr}$ be the associated graded algebra. Then by Theorem~ \ref{thm:Serre} and the PBW theorem for $\tUi$,
there exists a natural algebra monomorphism
$\phi: \U^-\longrightarrow \tU^{\imath,\gr}$ by mapping
$F_i\mapsto B_i$ for any $i\in\I$. Moreover, $\tU^{\imath,\gr}=\Im\phi \cdot \widetilde{\bU}^{\imath 0}$.

Let $(Q,\btau)$ be a virtually acyclic $\imath$quiver. Recall $\I_\btau$ from \eqref{eq:ci}.
The following result due to Ringel and Green is well known (except that we follow Bridgeland's Hall multiplication here). 

\begin{lemma}[\cite{Rin,Gr95}; cf. \cite{LW19a}]
\label{lem: Green}
There exists an algebra monomorphism
\begin{align*}
&R: \U^-|_{v=\sqq}\longrightarrow   \widetilde{\ch}(\K Q)  \\
F_j \mapsto &\frac{-1}{q-1}[S_{j}],\text{ if } j\in\ci,\qquad F_j \mapsto \frac{{\sqq}}{q-1}[S_{j}],\text{ if }j\notin \ci.
\end{align*}
\end{lemma}

Recall from \cite[Lemma 5.3]{LW19a} that there is a filtered algebra structure on $\tMHk$, and we denote the associated graded algebra
\[
\widetilde{\ch}(\K Q, \btau)^{\gr} = \bigoplus_{\alpha \in K_0(\mod(\K Q))} \widetilde{\ch}(\K Q, \btau)_{\alpha}^{\gr}.
\]
It is natural to view the quantum torus $\widetilde{\ct}(\Lambda^\imath)$ (see the end of \S\ref{subsec:Hall basis}) as a subalgebra of $\widetilde{\ch}(\K Q, \btau)^{\gr}$. Then $\widetilde{\ch}(\K Q, \btau)^{\gr}$ is also a $\widetilde{\ct}(\Lambda^\imath)$-bimodule.

Just as in \cite[Lemma 5.4 (ii)]{LW19a}, the linear map
\begin{align}
\label{def:morph}
\varphi: \widetilde{\ch}(\K Q)\longrightarrow \widetilde{\ch}(\K Q, \btau)^{\gr},\quad
\varphi([M])=[M],\;  \forall M\in \mod(\K Q),
\end{align} is an embedding of algebras. Now we are ready to establish the main result of this paper.

\begin{theorem}
   \label{thm:main}
Let $(Q, \btau)$ be a virtually acyclic $\imath$quiver. Then there exists a $\Q(\sqq)$-algebra monomorphism
\begin{align*}
   \widetilde{\psi}: \tUi_{|v= \sqq} &\longrightarrow \tMHk,
\end{align*}
which sends
\begin{align}
B_j \mapsto \frac{-1}{q-1}[S_{j}],\text{ if } j\in\ci,
&\qquad\qquad
\tk_i \mapsto - q^{-1}[\bK_i], \text{ if }\btau i=i \in \I;
  \label{eq:split}
\\
B_{j} \mapsto \frac{{\sqq}}{q-1}[S_{j}],\text{ if }j\notin \ci,
&\qquad\qquad
\tk_i \mapsto \sqq^{\frac{-c_{i,\btau i}}{2}}[\bK_i],\quad \text{ if }\btau i\neq i \in \I.
  \label{eq:extra}
 \end{align}
\end{theorem}

\begin{proof} 

To show that $\widetilde{\psi}$ is a homomorphism, we verify that $\widetilde{\psi}$ preserves the defining relations \eqref{relation1}--\eqref{relation6} for $\tUi$.
According to Lemma \ref{lem:subalgebra}, the verification of the relations is local and hence is reduced to the rank 1 and rank 2 $\imath$quivers, which were treated in Section~\ref{sec:Relation5}, Section~\ref{sec:SerreRel} and earlier parts of this section.
More precisely, the relation \eqref{relation1} follows from Proposition~ \ref{prop:relation1}.
The relation \eqref{relation2} is obvious.
The relation \eqref{relation3} follows from Proposition~ \ref{prop:SerreU}.
The relation \eqref{relation5} follows from Proposition~\ref{prop:serre 1}.
Finally, the relation \eqref{relation6} follows from Theorem \ref{thm: iserre relation} and Proposition~ \ref{prop:relation62}.

The homomorphism $\widetilde{\psi}: \tUi_{|v={\sqq}} \rightarrow \tMHk$ restricts to an algebra homomorphism
\begin{align*}
\widetilde{\psi}:\widetilde{\bU}^{\imath 0}_{|v={\sqq}} &\longrightarrow  \tTL,
  \\
\tk_i\mapsto - q^{-1} [\bK_i], \text{ if }\btau i=i,
&\qquad
\tk_i\mapsto [\bK_i], \text{ if }\btau i\neq i.
\end{align*}
Since both $\widetilde{\bU}^{\imath 0}_{|v={\sqq}}$ and $\tTL$ are Laurent polynomial algebras in the same number of generators, $\widetilde{\psi}:\widetilde{\bU}^{\imath 0}_{|v={\sqq}} \rightarrow \tTL$ is an isomorphism.

It remains to prove that $\widetilde{\psi}: \tUi_{|v={\sqq}} \longrightarrow \tMHk$ is injective.
We observe that $\widetilde{\psi}$ is a morphism of filtered algebras. Let $\widetilde{\psi}^{\gr}:  \tU^{\imath,\gr}_{|v={\sqq}} \longrightarrow \widetilde{\ch}(\K Q, \btau)^{\gr}$ be its associated graded morphism, and we obtain the following commutative diagram
\begin{equation*}
\xymatrix{ \U^-|_{v=\sqq} \ar[r]^\phi \ar[d]^R &   \tU^{\imath,\gr}|_{v=\sqq} \ar[d]^{ \widetilde{\psi}^{\gr} }  \\
\widetilde{\ch}(Q) \ar[r]^\varphi & \widetilde{\ch}(\K Q, \btau)^{\gr} }
\end{equation*}
It follows that $\widetilde{\psi}^{\gr}\circ \phi$ is injective since $\varphi$ and $R$ are injective by Lemma \ref{lem: Green} and \eqref{def:morph}.

We claim that $\widetilde{\psi}^{\gr}$ is injective. Indeed, any element in $\tU^{\imath,\gr}$ is of form $\sum_{\alpha\in\Z\I}\phi(V_\alpha)\cdot \tk_\alpha$, for $V_\alpha\in \U^-$ . Here $\tk_\alpha=\prod_{i\in\I} \tk_i^{a_i}$ for $\alpha=\sum_{i\in\I} a_i \alpha_i$. Assume $\widetilde{\psi}^{\gr}(\sum_{\alpha}\phi(V_\alpha)\cdot \tk_\alpha)=0$, i.e.,  $\sum_{\alpha\in\Z\I}\widetilde{\psi}^{\gr}(\phi(V_\alpha))* \bK_\alpha=0$. Since $\widetilde{\ch}(\K Q, \btau)^{\gr}$ is graded, we obtain $\widetilde{\psi}^{\gr}(\phi(V_\alpha))* \bK_\alpha=0$ for any $\alpha$.
Together with Theorem \ref{thm:utMHbasis}, we obtain $\widetilde{\psi}^{\gr}(\phi(V_\alpha))=0$, and then $V_\alpha=0$. It follows that $\widetilde{\psi}^{\gr}$ is injective.

Now by a standard filtered algebra argument, we obtain that
$\widetilde{\psi}: \tUi_{|v={\sqq}} \longrightarrow \tMHk$ is an algebra monomorphism. The theorem is proved.
\end{proof}

\begin{remark}
We expect Theorem~ \ref{thm:main} to hold for general $\imath$quivers $(Q,\tau)$ without loops.

It will be interesting to develop a theory of quantum symmetric pairs $(\tU, \tUi)$ and $(\U, \Ui)$ associated to Borcherds-Cartan matrices (corresponding to quivers possibly with loops). We conjecture that a version of Theorem~ \ref{thm:main} holds for general $\imath$quivers with loops.
\end{remark}

\subsection{Variations}

The \emph{reduced Hall algebra associated to $(Q,\btau)$} (or {\em reduced $\imath$Hall algebra}), denoted by $\rMH$, is defined (cf. \cite{LW19a}) to be the quotient $\Q(v)$-algebra of $\tMHk$ by the ideal generated by the central elements
\begin{align*}
[\bK_i] +q \vs_i \; (\forall i\in \I \text{ with } \btau i=i), \text{ and }\; [\bK_i]*[\bK_{\btau i}] -\vs_i^2\; (\forall i\in \I \text{ with }\btau i\neq i).
\end{align*}

The following corollaries of Theorem~\ref{thm:main} are immediate.

\begin{corollary}
Let $(Q,\btau)$ be a virtually acyclic $\imath$quiver. Then there exists an injective homomorphism $\psi: \Ui_{|v={\sqq}}\longrightarrow \rMH$, which sends  $k_i \mapsto  \sqq^{\frac{-c_{i,\btau i}}{2}}  \frac{[\bK_i]}{\vs_{\btau i}},
B_i \mapsto \frac{-1}{q-1}[S_i],\text{ for } i \in\ci,$
 and $B_i \mapsto \frac{{\sqq}}{q-1}[S_i],\text{ for } i \notin \ci.$
\end{corollary}

Let $\tCMH$ be the $\Q(\sqq)$-subalgebra (called the {\em composition algebra}) of $\tMHk$ generated by $[S_i]$ and $[\bK_i]^{\pm 1}$, for $i\in\I$.

\begin{corollary}   \label{cor: composition}
Let $(Q,\btau)$ be a virtually acyclic $\imath$quiver. Then there exists an algebra isomorphism:
$
\widetilde{\psi}: \tUi_{|v={\sqq}} \stackrel{\cong}{\longrightarrow} \tCMH
$ 
given by \eqref{eq:split}--\eqref{eq:extra}.
\end{corollary}

Following Ringel, we define a {\em generic composition subalgebra} $\tCMHg$ below. Let $\bfK$  be an infinite set of (nonisomorphic) finite fields, and let us choose for each $\K\in\bfK$ an element $\sqq_\K\in\C$ such that $\sqq_\K^2=|\K|$. Consider the direct product $\tCMHg:=\prod_{\K\in\bfK} \tCMH.$ We view $\tCMHg$ as a $\Q(v)$-module by mapping $v$ to $(\sqq_\K)_\K$.
As in \cite{Rin,Gr95}, we have the following consequence of Corollary \ref{cor: composition}.

\begin{corollary}
Let $(Q,\btau)$ be a virtually acyclic $\imath$quiver. Then we have the following algebra isomorphism
$\widetilde{\psi}: \tUi\longrightarrow \tCMHg$ defined by
\begin{align*}
B_j \mapsto \Big(\frac{-1}{|\K|-1}[S_{j}] \Big)_\K,\text{ if } j\in\ci,
&\qquad\qquad
\tk_i \mapsto \Big(- |\K|^{-1}[\bK_i]\Big)_\K, \text{ if }\btau i=i;
\\
B_{j} \mapsto \Big(\frac{{\sqq_\K}}{|\K|-1}[S_{j}]\Big)_\K,\text{ if }j\notin \ci,
&\qquad\qquad
\tk_i \mapsto \Big(\sqq_\K^{\frac{-c_{i,\btau i}}{2}}[\bK_i]\Big)_\K,\quad \text{ if }\btau i\neq i.
\end{align*}
\end{corollary}



\begin{thebibliography}{AA1W18}

\bibitem[An98]{An98} G. Andrews,
{\em The theory of partitions},
The theory of partitions. Reprint of the 1976 original. Cambridge Mathematical Library. Cambridge University Press, Cambridge, 1998.





\bibitem[BK15]{BK15} M.~Balagovic and S.~Kolb,
{\em The bar involution for quantum symmetric pairs}, Represent. Theory {\bf 19} (2015), 186--210.









\bibitem[BW18a]{BW18a} H. Bao and W. Wang,
{\em  A new approach to Kazhdan-Lusztig theory  of type $B$ via quantum symmetric pairs}, Ast\'erisque {\bf 402}, 2018, vii+134pp,  \href{https://arxiv.org/abs/1310.0103}{arXiv:1310.0103}

\bibitem[BW18b]{BW18b} H. Bao and W. Wang,
{\em Canonical bases arising from quantum symmetric pairs}, Inventiones Math. {\bf 213} (2018), 1099--1177.



\bibitem[BeW18]{BeW18} C. Berman and W. Wang, {\em Formulae of $\imath$-divided powers in ${\mathbf U}_q(\mathfrak{sl}_2)$}, J. Pure Appl. Algebra {\bf 222} (2018), 2667--2702.

\bibitem[Br13]{Br} T. Bridgeland,
{\em Quantum groups via Hall algebras of complexes}, Ann. Math. {\bf 177} (2013), 739--759.


\bibitem[CLW18]{CLW18} X. Chen, M. Lu and W. Wang,
{\em A Serre presentation for the $\imath$quantum gruops}, Transform. Groups (to appear),
\href{https://arxiv.org/abs/1810.12475}{arxiv:1810.12475}

\bibitem[CLW20]{CLW20} X. Chen, M. Lu and W. Wang,
{\em Serre-Lusztig relations for $\imath$quantum groups},
\href{http://arxiv.org/abs/2001.0381}{arxiv:2001.03818}





\bibitem[EJ00]{EJ} E.E. Enochs and O.M.G. Jenda,
Relative homological algebra. de Gruyter Exp. Math. {\bf 30}, Walter de Gruyter Co., 2000.

\bibitem[FL${}^3$W20]{FL+20}  Z. Fan, C.~Lai, Y.~ Li, L.~Luo and W.~Wang,
{\it  Affine flag varieties and quantum symmetric pairs},
Memoirs AMS, vol. {\bf 265} (2020), no. {\bf 1285}, 123pp.



\bibitem[Gor13]{Gor1} M. Gorsky,
{\em Semi-derived Hall algebras and tilting invariance of Bridgeland-Hall algebras}, \href{https://arxiv.org/abs/1303.5879}{arXiv:1303.5879v2}

\bibitem[Gor18]{Gor2} M. Gorsky,
{\em Semi-derived and derived Hall algebras for stable categories}, IMRN, Vol. {\bf 2018}, No . 1,  138--159. \href{https://arxiv.org/abs/1409.6798}{arXiv:1409.6798}

\bibitem[Gr95]{Gr95} J.A. Green,
{\em Hall algebras, hereditary algebras and quantum groups}, Invent. Math. {\bf 120} (1995), 361--377.


\bibitem[Ha91]{Ha3} D. Happel,
{\em On Gorenstein algebras}, In: Progress in Math. {\bf 95}, Birkh\"{a}user Verlag, Basel, 1991, 389-404.


\bibitem[Ke90]{Ke1} B. Keller, {\em Chain complexes and stable categories}, Manus. Math. {\bf67} (1990), 379--417.

\bibitem[Ke94]{Ke2} B. Keller, {\em Deriving DG categories}, Ann. Sci. Ec. Horm. Super. (4){\bf 27}(1) (1994), 63--102.


\bibitem[Ko14]{Ko14} S. Kolb,
{\em Quantum symmetric Kac-Moody pairs}, Adv. Math. {\bf 267} (2014), 395--469.



\bibitem[Let99]{Let99} G. Letzter,
{\em Symmetric pairs for quantized enveloping algebras}, J. Algebra {\bf 220} (1999), 729--767.

\bibitem[Let02]{Let02}
G. Letzter,
{\em Coideal subalgebras and quantum symmetric pairs},
New directions in Hopf algebras (Cambridge), MSRI publications, {\bf 43}, Cambridge Univ. Press, 2002, pp. 117--166.


 \bibitem[Li20]{Li20}  Y.~Li,
 {\em On canonical bases for the Letzter algebra $\U^\imath(\mathfrak{sl}_2)$}, J. Pure Appl. Algebra {\bf 224} (2020), 106227.



\bibitem[Lu19]{Lu19} M. Lu,
{\em Semi-derived Ringel-Hall algebras of 1-Gorenstein algebras},
{\em Appendix A to \cite{LW19a}}.

\bibitem[LP16]{LP} M. Lu and L. Peng,
{\em Semi-derived Ringel-Hall algebras and Drinfeld doubles}, \href{https://arxiv.org/abs/1608.03106}{arXiv:1608.03106v2}

\bibitem[LRW20]{LRW20} M. Lu, S. Ruan and W. Wang,
{\em $\imath$Hall algebra of the projective line and $q$-Onsager algebra}, \href{https://arxiv.org/abs/2010.00646}{arXiv:2010.00646}

\bibitem[LW19a]{LW19a} M. Lu and W. Wang,
{\em Hall algebras and quantum symmetric pairs I: foundations}, \href{http://arxiv.org/abs/1901.11446}{arXiv:1901.11446}

\bibitem[LW19b]{LW19b} M. Lu and W. Wang,
{\em Hall algebras and quantum symmetric pairs II: reflection functors}, Preprint, \href{http://arxiv.org/abs/1904.01621}{arXiv:1904.01621}



\bibitem[LZ17]{LZ} M. Lu and B. Zhu,
{\em Singularity categories of Gorenstein monomial algebras}, \href{https://arxiv.org/abs/1708.00311}{arXiv:1708.00311}

\bibitem[Lus90]{L90} G. Lusztig,
{\em Canonical bases arising from quantized enveloping algebras}, J. Amer. Math. Soc. {\bf 3} (1990),  447--498.

\bibitem[Lus93]{L93} G. Lusztig, Introduction to Quantum Groups, Birkh\"{a}user, Boston, 1993.

\bibitem[M06]{M06} K.E. Morrison,
{\em Integer sequences and matrices over finite fields}, J. Integer Seq. {\bf 9} (2006),
Article 06.2.1.





\bibitem[Rin90]{Rin} C.M. Ringel,
{\em Hall algebras and quantum groups}, Invent. Math. {\bf 101} (1990), 583--591.




\bibitem[Sch06]{Sch06}  O. Schiffmann,
{\em Lectures on Hall algebras}, Geometric methods in representation theory II, 1--141,
S\'emin. Congr., 24-II, Soc. Math. France, Paris, 2012, \href{https://arxiv.org/abs/math/0611617}{arXiv:math/0611617}



\end{thebibliography}
\end{document}